\documentclass{amsart}[12]

\usepackage{geometry}
\usepackage{amsmath}
\usepackage{amsthm}
\usepackage{mathrsfs}
\usepackage{amsfonts}
\usepackage{amssymb}
\usepackage{stmaryrd}
\usepackage{amscd}
\usepackage{array}
\usepackage{amssymb}
\usepackage[all, cmtip]{xy}
\usepackage[dvipsnames]{xcolor}
\usepackage{tikz}
\usetikzlibrary{arrows}
\usetikzlibrary{cd}
\usepackage{enumitem}
\usepackage{blkarray}
\usepackage{hyperref}
\usepackage{mathtools}
\usepackage{tabu}

\newtheorem{theorem}{Theorem}[subsection]
\newtheorem{lemma}[theorem]{Lemma}

\newtheorem{definition}[theorem]{Definition}

\newtheorem{proposition}[theorem]{Proposition}
\theoremstyle{remark}
\newtheorem{remark}[theorem]{Remark}
\newtheorem{example}[theorem]{Example}

\newcommand{\CC}{\mathbb{C}}
\newcommand{\QQ}{\mathbb{Q}}
\newcommand{\RR}{\mathbb{R}}
\newcommand{\PP}{\mathbb{P}}
\newcommand{\ZZ}{\mathbb{Z}}
\newcommand{\TT}{\mathbb{T}}

\newcommand{\KK}{\mathbb{K}}

\renewcommand{\AA}{\mathbb{A}}

\newcommand{\OO}{\mathscr{O}}

\newcommand{\sB}{\mathcal{B}}
\newcommand{\Ieff}{\mathrm{I}\text{-}\mathrm{Eff}}

\renewcommand{\a}{\mathbf{a}}

\newcommand{\n}{\mathfrak{n}}

\newcommand{\wP}[1]{\mathbb{P}^1_{a,1}}

\newcommand{\fX}{\mathfrak{X}}

\newcommand{\cB}{\mathcal{B}}

\renewcommand{\sslash}{\mathord{/\mkern-6mu/}}

\usepackage{etoolbox}
\let\tinymatrix\smallmatrix

\patchcmd{\tinymatrix}{\scriptstyle}{\scriptscriptstyle}{}{}
\patchcmd{\tinymatrix}{\scriptstyle}{\scriptscriptstyle}{}{}
\patchcmd{\tinymatrix}{\vcenter}{\vtop}{}{}
\patchcmd{\tinymatrix}{\bgroup}{\bgroup\scriptsize}{}{}
\newcommand{\bmu}{\pmb{\mu}}
\newcommand{\bi}{\mathfrak{i}}
\newcommand{\sA}{\mathcal{A}}

\newcommand{\Gm}{\mathbb{G}_m}
\newcommand{\diag}{\mathrm{diag}}
\newcommand{\Stab}{\mathrm{Stab}}

\newcommand{\age}{\mathrm{age}}
\newcommand{\cX}{\mathcal{X}}

\newcommand{\cC}{\mathcal{C}}
\newcommand{\cL}{\mathcal{L}}
\renewcommand{\L}{\mathscr{L}}
\newcommand{\cA}{\mathcal{A}}

\newcommand{\fp}{\mathfrak{p}}

\newcommand{\In}[1]{\mathcal{I}(#1)}
\newcommand{\Beta}{B}

\usepackage[backend=biber,
style=alphabetic,
isbn=false,
url=false,
sorting=nyt]{biblatex}
\usepackage{bm}
\newcommand{\cM}{\mathcal{M}}

\newcommand{\BM}{\mathrm{BM}}
\newcommand{\CR}{\mathrm{CR}}
\newcommand{\bX}{\mathbf{X}}
\newcommand{\bG}{\mathbf{G}}
\newcommand{\bp}{\mathbf{p}}
\newcommand{\loc}{\mathrm{loc}}
\newcommand{\un}[1]{\underline{#1}}

\renewcommand{\AA}{\mathbb{A}}
\newcommand{\GG}{\mathbb{G}}

\newcommand{\cY}{\mathcal Y}

\newcommand{\bt}{\mathbf{t}}
\newcommand{\bT}{\mathbf{T}}
\newcommand{\bS}{\mathbf{S}}
\newcommand{\bY}{\mathbf{Y}}
\newcommand{\bx}{\mathbf{x}}

\newcommand{\vir}{\mathrm{vir}}

\newcommand{\bxi}{\pmb{\xi}}
\newcommand{\one}{\mathbf{1}}

\DeclareMathOperator{\Spec}{Spec}

\newcommand{\inv}{\mathrm{inv}}
\newcommand{\br}{\mathbf{r}}

\usepackage{bbm}
\usepackage{dsfont}

\DeclareMathOperator{\Sym}{Sym}
\DeclareMathOperator{\Pic}{Pic}

\DeclareMathOperator{\Hom}{Hom}
\DeclareMathOperator{\Eff}{Eff}

\usepackage[backend=biber,
style=alphabetic,
isbn=false,
url=false,
sorting=nyt]{biblatex}
\usepackage{bm}
\usepackage{xcolor}

\newcommand{\NOTEoff}{\newcommand{\Commentn}[1]{}}
\newcommand{\note}[1]{\Commentn{#1}}

\newcommand\TODOon{\newcommand{\Comment}[1]{\noindent\color{red}{\texttt TODO: }##1\color{black}}}

\newcommand\RACHELon{\newcommand{\Commentr}[1]{\noindent\color{blue}{\texttt Rachel: }##1\color{black}}}

\newcommand{\Rachel}[1]{\Commentr{#1}}

\newcommand\Nawazon{\newcommand{\Commentp}[1]{\noindent\color{green}{\texttt Nawaz: }##1\color{black}}}

\newcommand{\Nawaz}[1]{\Commentp{#1}}

\begin{document}

\TODOon

\NOTEoff

\Nawazon

\RACHELon

\title{Some Applications of Abelianization in Gromov-Witten Theory}

\author{Nawaz Sultani}
\address[N.Sultani]{Institute of Mathematics\\
Academia Sinica\\
Taipei, 10617, Taiwan}
\email{sultani@gate.sinica.edu.tw}

\author{Rachel Webb}
\address[R. Webb]{Department of Mathematics\\
Cornell University\\
Ithaca, NY 14853\\
U.S.A.}
\email{r.webb@cornell.edu}

\date{\today}

\begin{abstract}
Let $G$ be a complex reductive group and let $X$ and $E$ be two linear representations of $G$. Let $Y$ be a complete intersection in $X$ equal to the zero locus of a $G$-equivariant section of the trivial bundle $E \times X \to X$. We explain some general techniques for using quasimap formulas to compute useful $I$-functions of $Y\sslash G$. We work several explicit examples, including a rigorous derivation of the conjectural quantum period in \cite{OP}. 
\end{abstract}


\maketitle

\setcounter{tocdepth}{1}
\tableofcontents

\section{Introduction}



A major challenge in computing the genus-zero Gromov-Witten theory of a Deligne-Mumford stack $\cX$ is computing invariants with several insertions in the twisted sectors of $\cX$. Familiar techniques used to study Gromov-Witten theory of schemes cannot access all of these invariants: for example, the divisor equation only applies to untwisted divisor classes \cite{Tseng10}. 
As of this writing, the genus-zero Gromov-Witten invariants of $\cX$ are only well understood when $\cX$ is a toric stack \cite{CCIT15} or an iterated root stack \cite{TY23}. Theorems in \cite{CCK15, gonzalez} applying to more general targets only compute invariants with at most one twisted insertion.

This paper provides a new technique for computing genus-zero invariants of $\cX$ with multiple twisted insertions when $\cX$ is a GIT quotient of an affine scheme.
Our approach is to adapt the extended stacky fans used in \cite{CCIT15} and \cite{TY23} to a construction that applies to any GIT quotient of affine space, a construction we call an \textit{extended GIT presentation of $\cX$}. Indeed, our treatment of extended presentations may be of independent interest. 
We demonstrate our Gromov-Witten technique in several examples, most notably recovering a conjectural formula of Oneto--Petracci for the quantum period of an orbifold del Pezzo surface.

\subsection{Summary of Results} \label{sec:results}

Let $X$ be a representation of a connected complex reductive group $G$ and let $\theta$ be a character of $G$ such that the stacky GIT quotient
\[
X\sslash_\theta G := [X^{ss}_\theta(G)/G]
\]
is a nonempty Deligne-Mumford stack. We set $\cX = X\sslash_\theta G$. Let $T \subset G$ be a maximal torus and let $W$ be the associated Weyl group. We further assume that every $g \in G$ acting with nonempty fixed locus in $X^{ss}_\theta(G)$ is semi-simple and has connected centrilizer (see Section \ref{sec:I-func}).

\subsubsection{A general theorem}
The vector space $H^\bullet_{\CR}(\cX; \QQ)$ is 
is equal to the singular cohomology of a disconnected space (the inertia stack) with open and closed subspaces called \textit{sectors}, which are indexed by certain conjugacy classes in $G$.
We define the \textit{Weyl-invariant sectors} to be those sectors whose associated conjugacy class contains a Weyl-invariant element of $T$. Our main theorem provides a formula for a series containing the invariants of $\cX$ with arbitrarily many insertions equal to fundamental classes of Weyl-invariant sectors.

To state it, recall the $J$-function $J=J(\bt, Q, z)$ of $\cX$ introduced by Tseng in \cite{Tseng10} (following Givental \cite{Giv01} in the scheme case). The parameter $Q$ lives in the Novikov ring $\Lambda_\cX$ and $z$ is a formal variable. The parameter $\bt = \sum_{k \geq 0} \bt_k z^k$ lives in $H^\bullet_{\CR}(\cX; \QQ)\llbracket z \rrbracket$ (although one often uses that the restriction $J(\bt_0, Q, z)$ determines all of $J$---see e.g. the introduction of \cite{CCIT19}). 
The following theorem computes a specialization of $J(\bt, Q, z)$ and is stated more precisely later as Theorem \ref{thm:only}.

\begin{theorem}[Theorem \ref{thm:only}]\label{thm:intro}

Let $\gamma_1, \ldots, \gamma_n$ be the fundamental classes of the Weyl-invariant sectors of $\cX$. Then there is a $\QQ$-algebra $\Lambda$ (Definition \ref{def:i-eff nov}) with a maximal ideal $\mathfrak{q}$, a homomorphism $\iota: \Lambda_{\cX} \to \Lambda$, and a series
\[ 
\mu \in \Lambda\llbracket t_i, z \rrbracket \otimes_\QQ H^\bullet_{CR}(\cX; \QQ ) \quad \quad \quad \quad \text{satisfying} \quad \quad \quad \quad \partial \mu/\partial t_i \equiv \gamma_i  \mod \mathfrak{q},
\]
such that the corresponding specialization $J(\mu, \iota(Q), z)$ has an explicit formula. Moreover, the homomorphism $\iota$ is very often injective (see Remark \ref{rmk:extend-pic}).
\end{theorem}

We make two remarks about the above statement. The first is that  because of the derivative condition on $\mu$, the specialization $J(\mu, \iota(Q), z)$ in Theorem \ref{thm:intro} contains all the Gromov-Witten invariants of $\cX$ with one arbitrary insertion and the remaining insertions equal to $\gamma_i$'s.  In sufficiently nice settings, one can extract these  invariants from $J(\mu, \iota(Q), z)$ (we do this for example in Theorem \ref{thm:intro2}). The second is that even though Theorem \ref{thm:intro} only applies to $G$-representations $X$, the precise statement in Theorem \ref{thm:only} can profitably be used to study Gromov-Witten invariants of $Y\sslash_\theta G$ for certain $Y$ that are complete intersections in $X$. 

\begin{remark}\label{rmk:intro}
We have stated Theorem \ref{thm:intro} in a way that highlights its implications for Gromov-Witten theory. However, the precise formulation in Theorem \ref{thm:only} uses the quasimap theory of \cite{CCK15}. In particular, the specialization $J(\mu, \iota(Q), z)$ is quasimap $I$-function arising from a specific presentation of $\cX$ as a GIT quotient. This $I$-function is known to be a specialization of the $J$-function by the mirror theorem of Yang Zhou \cite{yang}, and its explicit formula mentioned in the theorem can already be found in \cite{Webb21} (recalled in Theorem \ref{thm:Ifunc-formula} below). The content of Theorem \ref{thm:only} is that there exists a GIT presentation of $\cX$ such that the associated series $\mu$ has the desired derivative property. We discuss our construction of the correct GIT presentation further in Section \ref{sec:intro-git}. 
\end{remark}






\subsubsection{Applications of the theorem}
Using Theorem \ref{thm:intro} (and its implications for complete intersections) we compute restrictions of $J(\bt, Q, z)$ for several examples previously not known in the literature, including a weighted Grassmannian (Section \ref{sec:wgr}), a weighted flag space (Section \ref{sec:wfl}), and the total space of a bundle on a weighted Grassmannian (Section \ref{sec:wbun}). Perhaps of most interest, we recover the conjectural formula for a restriction of the quantum period of the orbifold del Pezzo surface $X_{1, 7/3}$ given in \cite[Sec~6.2]{OP}. In fact, by extending the GIT presentation, we are able to compute the \textit{full} quantum period.

\begin{theorem}[Section \ref{sec:delpez}]\label{thm:intro2}
The full quantum period of $X_{1, 7/3}$ is equal to
\[
 e^{-t(x+5)}\sum_{\tilde \beta_i \in \ZZ_{\geq 0}} A_{\tilde \beta}\left(1 + \frac{\tilde \beta_1-\tilde \beta_2}{2}(-3B_{\tilde \beta_1}+3B_{\tilde \beta_2} - 2B_{2\tilde \beta_1 + \tilde \beta_2 + \tilde \beta_3} + 2B_{\tilde \beta_1 + 2\tilde \beta_2 + \tilde \beta_3} )\right)
\]
where $A_{\tilde \beta}$ is the hypergeometric series in $t$ and $x$ defined in Section \ref{sec:period} and $B_\ell = \sum _{k=1}^\ell 1/k$.
\end{theorem}
The regularization of the quantum period in Theorem \ref{thm:intro2} is conjecturally equal to the classical period of the Laurent polynomial mirror to $X_{1, 7/3}$. Using Theorem \ref{thm:intro2} we check this equality holds for the first few terms.

\subsubsection{Extended GIT presentations}\label{sec:intro-git}

To prove Theorems \ref{thm:intro} and \ref{thm:intro2}, we use the technique of \textit{extended GIT presentations}, which we believe to be of independent interest. This technique leverages the fact that a Deligne-Mumford stack $\cX$ may have many presentations as a GIT quotient $X\sslash_\theta G$, and presentations using higher dimensional $X$ and $G$ can be used to explicitly construct line bundles on $\cX$. The key observation for for this paper is that the quasimap $I$-functions mentioned in Remark \ref{rmk:intro} depend on the GIT triple $(X, G, \theta)$ and not just on $X\sslash_\theta G$.

The nonuniqueness of GIT presentations leads to the question, given one description of $\cX$ as a GIT quotient $X\sslash_\theta G$, how can one systematically construct other GIT presentations of $\cX$? Jiang answered this question for abelian $G$ with his construction of extended stacky fans in \cite{Jiang08}. For nonabelian $G$, we present a recipe for constructing triples $(X \times \AA^1, G \times \Gm, \vartheta)$ whose associated GIT quotients are isomorphic to $X\sslash_\theta G$. The input to the construction is a 1-parameter subgroup of $GL(X)$ commuting with the image of $G$, and a precise result can be found in Proposition \ref{prop:extend-presentation}.


\note{here's how to write down an extension that gets a basis for chen-ruan cohomology. Nawaz says, let $t^J_\alpha$ denote the element of Chen-Ruan cohomology in the twisted sector $\alpha$ given by multiplying the divisor classes corresponding to the columns of the weight matrix corresponding to the index set $J$ (so $J$ is a subset of $1 \ldots n$ where your weight matrix has $n$ columns). The set $J$ should be a subset of the indices corresponding to the twisted sector $\alpha$. Let $\nu$ be the 1-ps that you use to "get" the fundamental class of the $\alpha$ twisted sector. Now subtract 1 from each coordinate of $\nu$ whose index appears in $J$. When you extend by the resulting 1-ps you get $t^J_\alpha z^{-1} 1_\alpha$ times a new novikov variable. It remains to check that the $t^J_\alpha$ generate Chen-Ruan cohomology of $X\sslash T$ (as a $\QQ$-vector space). The key is that if the columns of the weight matrix fail to span the character lattice (over $\QQ$), then you can find a 1-ps that acts trivially on the whole vector space (so there are no semistable points). Call this observation (*). First, note that observation (*) implies that columsn of the weight matrix generate the character lattice, so $t^J_\alpha$ will generate $H^2$ of the untwisted sector. Second, note that for any twisted sector $\alpha$ observation (*) applied to the vector space $X_\alpha$ implies that if $X^{ss}_\alpha$ is not empty then the $t^J_\alpha$ will generate $H^2$ of this twisted sector. To get higher codimension classes (elements of $H^k$ for $k>2$), observe that these can all be written as polynomials in the $x_i$. We have performed an invertible linear change of basis on the $x_i$'s (by choosing a subset of the columns of the weight matrix that form a $\QQ$-basis for the character lattic), so for example all degree 2 monomials in the new basis will still generate all degree-2 monomials in this polynomial ring.}

\subsection{Scope of our methods}
It is interesting to ask, what is the scope of Theorem \ref{thm:intro}? For example, how much of $H^\bullet_{CR}(X\sslash_\theta G; \QQ)$ is supported on Weyl-invariant sectors?

This is a soft question, but we answer with two examples.
The first is a stack with nonabelian isotropy, namely $BS_3$. In this case, our methods recover Gromov-Witten invariants with insertions coming from the identity and the 2-cycle sectors, and we prove that this is all we can do (see Section \ref{sec:bs3}). Hence our methods recover less of the Gromov-Witten theory of this target than other techniques in the literature (e.g. \cite[Prop~3.1]{JK02}, \cite{TT21}).

Second, our weighted Grassmannian example shows that even if a stack has abelian isotropy groups, in general only a portion of its Chen-Ruan cohomology is supported on Weyl-invariant sectors (Remark \ref{rmk:w-inv-sectors}).


\subsection{Summary of the paper}
In Section \ref{sec:background} we recall the precise definition of the series $J^{\cX}(\bt, Q, z)$ that we will use in this paper. This differs slightly from other defintions in the literature, particularly in the case when $\cX$ is noncompact (we do not need a torus action to make our definition). The approaches are compared in Appendix \ref{sec:noncompact}. In the remainder of Section \ref{sec:background} we introduce the series $I^{X, G, \theta}(q, z)$ and recall the relevant formulas for this series from \cite{Webb21}. 

In Sections \ref{sec:git-pres} we introduce our notions of GIT and CI-GIT presentations. Section \ref{sec:computing-I-effective-classes} contains vanishing conditions for summands of $I^{X, G, \theta}(q, z)$, and some related results that facilitate the computation of the mirror map $\mu$. We then apply these results and the extension technique of Section \ref{sec:git-pres} to prove Theorem \ref{thm:intro}.
Finally, Section \ref{sec:examples} contains explicit examples. The appendix compares various approaches to $I$ and $J$-functions, including a comparison of $J$-functions of noncompact targets defined using Borel-Moore homology versus torus localization.

\subsection{Notation and conventions}
\begin{itemize}
\item We work over $\CC$.
\item A group notated $G$ or $T$ is written multiplicatively with identity $1$.
\item If $E \to X$ is a vector bundle on a scheme or algebraic stack $X$ and $s: X \to E$ is a section, then $Z(s)$ is the zero locus of $s$.
\item If $X$ is a scheme, we will use $\Pic^G(X)$ to denote the $G$-equivariant Picard group of $X$. Note that $\Pic^G(X) \simeq \Pic([X/G])$. 
\item If $\theta$ is a character of a group $G$ and $\lambda$ is a 1-parameter subgroup of $G$ then $\langle \theta, \lambda \rangle$ is the integer such that $\theta \circ \lambda$ is equal to $t \mapsto t^{\langle \theta, \lambda \rangle}.$
\end{itemize}

\subsection{Acknowledgements}
The first author thanks Felix Janda for helpful conversations. The second author thanks Yang Zhou for much helpful discussion and Elana Kalashnikov for pointing out the conjectural formula for a quantum period in \cite{OP}. Both authors are grateful to Andrea Petracci for his conversation and to Tom Coates for helpful comments on an early draft of the paper. The second author was partially supported by an NSF Postdoctoral Research Fellowship, award number 200213. 

\section{Background}\label{sec:background}

\subsection{GIT quotient stacks and their inertia}\label{sec:inertia}


Let $G$ be a connected complex algebraic group acting on an affine variety $X$ and let $\chi(G)$ denote the character group of $G$. Choose $\theta \in \chi(G)$ a character of $G$. Let $X^s(G)$ (resp. $X^{ss}(G)$) denote the loci of stable (resp. semi-stable) points of $X$ with respect to $G$ and $\theta$, defined as in \cite[Def~2.1]{king} (note that we omit $\theta$ from the notation). Let $X^{us}(G) = X \setminus X^{ss}(G)$ be the locus of unstable points.

We define the stacky GIT quotient
\[
X\sslash_\theta G := [X^{ss}(G)/G].
\]
We will omit the character $\theta$ when there is no risk of confusion.
Moreover, for any subscheme $Y \subset X$ and subgroup $H \subset G$, we define
\begin{equation}\label{eq:notation}
Y\sslash_G H := [X^{ss}(G)\cap Y/H] \quad \quad \quad \Stab_{H}(Y) := \{g \in H \mid Y^g \neq \emptyset\}.
\end{equation}
Observe that $\Stab_{H}{Y}$ is a union of $H$-conjugacy classes.
\note{It doesn't make sense to say it is a union of $G$-conjugacy classes because in the potential example of interest, $H=T$, the subgroup $H$ is not normal.}
In general, if $Z$ is any variety with a $G$-action and $\xi \in \chi(G)$ is a character, we denote the induced line bundle on $[Z/G]$ by $\cL_{\xi}.$

If $\cX$ is any Deligne-Mumford stack, its inertia stack is $\In{\cX}:= \cX \times_{\Delta, \cX \times \cX, \Delta} \cX$, where $\Delta: \cX \to \cX \times \cX$ is the diagonal morphism. The stack $\In{X\sslash G}$
is a disjoint union of open and closed substacks $\In{X\sslash G}_{(g)}$ indexed by congugacy classes $(g)$ in $\Stab_G(X^s(G))$. The substacks $\In{X\sslash G}_{(g)}$ are called the \textit{sectors} of $\In{X\sslash G},$ and the sector indexed by the identity element of $G$ is called the \textit{untwisted sector}. We use $\one\in H^\bullet(\In{X\sslash G}; \QQ)$ to denote (the Poincar\'e dual of) its fundamental class. The remaining sectors (not indexed by the identity) are called the \textit{twisted sectors.}
If $T \subset G$ is a maximal torus, then the Weyl group $W$ of $T$ in $G$ acts on $\In{X\sslash T}$, and this action sends $\In{X\sslash T}_t$ to $\In{X\sslash T}_{w\cdot t}$.


\begin{example}\label{ex:inertia}
Let $X$ be a vector space and suppose $G=T$ is abelian. Let $\xi_1, \ldots, \xi_n$ be the weights of $X$. We have
\[
\In{X\sslash T}_{g} = [(X^g \cap X^{ss}(T)) / T]
\]
where $X^g$ is the subspace fixed by $g$; i.e., $X^g$ is the subspace of $X$ spanned by eigenvectors for the weights $\xi_i$ where $\xi_i(g) =1.$ 
\end{example}




\subsection{Gromov-Witten invariants and the mirror map}

In this section we define the series $J^{X\sslash G}(\bt(z), Q, z)$ and $I^{X, G, \theta}(q, z)$ (nonequivariantly, and for noncompact $X\sslash G$). 
The key feature of the quasimap theory defined in \cite{CCK15} is that these series are defined in similar ways, using closely related moduli spaces, and the precise difference can be computed by a wall crossing formula. We will not explain this theory here, but the resulting relationship between $I$ and $J$ is recalled in \eqref{eq:mirror-thm} below. 


\subsubsection{$J$-functions}\label{sec:jfuncs}
Let $G$ be a connected complex algebraic group acting on an affine variety $X$ and let $\theta \in \chi(G)$. For the rest of this paper we assume $X^s(G) = X^{ss}(G)$ and this locus is smooth and nonempty.\note{private heuristic for why we are not missing out on much to assume $G$ is connected: If $G$ is not connected let $G_0$ be the identity component and $H = G/G_0$. There is a fiber square
so a principal $G_0$ bundle is equivalent to a principal $G$-bundle such that the induced principal $H$-bundle is trivial. But I claim that on $\wP{a}$, every principal $G$ bundle induces a trivial $H$-bundle. This is because the map $\wP{a} \to BH$ must factor through the structure map $\wP{a} \to B\mu_a $ (I think . . . but why?), and the map $\wP{a} \to B\mu_a $ is equivalent to a homomorphism $\mu_a \to \Pic(\wP{a}),$ but this picard group is $\ZZ$ so the homomorphism is trivial. SO giving a principal $G$-bundle is equivalent to giving a principal $G_0$ bundle, and I expect that maps from $\wP{a}$ to $[X/G]$ are equivalent to maps from $\wP{a}$ to $[X/G_0]$.}
Then the stack $\cX := X\sslash G$ is a smooth Deligne-Mumford stack with a coarse moduli space $\underline{\cX}$ that is projective over the affine scheme $X\sslash_0 G:= \Spec(\Gamma(X, \OO_X)^G)$.
When $\cX$ is proper, the $J$-function of $\cX$ is defined in \cite{Tseng10}. Our $J$-function is a generalization of Tseng's following a recipe of Zhou \cite{zhou-email}.

We fix notation and conventions. The \textit{Chen-Ruan cohomology} of $\cX$ is $H^\bullet_{\CR}(\cX; \QQ)$. A \textit{stable map to} $\cX$ is a twisted curve $\cC$ and morphism $f: \cC \to \cX$, together with sections of the marked points of $\cC$, 
such that $f$ is stable in the sense of \cite[Sec~4.3]{AGV08}. If $f: \cC \to \cX$ is a stable map, its \textit{class} is the element $B \in H_2(\cX; \QQ):= H_2(\underline{\cX}; \QQ)$ equal to the image of the induced map of coarse spaces. The moduli stack of stable maps to $\cX$ of class $B$ and genus-0 source curve is $\overline{\cM}_{0, n}(X, B)$ (see \cite[Sec~6.1.3]{AGV08}). 
We will write $\overline{\cM}:= \overline{\cM}_{0, n}(X, B)$ when there is no risk of confusion, and we will use $ev_i: \overline{\cM} \to \In{\cX}$ for the evaluation maps.

Let $\Eff(\cX) \subset H_2(\cX; \QQ)$ be the monoid generated by classes of stable maps to $\cX$. If $B \in \Eff(\cX)$ and $\cL \in \Pic(\cX)$ we define $B(\cL)\in \QQ$ to be the following natural pairing: if $n$ is an integer such that $\cL^n$ descends to a line bundle $\underline{\cL'}$ on $\underline{\cX}$, then $B(\cL) =(1/n)\deg_{\underline{\cX}}(\underline{\cL'})$. 
Let $e \in \ZZ$ have the property that $e\Beta(\cL_\theta) \in \ZZ$ for all $\Beta \in \Eff(\cX)$ (this exists by \cite[~Prop 2.1.1]{AGV08}, using \cite[Lem~2.1.2]{Webb21} in place of properness of $\cX$). Define a valuation $\nu: \QQ[\Eff(\cX)] \to \ZZ$ by 
\[
\nu\left( \sum_{\Beta \in \Eff(\cX)} a_\Beta Q^\Beta\right) := \min_{a_\Beta \neq 0} e\Beta(\cL_\theta).
\]
The \textit{Novikov ring} $\Lambda_\cX$ is the completion of the group ring $\QQ[\Eff(\cX)]$ with respect to the ideals $F_\Lambda^1 \supset F_\Lambda ^2 \supset . . .$ where $F_\Lambda^i := \{\alpha \in \QQ[\Eff(\cX)] \mid \nu(\alpha) \geq i\}.$
\note{Yang uses $\QQ$ coefficients.}
Also define the ring of convergent Laurent series
\[\Lambda_{\cX}\{z, z^{-1}\}:= \left\{ \sum_{n \in \ZZ} r_nz^n\;\; \bigg \vert \;\;r_n \in \Lambda,\;\; \nu(r_n) \to \infty \;\text{as}\; |n| \to \infty\right\}.\]
 (see \cite[Sec~3]{CCIT09}). 
 \note{The valuation is nonnegative because the pullback of an ample line bundle is nef.}
We write an element of $\Lambda_\cX$ as an infinite $\QQ$-linear combination of variables $Q^\Beta$ where $\Beta \in \Eff(\cX)$. Hence, elements of $\Lambda_{\cX}\{z, z^{-1}\}$ are formal power series in $z, z^{-1}$ with coefficients in $\Lambda_{\cX}$, such that the coefficient of $Q^\Beta$ is in $\Lambda[z, z^{-1}]$.

The discussion of (co)homology for proper Deligne-Mumford stacks in \cite[Sec~2.2]{AGV08} can be extended to noncompact stacks using Borel-Moore homology. In particular, we define the Borel-Moore homology $H^{\BM}_\bullet(\cX; \QQ)$ to be the Borel-Moore homology of the coarse space $\underline{\cX}$. We use two properties of Borel-Moore homology: it is covariant for proper morphisms, and when $\cX$ is smooth there is a Poincar\'e duality isomorphism
\[
H^i(\underline{\cX}; \QQ) \xrightarrow{\sim} H^{\BM}_{2\dim(\underline{\cX})-i}(\underline{\cX}; \QQ)
\]
(for a proof of duality, see e.g. \cite[Sec~8.2]{HTT}).\note{There are many definitions of Borel-Moore homology. The definition used in [cite Fulton ch 19] agrees with the one used in [cite Yang's ref] by [cite Iversen].} We lift this duality to $\cX$ as in \cite[Sec~2.2]{AGV08}. 

The $J$-function is defined as follows. For any set of formal variables $\{t_i\}_{i=1}^N$, and for any $\bt(z) \in H^\bullet_{\CR}(\cX; \QQ)\otimes \Lambda_{\cX} \llbracket t_i, z\rrbracket$ we set
\begin{align}\label{eq:jfunc-def}
J^{\cX}(\bt(z), Q, z) = \one &+ \frac{\bt(-z)}{z} \\
&+ \sum_{\substack{n \geq 0\\B \in \Eff(\cX)}}\frac{Q^B}{n!}  \inv^* ev_{1*}\left[\left( \frac{1}{z(z-\psi_1)} \cup 
\prod_{i=2}^{n+1} ev_i^*(\bt(\psi_i))
\right) \cap [\overline{\cM}]^w\right]\notag
\end{align}
where $\inv: \In{\cX} \to \In{\cX}$ is the involution
and $[\overline{\cM}]^w$ is the \textit{weighted} virtual class defined in \cite[Sec~4.6]{AGV02} (see also \cite[Sec~2.5.1]{Tseng10}).
The stack $\overline{\cM}$ is proper over $X\sslash_0 G$, as is $\In{\cX}$, and hence the evaluation morphisms are proper. We interpret \eqref{eq:jfunc-def} by applying the cycle map to $[\overline{\cM}]^w$ so that it is an element of $H^{\BM}_\bullet(\overline{\cM}; \QQ)$, capping with the class shown to obtain another element of Borel-Moore homology, applying the proper pushforward $ev_{1*}$, and finally applying Poincar\'e duality on $\In{\cX}$ to obtain an element of $H^\bullet(\In{\cX}; \QQ)$, to which we can apply $\inv^*$.

\begin{remark}\label{rmk:different-J}
Our series \eqref{eq:jfunc-def} differs from the series $J^{\cX}(\bt(z), Q, z)$ used in \cite{yang} by a locally constant factor and in the choice of virtual class. Our series differs from the one in \cite[Def~3.1.2]{Tseng10} by a sign, a factor of $z$, and our handling of noncompact $\cX$. These differences are explained in detail in Appendix \ref{sec:compare}. 
\end{remark}

Observe that $J^\cX(\bt(z), Q, z)$ is an element of $H^\bullet_{\CR}(\cX; \QQ) \otimes \Lambda_\cX\llbracket t_i \rrbracket\{z, z^{-1}\}$, and it depends only on the stack $\cX$ and its polarization $\cL_\theta$ (the choice of $\cL_\theta$ affects the Novikov ring). 



\subsubsection{Quasimap $I$-functions}\label{sec:I-func}
Let $X$, $G$, and $\theta$ be as in Section \ref{sec:jfuncs} and assume that $X$ has at worst l.c.i. singularities.\note{need lci because i say qmap theory is defined, later}  We further assume that every $g \in G$ acting with nonempty fixed locus in $X^{ss}_\theta(G)$ is semi-simple and has connected centrilizer---this is so we can apply the results of \cite{Webb21}; see \cite[Sec 1.1.1]{Webb21}.
The $I$-function $I^{X, G, \theta}(q, z)$ depends on the triple $(X, G, \theta)$, not just on $\cX = X\sslash G$, so we call $(X, G, \theta)$ a \textit{quasimap target}. 

We fix notation and conventions. A $k$-pointed, genus-$g$ quasimap to $(X, G, \theta)$ is a representable morphism $q: \cC \to [X/G]$ where $\cC$ is a twisted curve with genus $g$ and $k$ markings, such that $q^{-1}([X^{us}/G])$ is zero-dimensional and disjoint from nodes and markings.\footnote{
Later, in Section \ref{sec:formula} we will define a \textit{quasimap} (with no adjectives). This paper will primarily use quasimaps, rather than $k$-pointed genus-$g$ quasimaps.}\note{our definition follows Yang, not CCK, by requiring base locus disjoint from nodes and markings, so that thm about effective classes later is true} The \textit{class} of $q$ is the homomorphism $\beta: \Pic([X/G]) \to \QQ$ given by $\beta(\cL) = \deg_{\cC}(\cL)$. We denote by $\Eff(X, G, \theta) \in \Hom(\Pic^G(X), \QQ)$ the monoid generated by classes of $k$-pointed, genus-$g$ quasimaps. The Novikov ring $\Lambda_{X, G, \theta}$ is the completion of $\QQ[\Eff(X, G, \theta)]$ defined in analogy with $\Lambda_\cX$, replacing $\QQ[\Eff(\cX)]$ with $\QQ[\Eff(X, G, \theta)]$.
We write an element of $\Lambda_{X, G, \theta}$ as an infinite $\QQ$-linear combination of variables $q^\beta$ for $\beta \in \Hom(\Pic^G(X), \QQ)$. 

The $I$-function is a series
\begin{equation}\label{eq:Ifunc-shape}
I^{X, G, \theta}(q, z) = \one + \sum_{\substack{\beta \in \Eff(X, G, \theta)\\ \beta \neq 0}} q^\beta I^{X, G, \theta}_\beta(z) \in H^\bullet_{\CR}(\cX; \QQ)\otimes \Lambda_{X, G, \theta}\{z, z^{-1}\}.
\end{equation}
We define the coefficients $I^{X, G, \theta}_\beta(z) \in H^\bullet_{\CR}(\cX; \QQ)\otimes \QQ[z, z^{-1}]$ as in \cite[Def~3.4.4]{Webb21}. (This is the definition in \cite{CCK}, modulo some minor conventions).
\note{Let's sort out all the different kinds of $I$-functions. There are at least four: $I^{CCK}, I^{Yang}, I^{Tseng},$ and $I^{Rachel}$. These are related by ($\varpi$ is the rigidification of $\In{X\sslash G}$
\begin{itemize}
\item $\varpi^*I^{CCK} = I^{Tseng}$ (p.23 of CCK)
\item $\varpi_*I^{Yang} = I^{CCK}$ (sec 1.6 of Yang)
\item $\mathbf{a} \varpi_* I^{Rachel} = I^{CCK}$ (Remark 3.4.5 of nonab2)
\end{itemize}
It follows from 2 and 3 that $I^{Yang} = \mathbf{a} I^{Rachel}$. Since $\varpi^*\varpi_* = 1/\mathbf{a}$ we see that $I^{Rachel} = I^{Tseng}$. This discussion ignores differences in multilplying by $z$ and in using $z$ vs $-z$.}
We give an explicit formula for $I^{X, G, \theta}_\beta(z)$ in Theorem \ref{thm:Ifunc-formula} below; when $G$ is a torus it is the expected hypergeometric series. 
\begin{remark}\label{rmk:different-I}
Our series $I^{X, G, \theta}(q, z)$ differs from the one defined in \cite[Def~1.10.1]{yang} by a locally constant factor. This is discussed in Appendix \ref{sec:compare}.
\end{remark}

\subsubsection{The mirror theorem}\label{sec:mirror-thm}
Let $X$, $G$, and $\theta$ be as in Section \ref{sec:jfuncs} and assume that $X$ has at worst l.c.i. singularities. Set $\cX  = X\sslash G$.
There is a monoid homomorphism $\iota: \Eff(\cX) \to \Eff(X, G, \theta)$ induced by 
the composition
\begin{equation}\label{eq:def-iota}
H_2^{\mathrm{alg}}(\cX; \QQ) \xrightarrow{\iota_1} \Hom(\Pic(\cX), \QQ) \xrightarrow{\iota_2} \Hom(\Pic([X/G]),\QQ),
\end{equation}
where $H^{\mathrm{alg}}_2(\cX; \QQ)$ is the subspace of $H_2(\cX; \QQ)$ spanned by classes of algebraic curves.
The homomorphism $\iota$ induces a ring homomorphism $\Lambda_\cX \to \Lambda_{X, G, \theta}$,
hence it makes sense to compare the $I$-function and the $J$-function of $\cX$ under $\iota$. 

To state this comparison, we define the mirror map $\mu(z) \in H^\bullet_{\CR}(\cX; \QQ)[z]$ as
\[
\mu(q, z) :=[zI(q, z)-z]_+
\]
 where $[\cdot]_+$ is the truncation 
 to nonnegative powers of $z$. The elegant mirror theorem of \cite[Thm~1.12.2]{yang} states
\begin{equation}\label{eq:mirror-thm}
J^{\cX}(\mu(q, -z), \iota(Q), z) = I^{X, G, \theta}(q, z).
\end{equation}



\begin{remark}
As noted in Remarks \ref{rmk:different-J} and \ref{rmk:different-I}, our definitions of the series $J^\cX$ and $I^{X, G, \theta}$ are different from those used in \cite{yang}. In Section \ref{sec:pf-of-mirror} we derive \eqref{eq:mirror-thm} from \cite[Thm~1.12.2]{yang}.
\end{remark}
\begin{remark}\label{rmk:extend-pic}


To use \eqref{eq:mirror-thm} to study $J^{\cX}(\bt, Q, z)$, one would like $\iota$ to be injective. The morphism $\iota_1$ is injective if $X\sslash G$ is a vector bundle over a stack $\cB$ with projective coarse moduli. To see this, suppose $C_1$ and $C_2$ are numerically equivalent cycles on $X\sslash G$.\note{a priori we're assuming numerical equivalence on $\cX$, meaning they define the same hom on $\Pic(\cX)$, but this means they define the same hom on the coarse group $\Pic(\underline{\cX})$.} Then their pushforwards on $\cB$ are also numerically equivalent.\note{ numerical equivalence is defined by intersections w/cartier divisors, aka line bundles, and these pull back.} Since $\cB$ has projective coarse moduli, by \cite[Prop~14]{kollar} the images of $[C_1]$ and $[C_2]$ in $H_2(\cB; \QQ)$ are equal. But pushforward defines an isomorphism $H_2(X\sslash G; \QQ) \to H_2(\cB, \QQ)$, so $[C_1]=[C_2]$ in $H_2(X\sslash G; \QQ).$

The morphism $\iota_2$ is injective if $\Pic([X/G]) \to \Pic(\cX)$ is surjective. Since $\Pic([X/G]) = \Pic^G(X)$ and $\Pic(\cX) = \Pic^G(X^{ss}(G))$, the question is about extending $G$-equivariant line bundles from $X^{ss}(G)$ to $X$. A sufficient condition is 
\[ 
X \;\text{is locally factorial and}\; X \setminus X^{ss}(G)\;\text{ has codimension at least 2}
\]
(one may extend first the line bundle and then the $G$-equivariant data by Hartog's).\note{for extending line bundles see \url{https://mathoverflow.net/questions/315752/extension-of-line-bundle-defined-over-an-open-subscheme} for extending G-equivariance see the last comment here \url{https://math.stackexchange.com/questions/1969391/equivariant-hartogs-extending-g-equivariant-vector-bundle-to-codimension-2}. note you are extending a morphism of line bundles $a^*L \to f^*L$, whcih is the same as extending a section of $(a^*L)^\vee \otimes p^*L$, and Hartog's lemma lets you extend sections of vector bundles. Here $a, p: G \times X \to X$ are the action and projection maps.} This condition is not necessary: see Remark \ref{rmk:embed}
\note{However, if \eqref{eq:codim-2} holds, then one can show that the extension of $G$-equivariant line bundles is unique,\note{the line bundle extends uniquely by \url{https://mathoverflow.net/questions/22111/extending-vector-bundles-on-a-given-open-subscheme?rq=1}, since $X$ being locally factorial implies normal implies S2 by \url{https://en.wikipedia.org/wiki/Serre\%27s_criterion_for_normality}. The equivariant data is equivalent to a section of a certain line bundle (see other comment), which is like a map to a separated scheme $\AA^1$ and $X$ is a variety hence reduced, so this section extends uniquely.} so $\iota$ is an isomorphism in this case.}
\end{remark}

\subsection{$I$-function formula}\label{sec:formula}
Let $X$, $G$, and $\theta$ be as in Section \ref{sec:jfuncs}, and assume moreover that $X$ has at worst l.c.i. singularities.

Let $\AA^{2}_{u, v} := \Spec(\CC[u, v])$. 
For $a$ a positive integer, let $\wP{a}$ be the weighted projective space given by $\AA^2_{u, v}\setminus\{0\}$ modulo $\Gm$, where $\Gm$ acts with weights $(a, 1)$ (note that $v=0$ is the unique stacky point).
A \emph{quasimap} to $(X, G, \theta)$ is a representable morphism $q: \wP{a} \to [X/G]$ such that the substack $q^{-1}([X^{us}/G])$ is disjoint from the stacky point $v=0$ in $\wP{a}$.

\begin{remark}
Every quasimap determines a $1$-pointed genus-0 quasimap (as defined in Section \ref{sec:I-func}), and the class $\beta$ of a quasimap is an element of $\Eff(X, G, \theta)$. 
\end{remark}
 
 Let $T\subset G$ be a maximal torus. Note that $\theta$ restricts to a character of $T$, which we also denote $\theta$.
Finally, let $E$ be a $G$-representation and let $s$ be a $G$-equivariant section of the vector bundle $E \times X \to X$ with $Y \subset X$ the zero locus of $s$. We allow the situation $E = \Spec(\CC)$, in which case $Y=X$. 

\subsubsection{Geometry}
We have the following commuting diagrams relating $Y\sslash G$ and $X\sslash T$ (resp. $[Y/G]$ and $[X/T]$):
\[
\begin{tikzcd}
Y\sslash_G T\arrow[r, hook] \arrow[d, "\varphi"]& X\sslash_G T \arrow[r, hook, "j"]\arrow[d] & X\sslash T &&{[Y/T]} \arrow[r] \arrow[d] & {[X/T]} \arrow[d] \\
Y\sslash G \arrow[r, hook, "\bi"] & X\sslash G &&& {[Y/G]} \arrow[r] & {[X/G]}
\end{tikzcd}
\]
In the left diagram, $j$ is an open immersion, $\bi$ is a closed immersion, the square is fibered, and $\varphi$ is flat (in fact, a $G/T$ fiber bundle). 
An analogous diagram holds for the inertia stacks of these spaces. The right diagram induces the following commuting squre.
\[
\begin{tikzcd}
{\Hom(\Pic^T(Y), \QQ)} \arrow[r, hook, "r^{Y, T}_{X, T}"] \arrow[d, "r^{Y, T}_{Y, G}"]& {\Hom(\Pic^T(X), \QQ)} \arrow[d, "r^{X, T}_{X, G}"] \\
\Hom(\Pic^G(Y), \QQ) \arrow[r, hook, "r^{Y, G}_{X, G}"] & \Hom(\Pic^G(X), \QQ)
\end{tikzcd}
\]

\subsubsection{Notation}

Observe that $\Pic^T(X)$ contains the character group $\chi(T)$ via the assignment $\xi \mapsto \cL_{\xi}$. Let $\tilde \beta \in \Hom(\Pic^T(X), \QQ)$. For notational ease, we will write $\tilde \beta(\xi)$ for $\tilde \beta(\cL_{\xi})$. 
There is a group homomorphism
$\Hom(\Pic^T(X), \QQ) \to T$ sending
$\tilde \beta $ to $g_{\tilde \beta}$,
where $g_{\tilde \beta}$ is defined by the property that for $\xi \in \chi(T)$,
\begin{equation}\label{eq:def-g}
\xi(g_{\tilde \beta}) = e^{2\pi i \tilde \beta(\xi)}.
\end{equation}

Let $\xi_1, \ldots, \xi_n$ (resp. $\epsilon_1, \ldots, \epsilon_r$) denote the weights of $X$ (resp. $E$) with respect to $T$. For $\tilde \beta \in \Hom(\chi(T), \QQ)$ we call $\tilde \beta$ \textit{I-nonnegative} if the set 
\[
\{ \;\epsilon \in \{\epsilon_j\}_{j=1}^r \;\mid \;\tilde \beta(\epsilon) \in \ZZ_{<0} \;\}
\]
is empty. We define $X^{\tilde \beta}$ to be the subspace
\begin{equation}\label{def:x-tilde-beta-intro}X^{\tilde \beta}=\{(x_1, \ldots, x_n) \in X \mid x_i = 0 \;\text{for all}\; i \;\text{such that}\; \tilde \beta(\xi_i) \not \in \ZZ_{\geq 0}\}.
\end{equation} 
Observe that $X^{\tilde \beta}$ is contained in the fixed locus of $g_{\tilde \beta}^{-1}$, so we can use \eqref{eq:def-g} and Example \ref{ex:inertia} to define substacks
\[
F^0_{\tilde \beta}(X\sslash T) := X^{\tilde \beta} \sslash_G T \subset \In{X\sslash_G T}_{g_{\tilde \beta}^{-1}} \quad \quad \quad F^0_{\tilde \beta}(Y\sslash T) := (Y \cap X^{\tilde \beta}) \sslash_G T \subset \In{Y\sslash_G T}_{g_{\tilde \beta}^{-1}}.
\]
For $\xi \in \chi(T)$ and $\tilde \beta \in \Hom(\chi(T), \QQ)$ we define operational Chow classes on $Y\sslash_G T$ by
\begin{align*}
C^\circ(\tilde\beta, \xi) 
&:= \begin{cases}
\prod_{\tilde \beta(\xi) < k < 0,  \;k-\tilde\beta(\xi) \in \ZZ}^{}(c_1(\L_{\xi}) + kz) & \tilde\beta(\xi) \leq 0\\
\left[\prod_{0 < k \leq \tilde\beta(\xi),  \;k-\tilde\beta(\xi) \in \ZZ}^{}(c_1(\L_{\xi}) + kz)\right]^{-1} & \tilde\beta(\xi) >0.
\end{cases}\\
C(\tilde\beta, \xi) 
&:= \begin{cases}
c_1(\L_{\xi})C^\circ(\tilde\beta, \xi) & \tilde\beta(\xi) \in \ZZ_{< 0}\\
C^\circ(\tilde\beta, \xi) & else.
\end{cases}\\
\end{align*}


\subsubsection{Formula}\label{sec:I1} 

We now restrict to the situation where $X = \AA^n$ for some $n$. 
If $E$ has rank $r$, then a section $s$ of the bundle $E \times X \to X$ determines (after a choice of basis) a sequence of $r$ elements of $\Gamma(X, \OO_X)$. We say $s$ is \textit{regular} if this sequence is regular.
Our $I$-function formula requires the following assumptions.
\begin{enumerate}
\item $G$ is connected
\item $X^s_\theta(G)=X^{ss}_\theta(G)$ and $X^s_\theta(T)=X^{ss}_\theta(T)$, and these are both nonempty
\item If $g \in \Stab_{G}(X^s(G))$ then $g$ is semi-simple (meaning it is contained in a maximal torus), and 
the centrilizer $Z_G(g)$ is connected.
\item The section $s$ is regular\note{Observe that this implies $Y$ has at worst l.c.i. singularities} and $X^s(G) \cap Y$ is smooth.
\end{enumerate}
Observe that if $G$ is abelian, then $G=T$ (by assumption (1)) and condition (3) is automatically satisfied.

\begin{theorem}[Formula for GIT $I$-functions \cite{Webb21}]\label{thm:Ifunc-formula}
Let $X=\CC^n, G, \theta, E, s$, and $Y$ be as above, and assume that these satisfy assumptions (1)-(4). Let $\xi_1, \ldots, \xi_n$ (resp. $\epsilon_1, \ldots, \epsilon_r$) denote the weights of $X$ (resp. $E$) with respect to $T$, and let $\rho_1, \ldots, \rho_m$ be the roots of $G$. If $\beta \in \Hom(\chi(G), \QQ)$ then there are $I_{\tilde\beta}(z)$ such that 
\begin{equation}\label{eq:Ifunc-formula}
\sum_{\delta \mapsto \beta} \varphi^* I^{Y, G, \theta}_\delta(z) = \sum_{\tilde \beta \mapsto \beta} I_{\tilde \beta}(z) \quad \text{where} \quad I_{\tilde \beta}(z) \in H^\bullet(\In{Y\sslash_G T}_{g_{\tilde \beta}^{-1}}; \QQ)\otimes \QQ[z, z^{-1}]
\end{equation}
where the first sum is over $\delta \in (r^{Y, G}_{X, G})^{-1}(\beta)$ and the second sum is over $\tilde \beta \in (r^{X, T}_{X, G})^{-1}(\beta)$. We have a formula for $I_{\tilde \beta}(z)$ in the following cases:
\begin{itemize}
\item If $\In{Y\sslash_G T}_{g_{\tilde \beta}^{-1}}$ is empty then $I_{\tilde \beta}(z)=0$.
More generally, the series $I_{\tilde \beta}(z)$ is supported on $F^0_{\tilde \beta}(Y\sslash T) \subset \In{Y\sslash_G T}_{g_{\tilde \beta}^{-1}}$. 
\item If $\tilde \beta \in \Hom(\chi(T), \QQ)$ is $I$-nonnegative, we have
\begin{equation}\label{eq:Ifunc-coeff-convex}
I_{\tilde \beta}(z) = \left( \prod_{i=1}^m C(\tilde \beta, \rho_i)^{-1}\right)
\left(\prod_{j=1}^r C(\tilde \beta, \epsilon_j)^{-1}\right)
\left(\prod_{\ell=1}^{n}C(\tilde \beta, \xi_\ell)\right)
\one_{g_{\tilde \beta}^{-1}} 
\end{equation}
where $\one_{g_{\tilde \beta}^{-1}}$ is the fundamental class of the sector $\In{Y\sslash_G T}_{g_{\tilde \beta}^{-1}}$.
\item If the inclusion $F^0_{\tilde \beta}(Y\sslash T) \hookrightarrow F^0_{\tilde \beta}(X\sslash T)$ is l.c.i. of codimension 
\[\# \{\epsilon_j \mid \tilde \beta(\epsilon_j) \in \ZZ_{\geq 0} \},\] 
then \eqref{eq:Ifunc-coeff-convex} holds after replacing $C(\tilde \beta, \epsilon_j)^{-1}$ by $C^\circ(\tilde \beta, \epsilon_j)^{-1}$, replacing $C(\tilde \beta, \xi_j)^{-1}$ by $C^\circ(\tilde \beta, \xi_j)^{-1}$, and replacing $\one_{g_{\tilde \beta}^{-1}}$ by $[F^0_{\tilde \beta}(Y\sslash T)]$.
\end{itemize}
\end{theorem}


We note that if $E$ has rank zero then the $I$-nonnegativity condition is always satisfied. 

The formulas in Theorem \ref{thm:Ifunc-formula} are still not completely explicit in the sense that many of the coefficients $I^{Y, G, \theta}_\delta(z)$ (or the summands $I_{\tilde \beta}(z)$) will vanish for reasons not apparent from these formulas. We will return to this problem in Section \ref{sec:computing-I-effective-classes}.

\section{GIT presentations}\label{sec:git-pres}
As explained in \cite{Jiang08}, a toric Deligne-Mumford stack may be described non-uniquely by an {extended stacky fan}.
The mirror theorem for toric stacks in \cite{CCIT15} gives a formula for an $I$-function determined by the extended stacky fan, rather than the underlying stack. The extra data provided by a larger stacky fan can give more information about the Gromov-Witten invariants of the underlying stack (see \cite{CCIT19}).
Under mild assumptions there is a bijection between extended stacky fans and GIT quotients $X\sslash_\theta T$ \cite[Sec~4.2]{CIJ18}. Under this correspondence, extending the stacky fan corresponds to changing the GIT data (see \cite[Rem~9]{CCIT19}). We now expand this story to include $X\sslash_\theta G$ for nonabelian $G$.


\begin{definition}\label{def:GIT}A \emph{GIT presentation} is a tuple $(X, G, \theta)$ where $X$ is equal to $\AA^n_{\CC}$ for some $n$, $G$ is a connected complex reductive algebraic group acting linearly on $X$, 
and $\theta$ is a character of $G$, such that the assumptions (1)--(3) in Section \ref{sec:I1} are satisfied for some choice of maximal torus. 
 If $\cX$ is a separated and finite type Deligne-Mumford stack, we say $(X, G, \theta)$ is a \emph{GIT presentation for $\cX$} if it is a GIT presentation and $\cX \simeq X\sslash_\theta G$.
\end{definition}
The definition of a GIT presentation $(X, G, \theta)$ for $\cX$ is chosen so that Theorem \ref{thm:Ifunc-formula} yields a formula for the small quasimap I-function of $\cX = X \sslash G$. 
\begin{remark}
If $(X, G, \theta)$ is a GIT presentation, then the assumptions (1)--(3) in Section \ref{sec:I1} are in fact satisfied for \textit{every} choice of maximal torus $T \subset G$. To justify this one must check that if $T, T' \subset G$ are two maximal tori then $X^s(T) = X^{ss}(T)$ implies $X^{s}(T') = X^{ss}(T')$. This is a straightforward computation using the numerical criterion in \cite[Prop~2.5]{king} and the fact that $T = gT'g^{-1}$ for some $g \in G$.
\end{remark}



\subsection{Extended GIT presentations}
Suppose $(X, G, \theta)$ is a GIT presentation for $\cX = X\sslash_\theta G$. The following is a way to build other GIT presentations of $\cX$.

\begin{proposition}\label{prop:extend-presentation}
Let $(X, G, \theta)$ be a GIT presentation for $X\sslash_\theta G$ and let $\nu$ be a 1-parameter subgroup of $GL(X)$ that commutes with the image of $G$ in $GL(X)$. Assume moreover that there is an integer $r \geq 1$ and a 1-parameter subgroup $\nu': \Gm \to G$ such that
\begin{equation}\label{eq:limit}
\lim_{t \rightarrow 0} \nu'(t^s)\nu(t^{-rs}) \cdot x \;\;\;\;\text{exists}
\end{equation}
for every $x\in X$ and $s\geq 0$. Then there is a GIT presentation $(X\times \AA^1, G \times \Gm, \vartheta)$ of $X\sslash_\theta G$, where $G \times \Gm$ acts on $X \times \AA^1$ by 
\begin{equation}\label{eq:extended-action}
(g, \gamma) \cdot (x, y) = (g\nu(\gamma)x, \gamma y).\end{equation}
for $(g, \gamma) \in G \times \Gm$ and $(x, y) \in X\times \AA^1$, and $\vartheta$ is equal to the character
\begin{equation}\label{eq:def-vartheta}\vartheta_N: G\times \Gm \rightarrow \Gm \quad \quad \quad \quad \vartheta_N(g, \gamma) = \theta(g) \gamma^N.
\end{equation}
for any $N > \langle \theta, \nu' \rangle/r$.
\end{proposition}

\begin{remark}\label{rmk:nu}
Since $G$ is reductive, $X$ can be written as a direct sum of irreducible representations $X_1, \ldots, X_k$. Requiring $\nu(t)$ to commute with the image of $G$ in $GL(X)$ is equivalent to requiring each $X_i$ to be contained in a weight space of $\nu$.\note{proof:If each $X_i$ is contained in a weight space it is easy to see that $G$ commutes with $\Gm$ acting on $X$. Conversely, if the actions commute, then $\nu(t)$ is an automorphism of the $G$-representation $X$, so $\nu(t)$ permutes the $X_i$ for each $t$. Since $\nu(t)$ is a continuous family and the $X_i$ are discrete, $\nu(t)$ must send each $X_i$ to itself. By Schur's lemma the action of $\nu(t)$ on each $X_i$ is a scalar.}
\end{remark}

The proof of Proposition \ref{prop:extend-presentation} consists of two easy lemmas.

\begin{lemma}\label{lem:stack-iso}
Let $X = \AA^n_\CC$ and let $G$ be an algebraic group with a representation $\rho: G \to GL(X)$. Let $\nu: \Gm \to GL(X)$ be a 1-parameter subgroup that commutes with $\rho(G) \subset GL(X)$. Let $G \times \Gm$ act on $X \times \AA^1$ as in \eqref{eq:extended-action}.
There is a morphism $j:[X/G] \to [(X\times \AA^1) / (G \times \Gm)]$ with a retraction, such that 
if $U \subset X$ is any $G$-invariant subscheme, $j$ restricts to an isomorphism of $[U/G]$ and $[(U\times \CC^*)/(G\times\Gm)]$. 
\end{lemma}
\begin{proof}
Recall from e.g. \cite[Prop~2.6, Thm~4.1]{romagny} the prestack $[X/G]^{pre}$: the objects over a scheme $S$ are $X(S)$ 
and an arrow from $x \in X(S)$ to $y \in X(S)$ is an element $g \in G(S)$ such that $g \cdot x = y$. 
The map $j$ is induced by the map of prestacks $[X/G]^{pre} \rightarrow [(X\times \AA^1)/(G\times \Gm)]^{pre}$ given by
\begin{equation}\label{eq:iso}
(x;g) \mapsto (x,1;g,1) \quad (x;g) \in X\times G,
\end{equation}
and its retraction $p$ is induced by the map of prestacks $[(X\times \AA^1)/(G\times \Gm)]^{pre} \to [X/G]^{pre}$ given by
\[
(x,y; g,\gamma) \mapsto (\nu(y)^{-1}x, g) \quad (x,y;g,\gamma) \in (X\times \CC^*)\times(G\times \Gm).
\]
One checks that $p \circ j$ is the identity, that $j$ factors through $[(X\times \CC^*)/(G \times \Gm)]^{pre}$, and that after restricting the codomain accordingly $j$ is essentially surjective and fully faithful on every groupoid fiber.

\end{proof}

For $N \in \ZZ$, define a character
\[\vartheta_N: G\times \Gm \rightarrow \Gm \quad \quad \quad \quad \vartheta_N(g, \gamma) = \theta(g) \gamma^N.
\]

\begin{lemma}\label{lem:check-ss} 
Let $(X, G, \theta)$ be a GIT presentation and let $\nu$ be as in Lemma \ref{lem:stack-iso}. Assume there is moreover an integer $r \geq 1$ and a 1-parameter subgroup $\nu'$ of $G$ satisfying \eqref{eq:limit} for every $x \in X$ and $s \geq 0$. 
Choose an integer $N$ such that $N>\langle \theta, \nu' \rangle/r$. \note{This works even when $\langle \theta, \nu' \rangle < 0$.}
\begin{enumerate}
\item There are equalities
\[(X\times \AA^1)^{ss}_{\vartheta_N}(G\times \Gm)  = (X\times \AA^1)^{s}_{\vartheta_N}(G\times \Gm)= X^{ss}_\theta(G) \times \CC^*.\]
\item The natural map $\Stab_G(X^{ss}_\theta) \rightarrow \Stab_{G \times \Gm}((X\times \AA^1)^{ss}_{\vartheta_N})$ given by $g \mapsto (g,1)$ is a bijection.
\end{enumerate}
\end{lemma}
\begin{proof}
We use the numerical criterion of \cite[Prop~2.5]{king}. 
To prove (1), we first show $(X\times \AA^1)^{ss}_{\vartheta_N} \subset X^{ss}_\theta\times \CC^*$. Let $(x,y) \in (X\times \AA^1)^{ss}_{\vartheta_N}$. If $x$ is not in $X^{ss}_{\theta}$, then there is a 1-parameter subgroup $\lambda: \Gm \rightarrow G$ such that $\lim_{t \rightarrow 0} \lambda(t)\cdot x$ exists but $\langle \theta, \lambda \rangle < 0$. The composition $\Gm \xrightarrow{\lambda} G \rightarrow G \times \Gm$ is a 1-parameter subgroup of $G \times \Gm$, where the second arrow is $g \mapsto (g,1)$, and this subgroup witnesses the unstability of $(x,y)$ for any $y \in \CC$. Similarly, if $y =0$, then $t \mapsto (\nu'(t), t^{-r})$ defines a 1-parameter subgroup of $G \times \Gm$ witnessing the unstability of $(x,y)$ for any $x \in X$.

Next we show $X^{ss}\times \CC^* \subset (X\times \AA^1)^s$. Let $(x,y) \in X^{ss}_\theta \times \CC^*$. An arbitrary 1-parameter subgroup of $G \times \Gm$ has the form $(\lambda(t), t^s)$. If $\lim_{t \rightarrow 0}(\lambda(t)\nu(t^s)x, t^sy)$ exists then $s \geq 0$ and $\lim_{t \rightarrow 0}\lambda(t)\nu(t^s)x$ exists. This means that $\lim_{t \rightarrow 0}\lambda(t^r)\nu(t^{rs})x$ exists. Since $\lim_{t \rightarrow 0} \nu'(t^s) \nu(t^{-rs})x'$ exists for every $x'$, we assert that 
\[\lim_{t \rightarrow 0}\nu'(t^s) \nu(t^{-rs})\lambda(t^r)\nu(t^{rs})x = \lim_{t \rightarrow 0} \lambda(t^r)\nu'(t^s)x
\]
exists. This can be proved using the triangle inequality and the fact that the 1-parameter subgroup $\nu'(t^s)\nu(t^{-rs})$ of $GL(X)$ can be written as diagonal matrices $\diag(t^{a_1}, \ldots, t^{a_n})$ in some basis for $X$, for some integers $a_i \geq 0$, so when $|t| < 1$ the functions $X \to X$ defined by $\nu'(t^s)\nu(t^{-rs})$ are contractions.\note{details:
Let $g_t = \lambda(t^r)\nut^{rs}$ and let $f_t = \nu'(t^s)\nu(t^{-rs})$. We have $\lim_{t\to 0} g_t(x)$ exists for a special $x$ (call the limit $g_0x$) and $\lim_{t \to 0} f_t(y)$ exists for every $y$ (call the limit $f_0y$), and I want to show $\lim_{t \to 0} f_t\circ g_t(x)$ exists. As discussed above $f_t$ is a contraction for $|t|<1$, which means that  $|f_t(x)-f_t(y)| \leq |x-y|$ (independent of $t$). 

We will show that $\lim_{t \to 0} f_t\circ g_t(x) = f_0g_0x$. Let $\epsilon>0$ be given. Choose $\delta$ such that
(1) $|f_tg_0x - f_0g_0x| < \epsilon/2$ (possible since $\lim f_t(g_0x) = f_0g_0x$) and (2) $|g_tx - g_0x| < \epsilon/2$ (possible since $\lim g_tx = g_0x$).
By the triangle inequality and the fact that $f_t$ is a contraction
\[
|f_tg_tx - f_0g_0x| \leq |f_tg_tx - f_tg_0x| + |f_tg_0x - f_0g_0x|\leq |g_tx - g_0x| + |f_tg_0x - f_0g_0x| < \epsilon.
\]
}
But $x \in X^{ss}_\theta$ and $\lambda(t^r)\nu'(t^s)$ is a 1-parameter subgroup of $G$, so we must have
\[
\langle \theta, r\lambda + s\nu' \rangle = r\langle \theta, \lambda \rangle + s\langle \theta, \nu' \rangle \geq 0.
\]
Since $\langle \theta, \nu'\rangle < rN$ and $s \geq 0$ this implies 
\begin{equation}\label{eq:ss1}
r\langle \theta, \lambda \rangle + rsN \geq 0
\end{equation}
or equivalently, since $r\geq 1$,
\begin{equation}\label{eq:ss}
\langle \vartheta, (\lambda(t), t^s) \rangle = \langle \theta, \lambda \rangle + sN \geq 0.
\end{equation}
If equality holds in \eqref{eq:ss}, then equality also holds in \eqref{eq:ss1}, which implies that both $r\langle \theta, \lambda \rangle + s\langle \theta, \nu' \rangle$ and $s$ are zero. 
Since $x \in X^{ss}_{\theta} = X^s_{\theta}$, this implies $\lambda(t^r)\nu'(t^s)$ acts trivially on $X$, but since $s=0$ we have $\lambda(t^r)\nu'(t^s) = \lambda(t^r)$. So $\lambda(t)$ also acts trivially on $X$.\note{You can check this pointwise: $\lambda(t) = \lambda(q^r)$ for some $q$, so if $\lambda(q^r)$ acts trivially for each $q$ so does $\lambda(t)$, for each $t$. You can also say something a little bit nicer: if $\lambda(t^r)$ factors through $\ker(\rho) \subset G$, so does $\lambda(t)$. The reason is that the \'etale morphism $t^r: \Gm \to \Gm$ has a section \'etale-locally, so \'etale locally on $\Gm$ the image of $\lambda: \Gm \to G$ is contained in $\ker(\rho)$, so the image of $\lambda$ is contained in $\ker(\rho)$.} Since $s=0$ we have that $(\lambda(t), t^s)$ acts trivially on $X \times \AA^1$.
Assertion (2) is immediate using part (1).

\end{proof}

\begin{proof}[Proof of Proposition \ref{prop:extend-presentation}]
We let $G \times \Gm$ act on $X \times \AA^1$ as in \eqref{eq:extended-action} and we let $\vartheta = \vartheta_N$ as defined in Lemma \ref{lem:check-ss}. By Lemma \ref{lem:check-ss} we know that $(X\times \AA^1, G \times \Gm, \vartheta_N)$ satisfies assumptions (1)--(3) of Section \ref{sec:I1}, and combined with Lemma \ref{lem:stack-iso} we get an isomorphism
\[
X\sslash_\theta G = (X\times \AA^1)\sslash_{\vartheta_N} G \times \Gm.
\]
\end{proof}


\begin{remark}\label{rmk:embed}
Let $(X, G, \theta)$ be a GIT presentation and let $\nu: \Gm \to GL(X)$ satisfy the assumptions of Proposition \ref{prop:extend-presentation}. Then the natural identification of character groups
\[
\chi(G \times \Gm) \simeq \chi(G) \times \chi(\Gm)
\]
leads to an isomorphism of abelian groups
\begin{equation}\label{eq:char-iso}\Hom(\Pic([X/G]), \QQ) \oplus\QQ \simeq \Hom(\Pic([(X\times \AA^1)/(G \times \Gm)]), \QQ).\end{equation}
In particular, $\Hom(\Pic([X/G]), \QQ)$ embeds into to the right hand side, and the embedding is induced by the pullback morphism of Picard groups associated to the stack morphism
\[[X/G] \to [(X\times \AA^1)/(G \times \Gm)].\]

\end{remark}

%

\subsection{CI-GIT presentations}\label{sec:ci-git-pres}

Many Deligne-Mumford stacks $\cX$ can be realized as complete intersections in stacks of the form $X\sslash_\theta G$ where $(X, G, \theta)$ is a GIT presentation. The quasimap $I$-functions of such stacks can be computed using Theorem \ref{thm:Ifunc-formula}. This motivates the following definition.
\begin{definition}\label{def:CI-GIT}
A \emph{complete intersection GIT presentation
(or CI-GIT presentation)} is a tuple $(X, G, \theta, E, s)$, where
\begin{itemize}
\item $(X, G, \theta)$ is a GIT presentation, 
\item $E$ is a $G$-representation, and \item $s$ is a $G$-equivariant section of $E \times X \rightarrow X$ such that $s$ is regular and $X^s(G) \cap Z(s)$ is smooth. 
\end{itemize}
 If $\cX$ is a separated and finite type Deligne-Mumford stack, we say $(X, G, \theta, E, s)$ a \emph{CI-GIT presentation for $\cX$} if it is a CI-GIT presentation and $\cX \simeq Z(s) \sslash_\theta G$.
\end{definition}

Suppose $(X, G, \theta, E, s)$ is a CI-GIT presentation for $\cX = Z(s)\sslash_\theta G$. We present a strategy for finding other CI-GIT presentations of $\cX$.

The first step is to find $\nu: \Gm \to GL(X)$ satisfying the hypotheses of Proposition \ref{prop:extend-presentation}. Then that proposition gives us a GIT presentation $(X \times \AA^1, G \times \Gm, \vartheta)$ for $X\sslash_\theta G$. If the $G$-action on $E$ is given by $\rho: G \rightarrow GL(E)$ and if $\mu: \Gm \to GL(E)$ is a 1-parameter subgroup commuting with the image of $\rho$ (see Remark \ref{rmk:nu}), let $E(\mu)$ be the $G \times \Gm$-representation with the same underlying vector space $E$ and homomorphism 
\begin{equation}\label{eq:extended-rep}
G \times \Gm \to GL(E) \quad \quad \quad (g, \gamma) \mapsto \rho(g)\mu(\gamma).
\end{equation}

\begin{lemma}\label{lem:extend-cigit}Let $\mu: \Gm \to GL(E)$ be a homomorphism commuting with the image of $\rho$ such that the morphism
\[
\tilde s: X \times \CC^* \to E \quad \quad \quad \tilde s(x, y) = \mu(y) s(\nu(y)^{-1}x)
\]
extends to all of $X \times \AA^1$.
Then $\tilde s$ induces a $G \times \Gm$-equivariant section of $E(\mu)\times X \to X$ such that $Z(\tilde s)\sslash_\vartheta (G \times \Gm) = Z(s) \sslash_\theta G$ and $Z(\tilde s) \cap (X \times \AA^1)^s$ is smooth. If $\tilde s$ is regular, then $(X, G, \theta, E(\mu), \tilde s)$ is a CI-GIT representation of $Z\sslash_\theta G$.
\end{lemma}

\begin{proof}
One checks using the definitions in \eqref{eq:extended-action} and \eqref{eq:extended-rep} that $\tilde s$ is $G \times \Gm$-equivariant. Next,
 the isomorphism \eqref{eq:iso} induces an isomorphism $Z(s)\sslash_\theta G \simeq Z(\tilde s)\sslash_\vartheta (G \times \Gm)$. This implies moreover that $Z(\tilde s) \cap (X\times \AA^1)^s$ is a $G \times \Gm$-torsor over $Z(s)\sslash_\theta G$. Since $Z(s)\sslash_\theta G$ is known to be smooth, so is $Z(\tilde s) \cap (X\times \AA^1)^s$.
\end{proof}

\begin{example}\label{ex:mu-r}
One way to be in the situation of Lemma \ref{lem:extend-cigit} is as follows. If the vector space $E$ has dimension $m$, then $s$ is given by functions $s_1(x), \ldots, s_m(x) \in \Gamma(X, \OO_X)$. Since $X$ is $\AA^n$, each $s_i(x)$ is a complex polynomial in $n$ variables, so for $r\in \ZZ$ sufficiently large the rational functions
\begin{equation}\label{eq:s-tilde}
\tilde s_i := y^r s_i(\nu(y)^{-1}x) \quad \quad \quad i=1, \ldots m
\end{equation}
are regular on all of $X \times \AA^1$. If we set $\mu(\gamma)=\gamma^r$, then one checks that the $\tilde s_i$ are $G \times \Gm$-equivariant maps $X \times \AA^1\to E(\mu)$. Note that the $\tilde s_i$ may not define a regular sequence.
\end{example}


\begin{example}To get a regular $\tilde s$, it can be necessary to take subgroups $\mu$ other than powers of the diagonal subgroup. For example,
let $X=\AA^4$ with coordinates $x_0, \ldots, x_3$ and let $G = \Gm$ act with weights $(1,1,1,3)$. If $\theta$ is the first power of the identity character
then $X\sslash_\theta G = \PP(1,1,1,3)$. Let $E$ be the three-dimensional representation of $G$ with weights $(1, 1, 3)$ and let $s$ be given by the regular sequence $(x_1, x_2, x_3)$. 

Choose $\nu: \Gm \to GL(4)$ to be the subgroup given by $\nu(\gamma) = (1,1,1,\gamma).$
If we restrict ourselves to subgroups $\mu$ of the form $\mu(\gamma) = \gamma^r$ as in Example \ref{ex:mu-r}, then
for $r \in \ZZ$, the section $\tilde s$ of \eqref{eq:s-tilde} is given by $(y^rx_1, y^rx_2, y^{r-1}x_3)$. For this sequence to be defined everywhere we must take $r \geq 1$, but no value of $r$ produces a regular sequence. 
Instead, one creates an extended CI-GIT presentation for $Z(s)\sslash G$ by setting $\mu(\gamma) = (1,1,\gamma),$ and then $\tilde s$ is given by the same regular sequence $(x_1, x_2, x_3)$.

\end{example}

The next lemma can be used in some applications to find a regular $\tilde s$. It can also sometimes be used to find a CI-GIT presentation for $\cX$ when $\cX$ is only known to be a smooth substack of $\AA^n\sslash_\theta G$ (a priori not given by a regular sequence).

\begin{lemma}\label{lem:regular}
Let $X=\AA^n$ let $V \subset \Gamma(X, \OO_X)$ be a finite-dimensional vector subspace with basis $\{s_i\}_{i=1}^m$. If the codimension of every component of $Z(s_1, \ldots, s_m)$ is at least $k$, the set of regular sequences of length $k$ form a dense subset of $V^k$.

\end{lemma}
\begin{proof}
For each $i=1, \ldots, k$, set
\[
S_i = \{(v_1, \ldots, v_k) \in V^k \mid (v_1, \ldots, v_{i}) \;\text{is regular}\}.
\]
Clearly $S_1 = (V\setminus \{0\})\times V^{k-1}$. For $i>1$ we will show that $S_i$ is a dense subset of $S_{i-1}$.

Let $U \subset S_{i-1}$ be open and let $\mu = (v_1, \ldots, v_k) \in U$, so in particular $(v_1, \ldots, v_{i-1})$ is regular. We claim that there is an open subset $V' \subset V$ such that for $v \in V'$, the sequence $(v_1, \ldots, v_{i-1}, v)$ is regular. Granting this, since $V'$ is dense in $V$, the intersection $U \cap (v_1 \times \ldots \times v_{i-1}\times V' \times V^{k-i})$ is not empty and the lemma statement follows.

Now we prove the claim. The complement of $V'$ is the set of $v \in V$ such that $v$ is in an associated prime of the ideal $(v_1, \ldots, v_{i-1}) \subset \Gamma(X, \OO_X)$. Since each associated prime is a vector subspace of $\Gamma(X, \OO_X)$, it is in particular a closed subset, and we see that the complement of $V'$ is closed. (This uses the fact that there are only finitely many associated primes). To see that $V'$ is nonempty, suppose for contradiction that every $v \in V$ is contained in some associated prime of $(v_1, \ldots, v_{i-1})$. Since $V$ is not a union of proper subspaces, we have $V \subset \fp$ for some associated prime $\fp$. Since $\Gamma(X, \OO_X)$ is Cohen-Macaulay and $(v_1, \ldots, v_{i-1})$ is regular, the height of $\fp$ is equal to $i-1$.\note{I read this on p.8 of hochster's notes 615W14 on Cohen-Macaulay Rings} Now $\fp$ defines a point of $X$ that is contained in the base locus of $V$ but has codimension $i-1< k$, a contradiction.
\end{proof}

\section{Conditions for $I_{\beta}$ to vanish}\label{sec:computing-I-effective-classes}
Suppose we have a GIT presentation $(X, G, \theta)$ with $T \subset G$ a maximal torus.
Then $I^{X, G, \theta}(q, z) = \sum_\beta q^\beta I^{X, G, \theta}_\beta (z)$ (see \eqref{eq:Ifunc-shape}),
where a formula for $\varphi^*I^{X, G, \theta}_\beta(z)$ was given in Theorem \ref{thm:Ifunc-formula} and the sum is over all $\beta \in \Hom(\chi(G), \QQ)$. 
In this section we give conditions on $\beta$ that force $I^{X, G, \theta}_\beta$ to vanish. 

For example, if $(X, G, \theta)$ is any quasimap target (where $X$ is not necessarily a vector space), it is immediate from the definition of $I^{X, G, \theta}_{\beta}$ in \cite[Def~3.4.4]{Webb21} that $I^{X, G, \theta}_\beta=0$ if $\beta$ is not equal to the class of any quasimap to $(X, G, \theta)$. Hence, we say $\beta \in \Hom(\Pic^G(X), \QQ)$ is \emph{$I^{X, G, \theta}$-effective} if it is the class of some quasimap to $(X, G, \theta)$. We denote the set of $I^{X, G, \theta}$-effective classes by $\KK^{\geq 0}(X, G, \theta)$. Our strongest vanishing condition in this section is just a 
formula for $\KK^{\geq 0}(X, G, \theta)$ when $(X, G, \theta)$ is a GIT presentation
(Lemma \ref{lem:i-effective}). We also provide a version for CI-GIT presentations (Lemma \ref{lem:i-effective-cigit}).

\begin{remark}
Consider a CI-GIT presentation $(X, G, \theta, E, s)$. If $\delta \in \Hom(\chi(G), \QQ)$ is not $I^{X, G, \theta}$-effective, no $\beta \in( r^{Y, G}_{X, G})^{-1}(\delta)$ can be $I^{Y, G, \theta}$-effective. Hence if $\delta \in \Hom(\chi(G), \QQ)$ is not $I^{X, G, \theta}$-effective then the left hand side of \eqref{eq:Ifunc-formula} is zero. 
\end{remark}



\begin{remark}

A second application of the $I^{X, G, \theta}$-effective classes (discussed in Section \ref{sec:implications}) is that they can be used to illuminate the $z$-asymptotics of the $I$-function. This is useful because the ideal application of \eqref{eq:mirror-thm} is when $\mu(q, -z)$ is constant in $z$ (but nonzero), and this situation occurs precisely when $I^{X, G, \theta}$ is of the form $\one + \OO(z^{-1})$, and the coefficient of $z^{-1}$ doesn't vanish.
\end{remark}

In general, the conditions in this section can be used to restrict the index set in \eqref{eq:Ifunc-shape}, yielding a different expression for $I^{X, G, \theta}$. This can be advantageous because it may not be easy to see directly that these vanishing summands define relations in $H^\bullet_{\CR}(X \sslash_G T; \QQ)$. The disadvantage is that removing the vanishing summands can make the index set of \eqref{eq:Ifunc-shape} much more complicated.


\subsection{Setup: $G$-effective anticones}

Let $(X=\CC^n, G, \theta)$ be a GIT presentation. Choose $T\subset G$ a maximal torus, and assume $(X, T, \theta)$ is also a GIT presentation. Let $\xi_1, \ldots, \xi_n$ denote the weights of the action on $X$ and let $x_1, \ldots, x_n$ be coordinates on $X$ that are dual to a corresponding weight basis for $X$; i.e., for $(x_1, \ldots, x_n) \in X$ we have
\[t\cdot (x_1, \ldots, x_n) = (\xi_1(t)x_1, \ldots, \xi_n(t)x_n).
\]Recall from \cite[Sec~4.1]{CIJ18} that an \textit{anticone} for $(X, T, \theta)$ is a subset $\sA \subset \{1, \ldots, n\}$ such that 
\[
\theta = \sum_{i \in \sA} a_i \xi_i \quad \quad \text{for some}\; a_i \in \RR_{>0}.
\]
To each anticone there is an associated open subset
\[
U_{\sA} = \{(x_i)_{i=1}^n \in X \mid x_j \neq 0 \;\text{for}\;j \in \sA\}
\]
and $\bigcup U_{\sA} = X^{ss}(T)$. The condition $X^s(T)=X^{ss}(T)$ implies that if $\sA$ is an anticone, then $\{\xi_i\}_{i \in \sA}$ spans $\chi(T) \otimes_{\ZZ} \RR$ (see for example \cite[Sec~4.1]{CIJ18}).\note{one can check this using one-parameter subgroups. For example suppose $X^s=X^{ss}(T)$. We already know $X^{ss} = \bigcup U_{\sA}$. Now suppose there is an anticone $\sA$ such that $\{\xi_i\}_{i \in \sA}$ does not span $\chi(T)\otimes \RR$. Then there is a nontrivial 1-ps $\lambda: \Gm \to T$ such that $\langle \xi, \lambda \rangle=0$. Let $x \in V_{\sA}$. Then $\lim_{t \to 0}$ exists (this 1-ps acts trivially on all the coords where $x$ is nonzero), and $\langle \theta, \lambda \rangle = \sum a_i \langle \xi_i, \lambda \rangle = 0$. But $\lambda$ is not trivial.} We also define a locally closed subset
\[
V_{\sA} = \{(x_i)_{i=1}^n \in U_{\sA} \mid x_{i}=0\;\text{for}\;i \not\in \sA\}.
\]

\begin{definition}
We say an anticone $\sA$ is \textit{G-effective} if $V_{\sA} \cap X^{ss}_{\theta}(G)\neq \emptyset$.
\end{definition}

\begin{lemma}\label{lem:contained}
Let $\sA, \sB$ be subsets of $\{1, \ldots, n\}$ with $\sA \subset \sB$. If $\sA$ is an anticone, so is $\sB$. If $\sA$ is $G$-effective, so is $\sB$.
\end{lemma}
\begin{proof}
If $\sA$ is an anticone, the assertion that $\sB$ is an anticone follows from the fact that $\{\xi_i\}_{i \in \sA}$ spans $\chi(T) \otimes \RR$. 

Now suppose $\sA$ is $G$-effective. This means $V_{\sA} \cap X^{ss}(G)$ is not empty.
Since
\[
V_{\sA} \cap X^{ss}(G) \subset \{x_i=0 \mid i \not \in \sA\} \cap X^{ss}(G) \subset \{x_i=0 \mid i \not \in \sB\} \cap X^{ss}(G),
\]
the rightmost set is a nonempty open subset of the subspace $\{x_i=0 \mid i \not \in \sB\}$. But $V_{\cB}$ is also a nonempty open subset of $\{x_i=0 \mid i \not \in \sB\}$, so 
\[
V_{\sB} \cap \{x_i=0 \mid i \not \in \sB\} \cap X^{ss}(G) = V_{\sB} \cap X^{ss}(G)
\]
is not empty.
\end{proof}

\begin{remark}\label{rmk:extended-anticones}
Assume we have a GIT presentation $(X,G,\theta)$ and additional data $\nu: \CC^*\to G$, $N \in \ZZ$ satisfying the assumptions of Lemma \ref{lem:check-ss}. Let $T \subset G$ be a maximal torus and let $\{\sA_j\}_{j \in J}$ be the set of $G$-effective anticones for $(X, T, \theta)$. Then the $G$-effective anticones for $(X\times \AA^1, T \times \Gm, \vartheta_N)$ are $\{\sA_j \cup \{n+1\}\}$. This is immediate from the definition and Lemma \ref{lem:check-ss}.
\end{remark}

\subsection{The inertia condition} 
The vanishing condition in this section shrinks the index set in \eqref{eq:Ifunc-shape} without making that set too complicated (at least in many examples). It is comparable to the index set used for the toric $I$-function computations in \cite{CIJ18}.
To state it, for $\tilde \beta \in \Hom(\chi(T), \QQ)$ set 
\[S_{\tilde \beta} := \{i \in \{1, \ldots, n\} \mid \tilde \beta(\xi_i) \in \ZZ\}.\]
and define a subset
\[
\mathbb{K} := \{r^{X, T}_{X, G}(\tilde \beta) \mid S_{\tilde \beta} \;\text{is a $G$-effective anticone}\}.
\]
\begin{proposition}\label{prop:I-vanish-1} Assume $(X, G, \theta)$ is a GIT-presentation with $T \subset G$ a maximal torus.
Then in \eqref{eq:Ifunc-shape} one may sum over $\beta \in \mathbb{K}$. That is, if $\beta\not \in \mathbb{K}$ then $I_\beta=0$.
\end{proposition}

The idea behind Proposition \ref{prop:I-vanish-1} is that elements of $\KK$ correspond to elements of $T$ with nontrivial fixed locus on $X^s(G)$ (see Lemma \ref{lem:compute-stab}), so if $\beta \not \in \KK$ then the corresponding component of the inertia stack $\In{X\sslash G}$ is empty.

\begin{remark}
If $G=T$ is a torus, then under the natural identification
\[
\Hom(\Gm, T) \otimes \QQ \simeq \Hom(\chi(T), \QQ),
\]
our set $\KK$ agrees with the one defined in \cite[Sec~4.6]{CIJ18}. 
\end{remark}

\begin{lemma}\label{lem:compute-stab}
Let $\tilde \beta \in \Hom(\chi(T), \QQ)$. The following are equivalent:
\begin{enumerate}
\item $g_{\tilde \beta}$ fixes some element of $X^s(G)$.
\item $S_{\tilde \beta}$ contains a $G$-effective anticone.
\item $S_{\tilde \beta}$ is a $G$-effective anticone.
\end{enumerate}
\end{lemma}
\begin{proof}
Items (2) and (3) are equivalent by Lemma \ref{lem:contained}. For the remaining directions, observe that the fixed locus of $g_{\tilde \beta}$ is the subspace $\{x_i=0 \mid i \not \in S_{\tilde \beta}\}$. We will denote this subspace by $X^{g_{\tilde \beta}}$. 

If (3) holds, then $V_{S_{\tilde \beta}} \cap X^{ss}(G) \neq \emptyset$, but $V_{S_{\tilde \beta}} \subset X^{g_{\tilde \beta}}$, so (1) holds.

Assume (1) holds, so there is some 
\begin{equation}\label{eq:compute-stab1}
\bx := (x_1, \ldots, x_n)  \in X^{g_{\tilde \beta}} \cap X^{ss}(G).
\end{equation} 
Since $X^{ss}(G) \subset X^{ss}(T) \subset \bigcup U_{\sA}$ this implies that there exists an anticone $\sA$ such that $\bx \in X^{g_{\tilde \beta}} \cap U_{\sA}$. Since $x_i=0$ for $i \not \in S_{\tilde \beta}$ and $x_i \neq 0 $ for $i \in \sA$, we have $\sA \subset S_{\tilde \beta}$. Without loss of generality we may assume $\sA = S_{\tilde \beta}$.

It remains to show that $S_{\tilde \beta}$ is $G$-effective, i.e. $V_{S_{\tilde \beta}} \cap X^{ss}(G) \neq \emptyset$. Because of \eqref{eq:compute-stab1} we know $X^{g_{\tilde \beta}} \cap X^{ss}(G)$ is a nonemtpy open subset of $X^{g_{\tilde \beta}}$. Since $V_{S_{\tilde \beta}}$ is also a nonempty open subset of $X^{g_{\tilde \beta}}$, we have that the intersection
\[
V_{S_{\tilde \beta}} \cap X^{g_{\tilde \beta}} \cap X^{ss}(G) = V_{S_{\tilde \beta}} \cap X^{ss}(G)
\]
is nonempty.
\end{proof}

\begin{proof}[Proof of Proposition \ref{prop:I-vanish-1}]
This proposition follows immediately from Theorem \ref{thm:Ifunc-formula} and Lemma \ref{lem:compute-stab}. Indeed, if $\beta \not \in \KK$, then for each $\tilde \beta \in (r^{X, T}_{X, G})^{-1}(\beta)$ we have $\In{X\sslash_G T}_{g_{\tilde \beta}} = \emptyset$. Hence $\In{X\sslash_G T}_{g_{\tilde \beta}^{-1}}$ is also empty, and the quantity $I_{\tilde \beta}(z)$ in \eqref{eq:Ifunc-formula} is zero by Theorem \ref{thm:Ifunc-formula}.
\end{proof}
\subsection{The $I$-effective condition}
The vanishing condition in this section is the best possible, at this level of generality, and it can significantly complicate the index set in \eqref{eq:Ifunc-shape}.

 To state it, for $\tilde \beta \in \Hom(\chi(T), \QQ)$ set 
\[S^{\geq 0}_{\tilde \beta} := \{i \in \{1, \ldots, n\} \mid \tilde \beta(\xi_i) \in \ZZ_{\geq 0}\}\]
and define a subset $\KK^{\geq 0} \subset \Hom(\chi(G), \QQ)$ by
\begin{equation}\label{eq:defk} \KK^{\geq 0}:= \{r^{X, T}_{X, G}(\tilde \beta) \mid S^{\geq 0}_{\tilde \beta} \;\text{is a $G$-effective anticone} \}.
\end{equation}

\begin{proposition} Assume $(X, G, \theta)$ is a GIT-presentation with $T \subset G$ a maximal torus.
Then the set \eqref{eq:defk} is precisely the set $\KK^{\geq 0}(X, G, \theta)$ of $I$-effective classes. In particular, in \eqref{eq:Ifunc-shape} one may sum over $\beta \in \KK^{\geq 0}$.
\end{proposition}

This proposition is an immediate corollary of the next lemma. To read the lemma, recall from Section \ref{sec:formula} that for us a quasimap to $(X, G, \theta)$ is a representable morphism $q: \wP{a} \to [X/G]$ such that the point $v=0$ maps to $X\sslash G$. 
Such a morphism is determined by a principal $G$-bundle on $\wP{a}$ and an equivariant morphism from the principal bundle to $X$, or equivalently a vector bundle (associated to the principal $G$-bundle) and a section. Moreover, any vector bundle on $\wP{a}$ can be written explicitly as a direct sum of line bundles $\OO_{\wP{a}}(d)$, where $\OO_{\wP{a}}(d)$ is the line bundle with total space equal to the quotient stack $[ \AA^3 / \Gm]$ defined by the action
\[
 t \cdot (u, v, z) = (t^au, tv, t^dz)\quad \text{for}\quad t \in \Gm, (u, v, z) \in \AA^3
\]
and the projection to $\wP{a}$ is given by projection to the $u, v$ coordinates. 

There is a $\Gm$ action on $\wP{a}$ given by 
\begin{equation}\label{eq:action}
\lambda \cdot (u, v) = (\lambda^a u, v),
\end{equation}
and this induces an action on quasimaps. We say that a quasimap $q$ is $\GG_{m, \lambda}$-\textit{fixed} if it is invariant under the $\Gm$ action in \eqref{eq:action}.

\begin{lemma}\label{lem:i-effective}
Let $\tilde \beta \in \Hom(\chi(T), \QQ)$. The following are equivalent.
\begin{enumerate}
\item There exists a ${\GG_{m, \lambda}}$-fixed quasimap to $(X, G, \theta)$ with a unique basepoint at $u=0$ whose associated vector bundle is isomorphic to $\oplus \OO_{\wP{a}}(a\tilde \beta(\xi_i))$ 
\item There exists a quasimap to $(X, G, \theta)$ whose associated vector bundle is isomorphic to $\oplus \OO_{\wP{a}}(a\tilde \beta(\xi_i))$ 
\item The set $S^{\geq 0}_{\tilde \beta}$ contains a $G$-effective anticone.
\item The set $S^{\geq 0}_{\tilde \beta}$ is a $G$-effective anticone.
\end{enumerate}
The image under $r^{X, T}_{X, G}$ of classes satisfying these four equivalent conditions is the set of $I^{X, G, \theta}$-effective classes. 
\end{lemma}

Before beginning the proof, we provide some background. 
A global section of $\oplus \OO_{\wP{a}}(a\tilde \beta(\xi_i))$ is a vector of polynomials 
$(p_i(u, v))_{i=1}^n$, where $p_i(u, v)$ is a homogeneous polynomial of degree $a\tilde \beta(\xi_i)$, where $u$ has weight $a$ and $v$ has weight 1. Recall the subspace $X^{\tilde \beta} \subset X$ from \eqref{def:x-tilde-beta-intro}. Finally, recall (e.g. from \cite[Sec~5.3]{CCK}) that a $\GG_{m, \lambda}$-fixed quasimap to $(X, T, \theta)$ of class $\tilde \beta$ with a unique basepoint at $u=0$ is uniquely determined by the image of $v=0$, which will be a point of $X^{\tilde \beta} \cap X^{ss}(T)$; conversely, given $\bx \in X^{\tilde \beta} \cap X^{ss}(T)$, the associated $\GG_{m, \lambda}$-fixed quasimap has section given by
\begin{equation}\label{eq:section}
(x_1 u^{a\tilde \beta(\xi_1)}, \ldots, x_n u^{a\tilde \beta(\xi_n)}).
\end{equation}

\begin{proof}[Proof of Lemma \ref{lem:i-effective}]

It is clear that (1) implies (2). Also, (3) and (4) are equivalent by Lemma \ref{lem:contained}.

For (2) implies (3) suppose we have a quasimap to $(X, G, \theta)$ whose associated vector bundle is isomorphic to $\oplus \OO_{\wP{a}}(a\tilde \beta(\xi_i))$ and let $\sigma$ be the section of this bundle defined by the quasimap. Because every principal $G$-bundle on $\wP{a}$ is induced from some principal $T$-bundle, the morphism $q$ lifts to a morphism $\wP{a} \to [X/T]$ given by the same vector bundle and section. 
Write $\sigma = (p_i(u, v))_{i=1}^n$, where $p_i(u, v)$ is a homogeneous polynomial of degree $a\tilde \beta(\xi_i)$ when $u$ has weight $a$ and $v$ has weight 1. Set
\[
\bp:= (p_1(1, 0), \ldots, p_n(1, 0)) \in X.
\]
Since $v=0$ is an orbifold point of $\wP{a}$, it is not a basepoint of the original quasimap $q$, so we have $\bp \in X^{ss}(G) \subset X^{ss}(T)$. 
Moreover, we claim that $\bp \in X^{\tilde \beta}$. This is because if $\tilde \beta(\xi_i) < 0$ then there are no homogeneous polynomials of degree $a\tilde \beta(\xi_i)$, so $p_i(u, v)$ is already 0; on the other hand, if $\tilde \beta(\xi_i) \not \in \ZZ$ then a polynomial of degree $a\tilde \beta(\xi_i)$ cannot have any pure-$u$ term, so $p_i(1, 0)=0$. 
In sum, we have found $\bp \in X^{\tilde \beta} \cap X^{ss}(T)$.
Thus there is an anticone $\sA$ such that $\bp \in U_{\sA}$; from the definitions of $X^{\tilde \beta}$ and $U_{\sA}$ we conclude $\sA \subset S^{\geq 0}_{\tilde \beta}$. Define 
\[\sA' := \{i \in \{1, \ldots, n\} \mid p_i(1, 0) \neq 0\}.\]
In other words, $\sA'$ is chosen so that $\bp \in V_{\sA'}$. We have $\sA \subset \sA' \subset S^{\geq 0}_{\tilde \beta}$, so by Lemma \ref{lem:contained} we have that $\sA'$ is an anticone contained in $S^{\geq 0}_{\tilde \beta}$. Since $\bp \in V_{\sA'} \cap X^{ss}(G)$ this anticone is $G$-effective.




For (4) implies (1), assume $S^{\geq 0}_{\tilde \beta}$ is $G$-effective. Then $S^{\geq 0}_{\tilde \beta}$ is an anticone and there is a point $\bx \in V_{S^{\geq 0}_{\tilde \beta}} \cap X^{ss}(G)$. But there is a containment $V_{S^{\geq 0}_{\tilde \beta}} \subset X^{\tilde \beta}$, so we have $\bx \in X^{\tilde \beta} \cap X^{ss}(T)$ and by the discussion preceeding this proof we get a $\GG_{m, \lambda}$-fixed quasimap to $(X, T, \theta)$ that is given by the vector bundle $\oplus \OO_{\wP{a}}(a\tilde \beta(\xi_i))$ and section \eqref{eq:section}. Since $\bx \in X^{ss}(G)$ this data in fact defines a quasimap to $[X/G]$.

The final assertion of the Lemma follows from condition (2) and the definition of $I$-effective classes.
\end{proof}

A $G$-effective anticone $\cA$ is \textit{minimal} if whenever $\cB$ is a $G$-effective anticone with $\cB \subset \cA$, we have $\cB=\cA$.
 To compute the $I$-effective classes in practice, one first computes the minimal $G$-effective anticones (this requires computation of $X^{ss}_\theta(G)$). For each $G$-effective anticone $\sA$ we get a set 
 \begin{equation}\label{eq:CA}
 C_\sA := \{\tilde \beta \in \Hom(\chi(T), \QQ) \mid \tilde \beta(\xi_i) \in \ZZ_{\geq 0} \;\text{for each}\;i \in \sA\}.
\end{equation}
  The set of $I$-effective classes is equal to the union of the images of the sets $C_{\sA}$ under the natural map $\Hom(\chi(T), \QQ) \to \Hom(\chi(G), \QQ)$.

\begin{remark}
A word of caution is warranted about the above computation. Observe that each $C_\sA$ generates a cone in $\Hom(\chi(T), \RR)$. It is possible to have two anticones $\sB, \sA$ such that that the cone generated by $C_{\sB}$ is contained in the cone generated by $C_{\sA}$, but $C_{\sB}$ is not contained in $C_{\sA}$: the set $C_{\sB}$ could allow larger denominators than those allowed in $C_{\sA}$. This happens in Section \ref{sec:bs3-ieffective}.
\end{remark}

\subsection{CI-GIT presentations}\label{sec:vanish-cigit}
Let $(X, G, \theta, E, s)$ be a CI-GIT presentation and let $Y = Z(s)$. Then in particular $(X, G, \theta)$ is a GIT presentation, and Lemma \ref{lem:i-effective} gives equivalent conditions for $\tilde \beta \in \Hom(\chi(T), \QQ)$ to map to an $I^{X, G, \theta}$-effective class. The next lemma is a refinement for the quasimap target $(Y, G, \theta)$.

\begin{lemma}\label{lem:i-effective-cigit}
Let $\tilde \beta \in \Hom(\chi(T), \QQ)$. The following are equivalent:
\begin{enumerate}
\item $r^{X, T}_{X, G}(\tilde \beta)$ is equal to the image of an $I^{Y, G, \theta}$-effective class.
\item $X^{\tilde \beta} \cap Y^{ss}(G) = X^{\tilde \beta} \cap X^{ss}(G) \cap Y$ is not empty.
\end{enumerate}
In particular, the image of $\KK^{\geq 0}(Y, G, \theta)$ in $\Hom(\chi(G), \QQ)$ is equal to the image of the $\tilde \beta$ such that either of these conditions holds.
\end{lemma}
\begin{proof}
Assume (1). Then we have a quasimap $q: \wP{a} \to [Y/G] \subset [X/G]$ with associated vector bundle $\oplus \OO_{\wP{a}}a\tilde \beta(\xi_i)$. We apply the proof of Lemma \ref{lem:i-effective}, specifically the argument for $2 \implies 3$, to produce $\bp \in X^{ss}(G) \cap Y \cap X^{\tilde \beta}$.

Assume (2). Let $\bx \in X^{\tilde \beta} \cap Y^{ss}(G) \subset X^{\tilde \beta} \cap X^{ss}(T)$. We apply the proof of Lemma \ref{lem:i-effective}, specifically the argument for $4 \implies 1$, to get a map to $(X, T, \theta)$ with associated vector bundle $\oplus \OO_{\wP{a}}a\tilde \beta(\xi_i)$ and section \eqref{eq:section}. Since $\bx \in Y^{ss}(G)$ this data defines a quasimap to $(Y, G, \theta)$. 
\end{proof}

\subsection{Asymptotics of the $I$-function}\label{sec:implications}
Let $(X, G, \theta, E, s)$ be a CI-GIT presentation and set $Y = Z(s)$. By \cite[Thm~3.9(2)]{CCK}, the $I$-function $I^{Y, G, \theta}(q, z)$ is homogeneous of degree 0 when we set $\deg (z)=1$, $\deg(q^\beta) = \beta(\det(T_{Y\sslash G}))$, and $\deg(\alpha) = k + \age_{Y\sslash G}(g)$ for $\alpha \in H^k(\In{Y\sslash G}; \QQ)$ 
(the quantity $\age_{Y\sslash G}(g) \in \QQ_{\geq 0}$ is defined in \cite[Sec~2.5.1]{CCK15}).

Hence, bounding the degree of $q^\beta$ for $\beta \in \KK^{\geq 0}{(Y, G, \theta)}$ has implications for the asymptotics of $I^{Y, G, \theta}(q, z)$ with respect to $z$---for example, if $\deg(q^\beta) \geq 1$ for nontrivial $\beta \in \KK^{\geq 0}{(Y, G, \theta)}$ then $I^{Y, G, \theta}(q, z) = \one + \OO(z^{-1})$. To bound the degree of $q^\beta$ we use the following lemma.

\begin{lemma}\label{lem:bound-degree}
Let $(X, G, \theta, E, s)$ be a CI-GIT presentation with $Y = Z(s)$ and let $T \subset G$ be a maximal torus. Let $\xi_1, \ldots, \xi_n$ (resp. $\epsilon_1, \ldots, \epsilon_r$) be the weights of $X$ (resp. E) with respect to $T$, and define $\pmb{\xi}:=\sum_{i=1}^n \xi_i$ and $\pmb{\epsilon}: = \sum_{i=1}^r \epsilon_i$. Let $ B$ be the set of $\tilde \beta \in \Hom(\chi(T), \QQ)$ satisfying the equivalent conditions of Lemma \ref{lem:i-effective-cigit}. Then,
\begin{enumerate}
\item For any $\delta \in \KK^{\geq 0}(Y, G, \theta)$ and any $\tilde \beta \in B$ satisfying $r^{Y, G}_{X, G}(\delta) = r^{X, T}_{X, G}(\tilde \beta)$, we have $\deg(q^{\delta}) = \tilde \beta(\bxi - \pmb{\epsilon}).$
\item  We have
\[
\displaystyle \mathrm{min}_{\delta \in \KK^{\geq 0}(Y, G, \theta)} \delta(\det(T_{Y\sslash G})) = \displaystyle \mathrm{min}_{\tilde \beta \in B} \tilde \beta(\pmb{\xi}-\pmb{\epsilon}).
\]
\end{enumerate}

\end{lemma}
\begin{remark}
When $E=0$, i.e. we are dealing with a GIT presentation $(X, G, \theta)$, the set $B$ in Lemma \ref{lem:bound-degree} is equal to $\bigcup C_{\sA}$, the set of all $\tilde \beta$ such that $S^{\geq 0}$ is a $G$-effective anticone.
\end{remark}
\begin{proof}[Proof of Lemma \ref{lem:bound-degree}]
If $\delta$ is $I^{Y, G, \theta}$-effective then there is a quasimap $q: \wP{a} \to [Y/G]$ of class $\delta$. Let $\beta$ be the class of the composition $q': \wP{a} \xrightarrow{q} [Y/G] \to [X/G].$ Since every principal $G$-bundle on $\wP{a}$ comes from a principal $T$-bundle, this lifts to a quasimap ${q}'': \wP{a} \to [Y/T] \subset [X/T]$; let $\tilde \beta$ denote the class of $q''$. By Lemma \ref{lem:i-effective} we see that $S^{\geq 0}_{\tilde \beta}$ is a $G$-effective anticone for $(X, G, \theta)$, and $Y^{ss}(G) \cap X^{\tilde \beta}$ is not empty so $\tilde \beta \in B$. 
Since $s$ is regular, we can identify $E_{Y\sslash G}$ with the normal bundle of $Y\sslash G$ in $X\sslash G$ and so there is an exact sequence
\[
0 \to T_{Y\sslash G} \to T_{X\sslash G}|_{Y\sslash G} \to E_{Y\sslash G} \to 0
\]
and it follows that
\begin{align}
\delta(\det(T_{Y\sslash G})) &= \delta(\det(T_{X\sslash G})|_{Y\sslash G}) - \delta(\det(E_{X\sslash G})|_{Y\sslash G})\label{eq:thing0}\\
&= \beta(\det(T_{X\sslash G})) - \beta(\det(E_{X\sslash G})).\notag
\end{align}
Similarly, there is a generalized Euler sequence (see \cite[811]{CCK}) on $X\sslash G$ as follows:
\[
0 \to X^s(G)\times_G \mathfrak{g} \to X^s(G)\times_G X \to \varphi^*T_{X\sslash G} \to 0.
\]
 Here $\mathfrak g$ is the Lie algebra with its adjoint $G$-action, and if $F$ is a $G$-representation we define 
 $X^s(G)\times_G F := [X^s(G)\times F/G]$, where $G$ acts diagonally on $X^s(G)\times F$.
 Since the weights of $\mathfrak{g}$ come in $+/-$ pairs, it follows that
\begin{equation}\label{eq:thing1}
\beta(\det(T_{X\sslash G})) = \beta(\det(X^s(G)\times_G X)).
\end{equation}
By definition, the right hand side of \eqref{eq:thing1} is equal to
\[
 \deg( (q')^*(\det(X^s(G)\times_G X)) = \deg( (q'')^*\varphi^*\det(X^s(G)\times_G X)) = \deg (q'')^*\cL_{\pmb{\xi}} = \tilde \beta (\pmb\xi).
\]
Likewise $\beta(\det(E_{X\sslash G})) = \tilde \beta(\pmb{\epsilon})$. This shows that $\deg(q^\delta) = \tilde \beta(\bxi-\pmb{\epsilon}).$

Conversely, if $\tilde \beta \in B$ then there is a quasimap $q$ to $(Y, G, \theta)$ that lifts to a quasimap to $(X, T, \theta)$ of degree $\tilde \beta$. If $\delta$ is the degree of $q$, the same computation as before shows that $\delta(\det(T_{Y\sslash G})) = \tilde \beta(\pmb \xi - \pmb \epsilon)$.
\end{proof}


\subsection{Extending by a Weyl-invariant sector}
In this section we prove Theorem \ref{thm:intro}, stated more precisely as Theorem \ref{thm:only} below. In order to make this more careful statement, we introduce the following definitions.

Let $E \subset \Hom(\chi(G), \QQ) \simeq \QQ^{\mathrm{rank}(G)}$ be a finitely generated monoid. Recall that the \textit{groupification} of $E$ is the subgroup $E^{gp} \subset \Hom(\chi(G), \QQ)$ of elements that can be written as $\alpha - \beta$ for some $\alpha, \beta \in E$. The \textit{saturation} of $E$ is the submonoid $E^{sat} \subset \Hom(\chi(G), \QQ)$ of elements $\beta \in E^{gp}$ with the property that $n\beta \in E$ for some $n \in \ZZ_{\geq 0}$.
Let $\QQ[E]$ denote the monoid algebra with elements $q^\alpha$ for $\alpha \in E$. If $\nu: \QQ[E] \to \ZZ$ is a valuation, then $\nu$ has a natural extension to $\QQ[E^{gp}]$ and hence to $\QQ[E^{sat}]$ by defining
\begin{equation}\label{eq:valuation}
\nu(q^{\alpha - \beta}) := \nu(q^\alpha) - \nu(q^\beta) \quad \quad \quad \quad \text{for}\;\alpha, \beta \in E.
\end{equation}

\begin{definition}\label{def:i-eff nov}
Let $(X, G, \theta)$ be a GIT presentation and let 
\[\Ieff(X, G, \theta) \subset \Hom(\Pic^G(X), \QQ)\] 
be the saturation of the monoid generated by $I^{X, T, \theta}$-effective classes, where $T \subset G$ is a maximal torus and we view the $I^{X, T, \theta}$-effective classes as a subset of $\Hom(\Pic^G(X), \QQ)$ via the morphism $r^{X, T}_{X, G}$.
The \emph{$I$-effective Novikov ring} $\Lambda_{X, G, \theta}^{\Ieff}$ is the completion of $\QQ[\Ieff(X, G, \theta)]$ with respect to the induced valuation \eqref{eq:valuation}.
\end{definition}

\begin{remark}
One can show that the definition of $\Ieff(X, G, \theta)$ is independent of the choice of maximal torus $T$.\note{It is enough to show that the image of the $I^{X, T, \theta}$-effective classes is independent of $T$. Say $gTg^{-1} = T'$. Then if $x$ is a weight vector for $T$ with weight $\xi$ we have $t\cdot x = \xi(t)x$ so $gtg^{-1}\cdot gx = \xi(t) gx$ so $\xi' := \xi(g^{-1} \bullet g)$ is a weight for $T'$. If $\theta = \sum a_i\xi_i$ with $a_i >0$ then $\theta(g^{-1}\bullet g) = \sum a_i\xi_i(g^{-1}\bullet g)$ but $\theta$ is a character of $G$ so we have $\theta  = \sum a_i \xi_i'$. This shows that $T$-anticones are the same as $T'$-anticones. Also $\tilde \beta$ maps to $\beta$ means for every $\chi \in \chi(G)$ we have $\tilde \beta(\chi|_T) = \beta(\chi) \in \QQ$. Now assume $\beta$ is in the image of the $T$-effective classes, so we are given $\tilde \beta$. We want to show it is also the image of the $T'$-effective classes, so we need to construct $\tilde \beta'$. We define $\tilde \beta'(\xi') = \tilde \beta(\xi'(g \bullet g^{-1})$ noting that $\xi'(g \bullet g^{-1}) \in \chi(T)$. You can check that this also maps to $\beta$. There is a $T$-anticone $\cA$ such that $\tilde \beta(\xi_i) \in \ZZ_{\geq 0}$ for all $i \in \cA$. We've seen that $\cA$ is also a $T'$-anticone and you can check that $\tilde \beta'(\xi'_i) = \tilde \beta(\xi_i) \in \ZZ_{\geq 0}$. This shows $\tilde \beta'$ is $T'$-effective.}
\end{remark}
Observe that the $I$-effective Novikov ring $\Lambda_{X, G, \theta}^{\Ieff}$ and the Novikov ring $\Lambda_{X, G, \theta}$ share a common subring, namely the completion of the monoid algebra generated by $I$-effective classes.

\begin{example}
If $X = \AA^5$ and $G = \Gm$ acts with weights $(1, 1, 1, 2, 3)$, and if $\theta$ is the identity character, then the monoid generated by $I$-effective classes is the subset
\[
\{0, 1/3, 1/2, 4/6, 5/6\} \cup \{1+n/6 \mid n \in \ZZ_{\geq 0}\}
\]
of $\QQ$. The monoid $\Ieff(X, G, \theta)$ is equal to $(1/6)\ZZ_{\geq 0}$.
\end{example}

\begin{theorem}\label{thm:only}
Let $(X, G, \theta)$ be a GIT presentation and let $T \subset G$ be a maximal torus with Weyl group $W$. If $g \in T$ is a Weyl-invariant element such that $\In{X\sslash G}_{(g)}$ is nonempty, then 
 there is an extended GIT presentation $(X', G', \theta')$ for $X\sslash G$ and an isomorphism
\begin{equation}\label{eq:nov iso}
\Lambda^{\Ieff}_{X, G, \theta} \llbracket t \rrbracket \simeq \Lambda^{\Ieff}_{X', G', \theta'}
\end{equation}
such that 
\[
\mu' = \mu + t \one_{g} + tO(t, q)
\]
where $\mu$ (resp. $\mu'$) is the mirror map associated to $(X, G, \theta)$ (resp. $(X', G', \theta')$) and $q$ is in the $I$-effective Novikov ring $\Lambda^{\Ieff}_{X, G, \theta}$.
\end{theorem}

\begin{remark}
The Novikov ring used in \cite{CCIT15} differs slightly from the ring $\Lambda_{\cX}$ used in this paper: in particular, it is already saturated in the sense of Definition \ref{def:i-eff nov}. Hence one does not see a larger Novikov ring introduced in \cite[Def 28]{CCIT15}.
\end{remark}

\begin{proof}[Proof of Theorem \ref{thm:only}] We present the proof in three steps (leaving some technical lemmas for the end).\\ 

\noindent
\emph{Construction of the extended presentation.} 
Let $B_g = \{\tilde \beta \in \Hom(\chi(T),\QQ) \mid g_{\tilde \beta} = g^{-1}\}$. 
Recall that $W$ acts on $\Hom(\chi(T), \QQ)$.
Choose some Weyl-invariant\note{not all $\tilde \beta \in B_g$ will be Weyl-invariant, even though $g$ is. for example, $g$ could be the identity group element, and $\tilde \beta$ could be $(1,0)$ or $(0,1)$ and these are not weyl-invariant.} $\tilde \beta \in B_g$ (the class $\tilde \beta$ need not be $I$-effective).
Fix a weight basis for $X$ with corresponding (not necessarily distinct) $T$-weights $\xi_1, \ldots, \xi_n$, and choose the basis elements so that each is contained in an irreducible $G$-subrepresentation of $X$. For $\alpha \in \QQ$, define $\lceil\alpha \rceil$ to be the smallest integer $n$ such that $n \geq \alpha$. In the coordinates determined by the weight basis, we define
\[
\nu(t) :=(t^{-\lceil \tilde \beta(\xi_1)\rceil},\, \ldots \,,\,t^{-\lceil \tilde \beta(\xi_n)\rceil}).
\]

We show that $\nu(t)$ satisfies the hypotheses of Proposition \ref{prop:extend-presentation}. The first step is to show that if $X' \subset X$ is an irreducible $G$-representation, then $X'$ is contained in a weight space of $\nu(t)$ (Remark \ref{rmk:nu}). It is enough to show that the numbers $\tilde \beta(\xi_i)$ are the same for each weight $\xi_i$ of $X'$. Consider the 1-parameter subgroup $\tau_{\tilde \beta}: \Gm \to G$ defined by 
\[
\zeta(\tau_{\tilde \beta}(t)) = t^{a \tilde \beta(\zeta)} \quad \quad \quad\text{for each}\;\zeta \in \chi(T),
\]
where $a \in \ZZ_{>0}$ is minimal such that $a\tilde \beta(\zeta) \in \ZZ$ for every $\zeta \in \chi(T)$ (see \cite[Lem~4.6]{CCK15}). It is enough to show that $X'$ is contained in a weight space of $\tau_{\tilde \beta}$. But $\tau_{\tilde \beta}$ is Weyl-invariant, since $\tilde \beta$ is, so this follows from Lemma \ref{lem:weyl-weight} below.
The other step is to demonstrate the existence of $r$ and $\nu'$. We set $\nu' := \tau_{\tilde \beta}$ and $r=a$. We desire
\[
-as\tilde \beta(\xi_i) + as \lceil \tilde \beta(\xi_i) \rceil  \geq 0
\]
for all $s \geq 0$, but this is immediate from the definition of the ceiling function. Hence $\nu(t)$ satisfies the hypotheses of Proposition \ref{prop:extend-presentation}.
By Proposition \ref{prop:extend-presentation} we have an extended GIT presentation $(X \times \AA^1, G \times \Gm, \vartheta)$ where $\vartheta(g, \gamma) = \theta(g)\gamma^N$ for some (or any) sufficiently large $N$. In \eqref{eq:bound N}, we will impose an additional lower bound on $N$ (besides the one arising from Proposition \ref{prop:extend-presentation}).\\

\noindent
\emph{Isomorphism of Novikov rings.} 
Let ${\xi}_1', \ldots, {\xi}'_{n+1}$ be the weights of $X' := X \times \AA^1$ with respect to $T':= T\times \Gm$, and note that ${\xi}'_{n+1}(t, s) = s$ for $(t, s) \in T \times \Gm$. Let ${\tilde \beta}' \in \Hom(\chi(T\times \Gm), \QQ)$ be the direct sum of $\tilde \beta$ and the morphism sending the identity character of $\Gm$ to 1. \note{Although the class $\tilde \beta$ may not have been $I$-effective, the class $\un{\tilde \beta} = \tilde \beta'$ is $I$-effective for the extended presentation. One reason is that we compute its coefficient in the next paragraph and we don't get zero (as long as the sector corresponding to $g$ is nonempty). Another reason is that having a nonempty fixed locus implies there is a set of $\xi_i$ spanning $\theta$ such that $\tilde \beta(\xi_i) \in \ZZ$. Hence $\lceil \tilde \beta(\xi_i) \rceil = \tilde \beta(\xi_i)$ for these weights, and $\un{\tilde \beta}(\un{\xi_i}) = \tilde \beta(\xi_i) - \lceil \tilde \beta(\xi_i) \rceil = 0 \in \ZZ_{\geq 0}$ for the same set of weights. But this set of $\un{\xi_i}$, together with $\un{\xi_{n+1}},$ spans $\vartheta$ by construction}
Let $\beta = r^{X, T}_{X, G}(\tilde \beta)$, let $G':= G \times \Gm,$ and recall the identification
\[
\Hom(\Pic^{G'}(X'), \QQ) = \Hom(\Pic^G(X), \QQ) \oplus \QQ
\]
from \eqref{eq:char-iso}. We define
\begin{align*}
F: \Ieff(X, G, \theta) \times \ZZ_{\geq 0} &\to \Hom(\Pic^G(X), \QQ) \oplus \QQ\\
(\delta, d) &\mapsto (\delta+d\beta, d).
\end{align*}
In a moment, we will show that the image of $F$ is precisely $\Ieff(X', G', \theta').$ Granting this, to show the isomorphism \eqref{eq:nov iso}, we also need to compare two valuations on $\QQ[\Ieff(X, G, \theta)][t]$. These are the valuations\note{couple points: the integer $e$ doesn't change when you extend because the extended bit of $\theta$ is integral. also, there is a claim that $\Lambda \llbracket t \rrbracket$ arises from the first valuation. Write down the double limit that apriori defines $\Lambda [[t ]]$. Limits commute. You can compute the double limit just on the diagonal ideals $\{I_n[t] + t^n\Lambda[t]\}$, and you can check that these diagonal ideals are correctly interlaced with the ideals $\{I_n + tI_{n-1} + t^2I_{n-2} + \ldots + t^n\Lambda[t]\}$, which are precisely the ideals defining the first valuation. The second valuation is the valuation on the extended presentation, pulled back via $F$.}
\[
\nu(q^{\delta}t^d) = e\delta(\theta) + d \quad \quad \quad \quad \nu'(q^\delta t^d) = e\big( (\delta + d\beta)(\theta) + dN \big)
\]
where $e$ is the smallest integer satisfying $e\delta(\theta) \in \ZZ$ for all $I^{X, G, \theta}$-effective $\delta$ (see Section \ref{sec:jfuncs}).
To show that $\nu$ and $\nu'$ produce isomorphic completions it is enough to show that for each $n \geq 0$ there exists $m \geq 0$ such that $\nu(q^\delta t^d) \geq m$ implies $\nu'(q^\delta t^d) \geq n$, and similarly with the roles of $\nu$ and $\nu'$ reversed. One can check that this holds as long as
\begin{equation}\label{eq:bound N}
e\beta(\theta) + eN -1 \geq 0
\end{equation}
which we can always ensure by choosing $N$ large enough. Indeed, granting this, we find that 
\[\nu(q^\delta t^d) \geq m \quad \quad \text{implies } \quad \quad \nu'(q^\delta t^d) \geq m.\]
Conversely, if $n \geq 0$ is given, the assumption
\[
\nu'(q^\delta t^d) \geq (e\beta(\theta) + eN) n
\]
implies
\[
(e\beta(\theta) + eN)^{-1} e\delta(\theta) + d \geq n.
\]
Since $\delta(\theta) \geq 0$ for any $I^{X, G, \theta}$-effective class $\delta$ (\cite[Lem~3.2.1]{stable_qmaps}), this together with \eqref{eq:bound N} implies $\nu(q^\delta t^d) \geq n$ as required.

To show that $F$ factors through $\Ieff(X', G', \theta')$, since saturation is a left adjoint, it is enough to check the image under $F$ of pairs $(\delta, d)$ where $\delta$  is $I^{X, T, \theta}$-effective and $d \in \ZZ_{\geq 0}.$ By Lemma \ref{lem:i-effective} there is $\tilde \delta \in \Hom(\chi(T), \QQ)$ and an anticone $\cA$ such that $r^{X, T}_{X, G}(\tilde \delta) = \delta$ and $\tilde \delta(\xi_i) \in \ZZ_{\geq 0}$ for all $i \in \cA$. Then
\begin{equation}\label{eq:defF}
F(\delta, d) = (\delta, 0) + d(\beta, 1)
\end{equation}
is in the (saturation of the) monoid generated by $I^{X', T', \theta'}$-effective classes. Indeed, $(\delta, 0)$ is $I^{X', T', \theta'}$-effective by Lemma \ref{lem:i-effective} and Remark \ref{rmk:extended-anticones} using the lift $(\tilde \delta, 0)$ and the anticone $\cA \cup\{n+1\}$ (here $n$ is the dimension of $X$). On the other hand, by Lemma \ref{lem:compute-stab} our assumption that $\In{X\sslash G}_{(g)}$ is nonempty implies there is a ($G$-effective) anticone $\cB$ such that $\tilde \beta(\xi_i) \in \ZZ$ for $i \in \cB$. It follows that 
\[
\tilde \beta'(\xi_i') = \tilde \beta(\xi_i) - \lceil \tilde \beta(\xi_i)\rceil = 0 
\]
for $i \in \cB$, so $(\beta, 1)$ is $I^{X', T', \theta'}$-effective using the lift $(\tilde \beta, 1)$ and the anticone $\cB \cup \{n+1\}.$ This shows that $F$ factors through $\Ieff(X', G', \theta').$

Conversely, since $F$ is injective, to show that the image of $F$ is $\Ieff(X', G', \theta')$ it is enough to show that every $I^{X', T', \theta'}$-effective class is in the image of $F$.\note{From here, here is how to get surjectivity in general. Let $(\delta, d)$ be in $\Ieff(X', G', \theta')$, so $(\delta, d)$ is an integer linear combination of $I^{T'}$-effective classes and $(D\delta, Dd)$ is a positive integer linear combination of $I^{T'}$-effective classes for some $D \geq 0$. That is we may write
\[
(\delta, d) = \sum a_i(\delta_i, d_i) \quad \quad \quad \quad (D\delta, Dd) = \sum p_j(\epsilon_j, e_j)
\]
for some $a_i \in \ZZ$, $p_j \in \ZZ_{\geq 0}$, and $(\delta_i, d_i), (\epsilon_j, e_j)$ $I^{T'}$-effective. Write $(\delta_i, d_i) = F(x_i, d_i)$ and $(\epsilon_j, e_j) = F(y_j, e_j)$ for some $x_i$ and $y_j$ in $\Ieff(X, G, \theta)$ (note these may not be $I^T$-effective, but they are in the groupification of this set and some positive multiple is in this set). Then $(\delta, d) = F(\sum a_i (x_i, d_i))$ and $(D\delta, Dd) = F(\sum p_j (y_j, e_j))$, so injectivity of $F$ implies $D\sum a_i (x_i, d_i) = \sum p_j (y_j, e_j)$ and hence $D\sum \a_i x_i = \sum p_j y_j$. (Also, at this point we just need to show that $\sum a_i(x_i, d_i)$ is an element of $\Ieff(X, G, \theta) \times \ZZ_{\geq 0}$. We know $\sum a_i d_i = d \in \ZZ_{\geq 0}$ so we just have to check about $\sum a_i x_i$.) So now $\sum a_i x_i$ is an element of the groupification of $I^T$-effective classes with the property that $D \sum a_i x_i = \sum p_j y_j$, so if we take a further positive multiple $E$ such that each $Ey_j$ is $I^T$-effective we have that $ED\sum a_i x_i$ is $I^T$-effective. This shows that $\sum a_i x_i$ is in $\Ieff(X, G, \theta)$.
}
So let $(\delta, d) \in \Hom(\Pic^G(X), \QQ) \oplus \QQ$ correspond to an $I^{X', T', \theta'}$-effective class. According to Lemma \ref{lem:i-effective} and Remark \ref{rmk:extended-anticones} we have that $d \in \ZZ_{\geq 0}$ and there exists $\tilde \delta \in \Hom(\chi(T), \QQ)$ and an anticone $\cA$ such that $r^{X, T}_{X, G}(\tilde \delta) = \delta$ and
\begin{equation}\label{eq:main2}
\tilde \delta(\xi_i) + \lceil \tilde \beta(\xi_i) \rceil d \in \ZZ_{\geq 0}
\end{equation}
for each $i \in \cA$. We claim that $\delta -d\beta$ is in $\Ieff(X, G, \theta)$, so that $(\delta - d\beta, d)$ is an element of the left hand side of \eqref{eq:defF} mapping to $(\delta, d)$. Indeed, from \eqref{eq:main2} we have $\tilde \delta(\xi_i) \in \ZZ$ for each $i \in \cA$ and from the previous paragraph we have $\tilde \beta(\xi_i) \in \ZZ$ for each $i \in \cB$, so by Lemma \ref{lem:groupification} we have that $\delta - d\beta$ is in the groupification of the monoid generated by $I^{X, T, \theta}$-effective classes. Moreover we compute
\[
0 \leq \tilde \delta(\xi_i) -c\lceil \tilde \beta(\xi_i) \rceil  \leq \tilde \delta(\xi_i) - d\tilde \beta(\xi_i) = (\tilde \delta - d \tilde \beta)(\xi_i)
\]
where the first inequality is \eqref{eq:main2}. Hence for sufficiently large $n \in \ZZ_{\geq 0}$ we have $n(\tilde \delta - d \tilde \beta)(\xi_i) \in \ZZ_{\geq 0}$ for all $i \in \cA$. It follows that $\delta  - d\beta$ is in $\Ieff(X, G, \theta)$ as claimed.\\

\noindent
\emph{Computation of the mirror map.}
The subgroup $\nu$ was chosen so that ${\tilde \beta}'({\xi}_\ell') \in (-1, 0]$ for $\ell\leq n$ and ${\tilde \beta'}({\xi}'_{n+1})=1.$ Hence $C({\tilde \beta}', {\xi'}_\ell) = 1$ for $\ell \leq n$ and $C({\tilde \beta}', {\xi}_\ell') = z^{-1}$ (after restriction to $X\sslash G$). Moreover, Weyl-invariance of $\tilde \beta$ tells us that for any root $\rho_i$ of $G$ and $w \in W$, we have $\tilde \beta(\rho_i) = \tilde \beta(w \rho_i)$. Taking $w$ to be the reflection sending $\rho_i$ to $-\rho_i$ we see that $\tilde \beta(\rho_i)=0$, so if $ \rho_i'$ is the root of $G'$ corresponding to $\rho_i$, then $C({\tilde \beta}', \rho_i')=1$.
From Theorem \ref{thm:Ifunc-formula}, one sees that the $q^{\tilde \beta'}$-coefficient of $\varphi^*I^{X', G', \vartheta}$ is $z^{-1}\one_{g}$. Since $\tilde \beta$ was Weyl-invariant, it is the unique class mapping to $\beta$ (Lemma \ref{lem:inj r}), and this is also the coefficient of $t = q^\beta$. 

On the other hand, if $(\delta, 0) \in \Ieff(X', G', \theta')$ corresponds to an actual $I^{X', G', \theta'}$-effective class, then one can use Lemma \ref{lem:i-effective} to show that $\delta$ is $I^{X, G, \theta}$-effective. In this case Theorem \ref{thm:Ifunc-formula} shows that the coefficient of $q^{(\delta, 0)}$ in $I^{X', G', \theta'}$ is equal to the coefficient of $q^\delta$ in $I^{X, G, \theta}$.

\end{proof}

\begin{lemma}\label{lem:weyl-weight}
Let $X'$ be an irreducible $G$-representation, $T \subset G$ a maximal torus, and $\tau: \Gm \to T$ a Weyl-invariant homomorphism. Then there is an integer $b \in \ZZ$ such that for every $x \in X$ we have $\tau(t) \cdot x = t^bx$.
\end{lemma}
\begin{proof}
The natural inclusion of the center of $G$ to the Weyl-invariant part of $T$ is an isomorphism on identity components \cite{marty}. Hence every $W$-invariant 1-parameter subgroup is central. Suppose the subspace of $X$ of weight $b$ with respect to $\tau$ is nonempty. We claim this weight space is in fact a $G$-subrepresentation of $X$: for $x\in X$ of weight $b$, we have
\[
\tau(t) \cdot gx = g\tau(t) x = gt^b x = t^b(gx)
\]
using centrality of $\tau$. Since $X$ is irreducible, this weight space is all of $X$.
\end{proof}

\begin{lemma}\label{lem:groupification}
Let $(X, G, \theta)$ be a GIT presentation. The groupification of the monoid generated by $I^{X, T, \theta}$-effective classes in $\Hom(\chi(G), \QQ)$ is the set of integer linear combinations of elements of the form $r^{X, T}_{X, G}(\tilde \beta)$ such that $S_{\tilde \beta}$ contains an anticone.
\end{lemma}
\begin{proof}
The forward containment is clear from Lemma \ref{lem:i-effective}, so we only need to show that if $\tilde \beta \in \Hom(\chi(T), \QQ)$ has 
 the property that $S_{\tilde \beta}$ contains an anticone, then $r^{X, T}_{X, G}(\tilde \beta)$ is in the groupification of the monoid generated by $I^{X, T, \theta}$-effective classes. For such $\tilde \beta$, let $\cA$ be an anticone contained in $S_{\tilde \beta}$ such that the cone generated by $\cA$ is strongly convex.\footnote{We may arrange for $\cA$ to be strongly convex as follows. Given an anticone $\cA$, by definition $\theta$ is in the cone generated by $\cA$. Triangulate $\cA$. Since $X^s_\theta(T) = X^{ss}_\theta(T)$ the character $\theta$ is interior to one of the simplices in the triangulation. This simplex is strongly convex and its rays are the desired anticone.}\note{You can find a strongly convex $\cA$ that may not be $G$-effective as follows. Triangulate the cone spanned by the $S_{\tilde \beta}$. By assumption $\theta$ is contained in one of the simplices, but every simplex  is strongly convex because you can complete its generators to a basis and then change basis so that the simplicial cone looks like a positive quadrant, which is strongly convex.} Then there is a hyperplane in $\chi(T)_{\QQ}$ such that all the characters $\xi_i$ for $i \in \cA$ are on one side of this hyperplane, and hence there is an $\tilde \alpha \in \Hom(\chi(T), \ZZ)$ such that $\tilde \alpha(\xi_i) \in \ZZ_{\geq 0}$ for all $i \in \cA$. In particular we have that $r^{X, T}_{X, G}(\tilde \alpha)$ is $I^{X, T, \theta}$-effective by Lemma \ref{lem:i-effective}. But for sufficiently large $M$ we have $(\tilde \beta + M\tilde \alpha)(\xi_i) \in \ZZ_{\geq 0}$ for all $i \in \cA$, so 
 \[r^{X, T}_{X, G}(\tilde \beta + M\tilde \alpha) = r^{X, T}_{X, G}(\tilde \beta) + Mr^{X, T}_{X, G}(\tilde \alpha)\]
 is also $I^{X, T, \theta}$-effective. It follows that $r^{X, T}_{X, G}(\tilde \beta)$ is in the groupification of the monoid generated by $I^{X, T, \theta}$-effective classes.
\end{proof}

\begin{lemma}\label{lem:inj r}
Let $(X, G, \theta)$ be a GIT presentation. If $\tilde \beta, \tilde \alpha \in \Hom(\chi(T), \QQ)$ are two Weyl-invariant elements such that \begin{equation}\label{eq:inj r}r^{X, T}_{X, G}(\tilde \beta) = r^{X, T}_{X, G}(\tilde \alpha),\end{equation}
then $\tilde \beta = \tilde \alpha$.
\end{lemma}
\begin{proof}
The assumption \eqref{eq:inj r} implies that $\tilde \alpha(\theta) = \tilde \beta(\theta)$ for every Weyl-invariant $\theta \in \chi(T)$. Hence it suffices to show that if $\tilde \beta$ is Weyl-invariant, then it is determined by its value on Weyl-invariant characters. This is true because if $\theta \in \chi(T)$ is any character, using Weyl invariance of $\tilde \beta$ we have
\[
\tilde \beta(\theta) = \frac{1}{|W|} \sum_{w \in W} (w\cdot \tilde \beta)(\theta) = \frac{1}{|W|}   \tilde \beta\left(\sum_{w \in W} w^{-1} \cdot \theta\right)
\]
and we note that $\sum_{w \in W} w^{-1} \cdot \theta$ is Weyl-invariant.

\end{proof}

\section{Examples}\label{sec:examples}
The first two examples in this section are included as illustrations of the theorems in \cite{Webb21} and this paper. They are not new results. The remaining examples provide genuinely new computations of $J$-functions.
In the examples, we use the following notation:
\begin{itemize}
\item $\zeta_r$ is the complex number $e^{2 \pi i/r}$ and $\bmu_r$ is the group of $r^{th}$ roots of unity
\item For $a_1, \ldots, a_n \in \CC$, we write $\diag(a_1, \ldots, a_n)$ for the $n\times n$ matrix with the $a_i$ on the diagonal
\item $\one$ is the fundamental class of the untwisted sector of $\In{X\sslash G}$ (see Section \ref{sec:background})
\end{itemize}

In each example we will study a GIT presentation $(X, G, \theta)$ (sometimes there will also be a representation $E$ and section $s$). We will use $T$ to denote the maximal torus, which is sometimes only isomorphic to a subgroup of $G$. If the presentation is an extended one we will use the notation $(\bX, \bG, \vartheta)$ and $\bT$ for the maximal torus. In each case one can use the numerical criterion of \cite[Prop~2.5]{king} to show that $(X, G, \theta)$ is a GIT presentation. 

The torus $T$ will always be identified with a product of groups of diagonal matrices, and we denote the standard basis of projection characters by $e_1, \ldots, e_{\mathrm{rk}(T)}$, and for $\tilde \beta \in \Hom(\chi(T), \QQ)$ we will write $\tilde \beta_i:= \tilde \beta(e_i)$ and $H_i := c_1(\cL_{e_i}) \in H^\bullet(\In{X\sslash_G T}; \QQ)$. In this basis we have $g_{\tilde \beta} =  (e^{2 \pi i \tilde \beta_1}, \ldots, e^{2 \pi i \tilde \beta_{\mathrm{rk}(T)}})$. We will likewise choose a basis for $\chi(G)$ in each example and for $\beta \in \Hom(\chi(G), \QQ)$ we will write $\beta_1, \ldots, \beta_{\mathrm{rk}(\chi(G))}$ for the corresponding coordinates of $\beta$. We use $r$ to denote the morphism
\begin{equation}\label{eq:r}
r:= r^{X, T}_{X, G} \Hom(\chi(T), \QQ) \to \Hom(\chi(G), \QQ).
\end{equation}
Via our chosen bases for $\chi(T)$ and $\chi(G)$, we will identify $r$ with a map of $\QQ$-vector spaces.

In each example, we will also choose an ordered weight basis for $X$ with respect to $T$. An anticone is then a set $\{i_1, \ldots, i_a\}$ of distinct integers corresponding to indices of elements in this basis, and $C_{\{i_1, \ldots, i_a\}}$ is the set defined in \eqref{eq:CA}. The set $\KK^{\geq 0}(X, G, \theta)$ is then equal to $r(\bigcup C_{\sA})$ where the union is over all $G$-effective anticones $\sA$. Finally, we will write $I(q, z)$ for $I^{X, G, \theta}(q, z)$. We know that
\[
\varphi^*I(q, z) = \sum_{\beta \in \KK^{\geq 0}(X, G, \theta)} q^\beta \sum_{\tilde \beta \to \beta} I_{\tilde \beta}(z),
\]
where the second sum is over all $\tilde \beta \in \bigcup C_{\sA}$ mapping to $\beta$, and we will give a formula for $I_{\tilde \beta}(z)$.



\note{
\begin{remark}\label{rmk:only-weyl}
Remark \ref{rmk:bad-extension} suggests that if one attempts to apply Strategy \ref{strat:only} to an element $g \in T$ that is not Weyl-invariant, the subgroup $\nu$ will usually not commute with the image of $G$ in $GL(X)$. 
\end{remark}
}

\subsection{Trivial extension}\label{sec:trivial-extension}


Let $(X, G, \theta)$ be a GIT presentation. We construct an extended presentation so that the $I$-function has an additional parameter tracking the unit $\one \in H^\bullet_{CR}(X\sslash_\theta G; \QQ)$.

Let $\nu: \Gm \to GL(X)$ be the trivial homorphism. By Theorem \ref{thm:only} (with $\tilde \beta$ equal to the zero homomorphism), the extension $(X\times \AA^1, G \times \Gm, \vartheta_1)$ is also a GIT presentation for $X\sslash_\theta G$.
Let $I(q, z)$ denote the $I$-function for the original presentation and let $\hat I(\hat q, z)$ denote the $I$-function for the extended presentation. 
By Proposition \ref{prop:extend-presentation} there is an isomorphism $f: X\sslash G \to (X\times \AA^1) \sslash (G \times \Gm)$. Moreover, we have 
\[
\Hom(\chi(G\times \Gm), \QQ) = \Hom(\chi(G), \QQ) \oplus \Hom(\chi(\Gm), \QQ)
\]
so we write $\hat \beta \in \Hom(\chi(G\times \Gm), \QQ)$ as $\beta + d\alpha$ where $\beta \in \Hom(\chi(G), \QQ)$, $d \in \QQ$, and $\alpha: \chi(\Gm) \to \QQ$ sends the identity character to 1.
We will show
\begin{equation}\label{eq:trivial-extension}
f^* \hat I(\hat q, z) = e^{t_0/z} I(q, z)
\end{equation}
after setting $q^\beta = \hat q^\beta$ and $t_0 = \hat q^\alpha$ on the right hand side.

By Remark \ref{rmk:extended-anticones}, the anticones for the extended presentation are $\sA\cup\{n+1\}$ where $\sA$ is a $G$-effective anticone for $(X, G, \theta)$. One checks that
\[
C_{\sA\cup\{n+1\}} = C_{\sA} \times \ZZ_{\geq 0}
\]
where $C_{\sA}$ was defined in \eqref{eq:CA}.
It follows that $\beta + d\alpha$ is $I$-effective if and only if $\beta$ is $I$-effective and $d \in \ZZ_{\geq 0}$. Using Theorem \ref{thm:Ifunc-formula} one computes
\[
\varphi^*f^*\hat I(\hat q, z) = f^*\varphi^*\hat I(\hat q, z) = \left( f^*\sum_{d\geq 0}\hat q^{d\alpha} \frac{1}{\prod_{k=1}^d (c_1(\L_{\xi_{n+1}})+kz)} \right)  \varphi^* I(\hat q, z) 
\]\note{on the rhs, $\hat q$ is supposed to be there---the variable is $\hat q^\beta$ instead of $q^\beta$.}
where $\xi_{n+1}$ is the character of $G \times \Gm$ given by projection to $\Gm$, where we also use $f$ to denote the isomorphism $X\sslash T \to (X\times \AA^1)\sslash (T \times \Gm)$.
The character $\xi_{n+1}$ pulls back to the trivial character of $G$ under the inclusion $G \to G \times \Gm$ used to define $f$, so $f^*\L_{\xi_{\n+1}}$ is the trivial bundle and $f^*c_1(\L_{\xi_{n+1}})=0$. Setting $t_0=q^\alpha$ we obtain the desired result.

\begin{remark}
The $I$-function in \eqref{eq:trivial-extension} can also be computed with \cite[~Theorem 4.2]{CCK15}.
\end{remark}

\subsection{Classifying stack of $S_3$}\label{sec:bs3}
The following GIT presentation for $BS_3$ was told to us by Yang Zhou.
\begin{itemize}
\item $X := \Sym^3(\CC^2)$ with ordered basis $(x^3, x^2y, xy^2, y^3)$. We write $p \in X$ as a polynomial $p(x, y)$ and view it as a function on $\CC^2$.
\item $G := GL(2)/\bmu_3$, where $GL(2)$ acts on $X$ by $g \cdot p = p \circ g^{-1}$ and $\bmu_3$ is the subgroup $\diag(\zeta_3, \zeta_3) \subset GL(2)$
\item $\theta \in \chi(G)$ is induced by the function $GL(2) \to \Gm$ sending $g$ to $\det(g)^{-3}$
\end{itemize}


We choose $T' \subset G$ to be the image of the diagonal matrices in $GL(2)$. There is an isomorphism $T' \to (\CC^*)^2 =: T$ induced by the morphism from diagonal matrices to $T$ given by \begin{equation}\label{eq:bs3-1}
\diag(\lambda, \mu) \mapsto (\lambda^{-1} \mu^{-2}, \lambda^{-2}\mu^{-1}).
\end{equation}
The weights of $X$ under the action of $T$ (with the ordered basis given above) are given by the columns of the matrix
\[
\left(\begin{array}{cccc}
-1 & 0 & 1 & 2\\
2 & 1 & 0 & -1
\end{array}\right).
\]
The character $\theta$ restricts to the character of $T$ given by $(s, t) \mapsto st$.
One can check using the numerical criterion of \cite[Prop~2.5]{king} that $(X, G, \theta)$ is a GIT presentation for $BS_3$, with $X^{ss}(G)$ equal to the locus in $X$ of polynomials with distinct linear factors, and that each element of $\Stab_G(X^{ss}_{\theta}(G))$ is conjugate to an element of $T'$ that maps to the set $\{ (1,1),(1,\zeta_2), (\zeta_2, 1), (\zeta_3, \zeta_3^2), (\zeta_3^2, \zeta_3) \} \subset T$.

\subsubsection{Geometry of inertia orbifolds}
In this section we make the projection $\varphi: \In{X\sslash_G T} \to \In{X\sslash G}$ completely explicit.
For $a_1, a_2 \in \QQ$, let $X^{a_1, a_2}$ be the subspace of $X$ fixed by $(e^{2 \pi ia_1}, e^{2\pi i a_2}) \in T$. 
Using \cite[Sec 2.1]{Webb21}, 
we compute that each sector of $\In{X\sslash G}$ is isomorphic to a sector of $\In{X\sslash_G T}$, and fixing an isomorphism the map $\varphi$ can be written explicitly as follows:
\begin{equation}\label{eq:bs3-inertia}
\begin{tikzcd}
X\sslash_G T \arrow[d] &[-1em] X^{0, \frac{1}{2}}\sslash_G T \arrow[dr, "id"'] &[-5em]&[-5em] X^{\frac{1}{2}, 0}\sslash_G T \arrow[dl, "f"] &[-1em] X^{\frac{1}{3}, \frac{2}{3}}\sslash_G T\arrow[dr, "id"'] &[-5em]&[-5em] X^{\frac{2}{3}, \frac{1}{3}}\sslash_G T\arrow[dl, "g"]\\
 X\sslash G && X^{0, \frac{1}{2}}\sslash_G T &&& X^{\frac{1}{3}, \frac{2}{3}}\sslash_G T\\
\end{tikzcd}
\end{equation}
Here, the top row lists sectors of $\In{X\sslash_G T}$, the bottom row lists sectors of $\In{X\sslash G}$ (after an isomorphism), and the vertical maps are $\varphi$ (under the chosen isomorphism).
If we write
\[
X^{0, \frac{1}{2}} = \{ax^3 + cxy^2\},\quad X^{\frac{1}{2}, 0} = \{bx^2y + dy^3\},\quad  X^{\frac{1}{3}, \frac{2}{3}} = X^{\frac{1}{3}, \frac{2}{3}} = \{ax^3 + dy^3\},
\]
then the maps $f$ and $g$ are given by $f(bx^2y + dy^3) = dx^3 + bxy^2$ and $g(ax^3 + dy^3) = dx^3 + ay^3$. Observe that all stacks in the last (resp. penultimate) column of \eqref{eq:bs3-inertia} are isomorphic to $B\ZZ_3$ (resp. $B\ZZ_2$). 
A basis of $H^\bullet(\In{X\sslash G}; \QQ)$ is $\{\one, \one_{B\bmu_2}, \one_{B\bmu_3}\}$, where $\one_{B\bmu_2}$ (resp. $\one_{B\bmu_3}$) is the fundamental class of the middle sector (resp. rightmost sector) in \eqref{eq:bs3-inertia}.

\subsubsection{$I$-function}\label{sec:bs3-ieffective}
The minimal anticones for $X$ with respect to $T$ are $\{1, 4\}$, $\{1, 3\}$, $\{2,4\}$, and $\{2, 3\}$, and each is $G$-effective. 
\note{Each of these anticones is $G$-effective: for example, $V_{\{2, 3\}}$ is equal to polynomials of the form $a_2x^2y + a_3xy^2$ with $a_2a_3\neq 0$. This anticone is $G$-effective because $x^2y + xy^2$ has distinct roots.}
We have
\begin{align*}
C_{\{2, 3\}} &= \left\{ \begin{array}{c}
  (\tilde \beta_1, \tilde \beta_2) \in \ZZ^2 \\
  \tilde \beta_1 \geq 0,\;\;  \tilde \beta_2 \geq 0 \\
\end{array} \right\} & 
C_{\{2, 4\}} &= \left\{ \begin{array}{c}
  (\tilde \beta_1, \tilde \beta_2) \in \ZZ \times (\ZZ + \frac{1}{2}) \\
  \tilde \beta_1 \geq 0,\;\; -\tilde \beta_1 + 2\tilde \beta_2 \geq 0 \\
\end{array} \right\}\\
C_{\{1, 4\}} &= \left\{ \begin{array}{c}
  (\tilde \beta_1, \tilde \beta_2) \in (\ZZ+\frac{i}{3})\times (\ZZ + \frac{3-i}{3}) \\
  2\tilde \beta_1 - \tilde \beta_2 \geq 0,\;\; -\tilde \beta_1 + 2\tilde \beta_2 \geq 0 \\
\end{array} \right\} & C_{\{1, 3\}} &= \left\{ \begin{array}{c}
  (\tilde \beta_1, \tilde \beta_2) \in (\ZZ + \frac{1}{2})\times \ZZ \\
  2\tilde \beta_1 - \tilde \beta_2 \geq 0,\;\; \tilde \beta_2 \geq 0 \\
\end{array} \right\} &
\end{align*}
We choose $\theta$ to be our basis for $\chi(G)$, so the map $r$ in \eqref{eq:r} is identified with the linear map sending $(\tilde \beta_1, \tilde \beta_2)$ to $\tilde \beta_1+\tilde \beta_2$, and we have
\[
\KK^{\geq 0}(X, G, \theta) = (1/2)\ZZ_{\geq 0} \subset \QQ.
\]


Using \eqref{eq:bs3-1} we see that the roots of $G$ are equal to $\pm(e_1-e_2)$.
\note{Since $G$ is the quotient of $GL(2)$ by a finite group, the roots of $G$ with respect to $T'$ are identified with the roots of $GL(2)$ with respect to diagonal matrices, under the map \eqref{eq:bs3-1}. Hence the roots of $G$ are equal to $\pm(e_1-e_2)$.}
We compute
\begin{align} \label{eq:IBS3}
I_{\tilde \beta}(z) = &C(\tilde \beta, e_1-e_2)^{-1}\;C(\tilde \beta, e_2-e_1)^{-1}\\
&C(\tilde \beta, -e_1+2e_2)\;C(\tilde \beta, e_2)\;C(\tilde \beta, e_1)\;C(\tilde \beta, 2e_1-e_2)\,\one_{g_{\tilde \beta}^{-1}}\notag,
\end{align}
where the factors $C(\tilde \beta, e_i)$ always appear as denominators since $\tilde \beta(e_i) \geq 0$ for every $\tilde \beta \in \bigcup C_{\sA}$.
In Lemma \ref{lem:bs3-asymptotics} below, we show that
\begin{equation}\label{eq:bs3-asymptotics}
I(q, z) = \one + \frac{1}{z}  q^{1/2} \one_{B\bmu_2}+ \frac{3q}{2z^2}(\one_{B\bmu_3} + \one_{}) + O(z^{-3}).
\end{equation}

Let $\cX := X\sslash G = BS_3$. Since $H_2(\cX; \QQ)=0$, the Novikov ring $\Lambda_{\cX}$ is just $\QQ$ and the composition \eqref{eq:def-iota} is injective.
If we write $1 = \iota(Q)$ for the trivial Novikov variable of $\cX$, by \eqref{eq:mirror-thm} and \eqref{eq:trivial-extension} we have that
\begin{equation}\label{eq:bs3-jfunc}
J(t_0\one_{} + t_1\one_{B\bmu_2}, 1, z) = e^{t_0/z} I(q, z)|_{q^{1/2}=t_1}.
\end{equation}
 This restriction lets us compute the 3-point correlators of $BS_3$ with insertions equal to $\one$ and $\one_{B\bmu_2}$. (All the 3-point correlators of $BS_3$ can be obtained by other methods, e.g. \cite[Prop~3.1]{JK02}.)
Unfortunately, we cannot extend this GIT presentation to recover more invariants: the only homomorphisms $\nu: \Gm \to GL(4)$ satisfying the hypotheses of Lemma \ref{lem:stack-iso} are $\nu(t)=t^k$ for some integer $k$.
One checks that none of these extensions gives a more general restriction of the $J$-function than the one in \eqref{eq:bs3-jfunc}.

\begin{lemma}\label{lem:bs3-asymptotics}
The equality \eqref{eq:bs3-asymptotics} holds.
\end{lemma}
\begin{proof}
From Lemma \ref{lem:bound-degree} 
we know that for any $\beta \in \KK^{\geq 0}(X, G, \theta) = (1/2)\ZZ_{\geq 0}$ and any $\tilde \beta \in \bigcup C_{\sA}$ mapping to $\beta$, 
\[
\deg(q^\beta) = 2\tilde \beta_1 + 2\tilde \beta_2 = 2 \beta \geq 1
\]
Since the degrees of $\one_{}, \one_{B\bmu_2}$, and $\one_{\bmu_3}$ are all zero, homogeneity implies $I^{X\sslash G} = \one + O(1/z)$
and the coefficient of $z^{-1}$ in $I(q, z)$ is a multiple of $q^{1/2}$. We recall that $I_\beta(z) = \sum_{\tilde \beta \to \beta} I_{\tilde \beta}(z)$ and sum equation \eqref{eq:IBS3} over $\tilde \beta$ mapping to $\beta = 1/2$, obtaining
\begin{align*}
\varphi^*I_{1/2}(z) &= (H_1 - H_2 + (1/2)z) (2H_1-H_2+z)^{-1}(H_1+(1/2)z)^{-1}\one_{(\zeta_2, 1)}\\
&+ (H_2 - H_1 + (1/2)z) (2H_2-H_1+z)^{-1}(H_2+(1/2)z)^{-1}\one_{(1, \zeta_2)}
\end{align*}
Now $H_i=0$ since the ambient stacks are zero dimensional. So the coefficients of $\one_{(\zeta_2, 1)}$ and $\one_{(1, \zeta_2)}$ are series in $z$; in fact they are the same series because of the geometry of the situation (see \eqref{eq:bs3-inertia}).
The series is $1/z$.

Similarly, the coefficient of $z^{-2}$ is a multiple of $q$. From \eqref{eq:IBS3}, we have
\begin{align}
\varphi^*I_1(z) &= (H_2-H_1+(1/3)z)(-H_1+2H_2+z)^{-1}(H_2+(2/3)z)^{-1}(H_1+(1/3)z)^{-1}\one_{(\zeta_3^2, \zeta_3)}\label{eq:big-thing}\\
&+(H_1-H_2+(1/3)z)(H_2+(1/3)z)^{-1}(H_1+(2/3)z)^{-1}(2H_1-H_2+z)^{-1}\one_{(\zeta_3, \zeta_3^2)}\notag\\
&+\frac{H_1-H_2+z}{H_2-H_1}(-H_1+2H_2)(H_1+z)^{-1}(2H_1-H_2+z)^{-1}(2H_1-H_2+2z)^{-1}\one_{(1, 1)}\notag\\
&+\frac{H_2-H_1+z}{H_1-H_2}(2H_1-H_2)(-H_1+2H_2+z)^{-1}(-H_1+2H_2+2z)^{-1}(H_2+z)^{-1}\one_{(1, 1)}\notag
\end{align}
The first two summands are the pullback of a multiple $c$ of $\one_{B\bmu_3}$, and this can be computed as in the previous paragraph. One obtains $c = 3/(2z^2)$. \note{the last two summands came from (1,0) and (0,1).}
For the second two summands, note that $\In{X\sslash T'}_{(1,1)}$ is 2-dimensional, so a priori the classes $H_i$ are not zero here. The denominators in \eqref{eq:big-thing} only make sense when we consider the sum of the last two terms together, and the division must be computed by finding an anti-invariant numerator divisible by $H_1-H_2$. There is a unique way to do this: see \cite[Lem~5.4.1]{nonab1}. Indeed, the coefficient of $z^{-2}$ in the last two terms of \eqref{eq:big-thing} is equal to
\begin{align*}
\frac{z}{H_2-H_1}&(-H_1+2H_2)\left(\frac{1}{z}\right)\left(\frac{1}{z}\right)\left(\frac{1}{2z}\right)\one_{(1, 1)} - \frac{z}{H_2-H_1}\left(\frac{1}{z}\right)\left(\frac{1}{2z}\right)\left(\frac{1}{z}\right)(2H_1-H_2)\one_{(1, 1)}
\end{align*}
which simplifies to $3/(2z^2)\one_{(1, 1)}$.
\end{proof}

 \subsection{Weighted Grassmannian }\label{sec:wgr} Weighted Grassmannians were introduced in \cite{CR02}.
 Donagi and Sharpe \cite{DS08} give the following GIT data 
 for a certain weighted Grassmannian $wGr(2, 5)$:
\begin{itemize}
\item $X:=\AA^{10}$ is the space of $2 \times 5$ matrices $M$ over $\CC$
\item $G := GL(2)$, where $\Lambda \in GL(2)$ acts on $X$ by
\[
\Lambda \cdot M = \Lambda M \diag(1, 1, 1, \det(\Lambda), \det(\Lambda))
\]
\item $\theta \in \chi(G)$ is the determinant character
\end{itemize}
We choose $T \subset G$ to be the subgroup of diagonal matrices.
The weights of $X$ under the action of $T$ are 
\[
\begin{blockarray}{cccccccccc}
m_{11} & m_{12} & m_{13} & m_{14} & m_{15} & m_{21} & m_{22} & m_{23} & m_{24} & m_{25}  \\
\begin{block}{(cccccccccc)}
  1 & 1 & 1 & 2 & 2 & 0 & 0 & 0 & 1 & 1 \\
  0 & 0 & 0 & 1 & 1 & 1 & 1 & 1 & 2 & 2 \\
\end{block}
\end{blockarray}
 \]
where $(m_{ij})_{1\leq i \leq 2, 1\leq j \leq 5}$ are coordinates on an element of $X=M_{2 \times 5}$.
The set $X^{ss}(G)$ is equal to the set of full rank matrices, and each element of $\Stab_G(X^{ss}_{\theta}(G))$ is conjugate to an element of $\{ (1,1),(1,\zeta_2), (\zeta_2, 1), (\zeta_3, \zeta_3), (\zeta_3^2, \zeta_3^2) \} \subset T$.

\begin{remark}\label{rmk:w-inv-sectors}
The elements $(1, \zeta_2)$ and $(\zeta_2, 1)$ of $\Stab_G(X^{ss}_{\theta}(G))$ are not Weyl-invariant. This demonstrates that stacks with abelian isotropy groups may have Chen-Ruan cohomology that is not all supported on Weyl-invariant sectors.
\end{remark}

\subsubsection{$I$-function}\label{sec:wgr-ieff}
The minimal anticones for $X$ with respect to $T $ are $\{i, 5+j\}$ where $i, j \in \{1, \ldots, 5\}$, and of these the $G$-effective ones are precisely those with $i \neq j$. We have

\begin{center}
{\tabulinesep=.8mm
 \begin{tabu}{c|c||c|c}
 $\sA$ & $C_{\sA}$ &$\sA$ & $C_{\sA}$\\ \hline
  $i, j \in \{1,2,3\}$& 
$\left\{ \begin{array}{c}
  \tilde \beta\in \ZZ^2 \\
  \tilde \beta_1 \geq 0,\;\;  \tilde \beta_2 \geq 0 \\
\end{array} \right\}$ &
$\begin{array}{c}i \in \{1,2,3\}\\ j \in \{4,5\} \end{array}$&
$ \left\{ \begin{array}{c}
  \tilde \beta \in \ZZ^2+(0, \frac{1}{2})\ZZ\\
  \tilde \beta_1 \geq 0, \\ \tilde \beta_1 + 2\tilde \beta_2\geq 0 \\
\end{array} \right\}$ \\ 
$\begin{array}{c}i \in \{4,5\}\\j \in \{1,2,3\}\end{array}$ & $ \left\{ \begin{array}{c}
  \tilde \beta \in\ZZ^2+ (\frac{1}{2}, 0)\\
  \tilde \beta_2 \geq 0, \\ 2\tilde \beta_1 + \tilde \beta_2 \geq 0\\
\end{array} \right\}$ & 
$i,j \in \{4, 5\}$ & $ \left\{ \begin{array}{c}
  \tilde \beta \in \ZZ^2 + (\frac{1}{3}, \frac{1}{3})\ZZ \\
  2 \tilde \beta_1 + \tilde \beta_2  \geq 0,\\ \tilde \beta_1 + 2\tilde \beta_2  \geq 0
\end{array} \right\}$ \\

\end{tabu}
    }
\end{center}

We choose $\theta$ as a basis for $\chi(G)$. Then the map $r$ in \eqref{eq:r} is identified with the linear map $\QQ^2 \to \QQ$ sending $(\tilde \beta_1, \tilde \beta_2)$ to $(\tilde \beta_1 + \tilde \beta_2)$, and 
\[
\KK^{\geq 0}(X, G, \theta) = (1/2)\ZZ_{\geq 0} \cup (1/3)\ZZ_{\geq 0} \subset \QQ.
\]
\note{the image has $1/3\ZZ$ and not just $2/3\ZZ$: for example $(-1/3, 2/3)$ is in here and it maps to $1/3$.}

\label{sec:wgr-ifunc}
The roots of $G$ with respect to $T$ are $\pm(e_1-e_2)$. We compute
\begin{align}\label{equ:wgr-ifunc}
I_{\tilde \beta}(z) = &C(\tilde \beta, e_1-e_2)^{-1}\; C(\tilde \beta, e_2-e_1)^{-1} \\
&C(\tilde \beta, e_1)^3\; C(\tilde \beta, e_2)^{3}\; C(\tilde \beta, 2e_1+e_2)^2\; C(\tilde \beta, e_1+2e_2)\;  \one_{g_{\tilde \beta}^{-1}}\notag,
\end{align}
where the factors $C(\tilde \beta, 2e_i+e_j)$ always appear as denominators.
One checks using Lemma \ref{lem:bound-degree} that $\deg(q^\beta) \geq 3$\note{ (this is realized, for example, by $\tilde \beta = (2/3, -1/3)$} so the mirror map $\mu(q, z)$ is zero. Then \eqref{eq:mirror-thm} shows that $J(0, \iota(Q), z) = I(q, z)$. Observe that $\iota$ is injective by Remark \ref{rmk:extend-pic}.

Finally, we can get a ``bigger'' slice of the Lagrangian cone as follows. The $(\Gm)^5$-action on $X$ given by scaling the columns of $M$ descends to an action on $X\sslash G$, and the induced action on the coarse moduli space of $X\sslash G$ has isolated fixed points and isolated 1-dimensional orbits.\note{I'm a little bit unsure about what the correct fixed locus or one-dimensional orbit on the coarse space of $X\sslash G$ is, but it is certainly contained in the stupidest definition (the one in terms of sets). This one agrees with the description of fixed points and 1-dim orbits for Gr. In a little more detail: any point of $X^{ss}$ can be put into "standard form" where two of the columns are equal to the identity matrix. A point is stupid-fixed iff, when it is in standard form, all other coordinates are zero. There are 10 such points. (This is almost certainly not a fixed point of $X\sslash G$ without taking a cover of the stack; maybe it is a fixed point of the coarse space?) A point generates a 1-dim orbit iff, when it is in standard form, exactly one other coordinate is not zero. Scaling that coordinate is in a single orbit.} Hence we may apply \cite[Thm~4.2]{CCK}. Using $t_0$ and $t_1$ as formal variables, we obtain that $\varphi^*J(t_0\one + t_1 c_1(\cL_\theta)\one, q, z)$ is equal to 
\[
\one + \sum_{\beta \in \KK^{\geq 0}(X, G, \theta) \setminus\{ 0\}} q^\beta \mathrm{exp}\left( \frac{1}{z}(t_0\one + t_1(H_1 + H_2)\one + \beta z)\right) \sum_{\tilde \beta \to\beta} I_{\tilde \beta}(z).
\]

\subsubsection{Extension}\label{sec:wgr-extension}
We use Proposition \ref{prop:extend-presentation} to create a GIT presentation for $X\sslash_\theta G$ that ``sees'' the twisted sector indexed by $(\zeta_3, \zeta_3)\in T$. The resulting $I$-function has a complicated mirror map, so we will not write the formula, but we will use the computations in this section to work out the complete intersection example in Section \ref{sec:delpez}. 

By Theorem \ref{thm:only} (with $\tilde \beta = (-1/3, -1/3)$), if we define $\nu: \Gm \to GL(X)$  by 
\[
\nu(\gamma) \cdot M = M \diag(1,1,1,\gamma,\gamma) \quad \quad M \in X,\; \gamma\in \Gm.
\]
we get an extended presentation $(\bX, \bG , \vartheta_3)$ where the action of $\bG:=G \times \Gm$ on $\bX:=X \times \AA^1$ is induced by
\[
(\Lambda, \gamma) \cdot (M, y) = (\Lambda M \diag(1, 1, 1, \det(\Lambda)\gamma, \det(\Lambda)\gamma), \gamma y)
\]
for $(\Lambda, \gamma) \in GL(2) \times \Gm$ and $(M, y) \in X\times \AA^1$ .
We choose $\bT := T \times \Gm$ for our maximal torus. The weights of $\bX$ under the action of $\bT$ are the columns of the matrix
\[
\begin{blockarray}{ccccccccccc}
m_{11} & m_{12} & m_{13} & m_{14} & m_{15} & m_{21} & m_{22} & m_{23} & m_{24} & m_{25} & y \\
\begin{block}{(ccccccccccc)}
  1 & 1 & 1 & 2 & 2 & 0 & 0 & 0 & 1 & 1 & 0\\
  0 & 0 & 0 & 1 & 1 & 1 & 1 & 1 & 2 & 2 & 0\\
  0 & 0 & 0 & 1 & 1 & 0 & 0 & 0 & 1 & 1 & 1\\
\end{block}
\end{blockarray}.
 \]

\subsubsection{$I$-function}\label{sec:wgr-effective}
The minimal anticones for $\bX$ with respect to $\bT$ are $\{i, 5+j, 11\}$ where $i, j \in \{1, \ldots, 5\}$, and of these the $G$-effective ones are precisely those with $i \neq j$. 
\begin{center}
{\tabulinesep=.8mm
 \begin{tabu}{c|c||c|c}
 $\sA$ & $C_{\sA}$ &$\sA$ & $C_{\sA}$\\ \hline
  $i, j \in \{1,2,3\}$& 
$\left\{ \begin{array}{c}
  \tilde \beta \in \ZZ^3 \\
  \tilde \beta_1 , \tilde \beta_2, \tilde \beta_3 \geq 0
\end{array} \right\}$ &
$\begin{array}{c}i \in \{1,2,3\}\\ j \in \{4,5\} \end{array}$&
$\left\{ \begin{array}{c}
  \tilde \beta \in \ZZ^3 + (0,\frac{1}{2},0)\ZZ  \\
  \tilde \beta_1, \tilde \beta_3\geq 0,\\ \tilde \beta_1 + 2\tilde \beta_2 + \tilde \beta_3\geq 0 \\
\end{array} \right\}$ \\ 
$\begin{array}{c}i \in \{4,5\}\\j \in \{1,2,3\}\end{array}$ & $ \left\{ \begin{array}{c}
  \tilde \beta \in\ZZ^3+ (\frac{1}{2},0,0)\ZZ\\
  \tilde \beta_2, \tilde \beta_3\geq 0,\\ 2\tilde \beta_1 + \tilde \beta_2 + \tilde \beta_3\geq 0\\
\end{array} \right\}$ & 
$i,j \in \{4,5\}$ & $\left\{ \begin{array}{c}
  \tilde \beta \in \ZZ^3 + (\frac{1}{3}, \frac{1}{3}, 0)\ZZ \\
  2 \tilde \beta_1 + \tilde \beta_2 + \tilde \beta_3 \geq 0,\\ \tilde \beta_3 \geq 0 \\
  \tilde \beta_1 + 2\tilde \beta_2 + \tilde \beta_3 \geq 0
\end{array} \right\}$ \\

\end{tabu}
    }
\end{center}
The lattice $\chi(\bG)$ has a basis consisting of the characters induced by $\theta \in \chi(G)$ and the identity in $\chi(\Gm)$. The map $r$ in \eqref{eq:r} is then identified with the linear map $\QQ^3 \to \QQ^2$ sending $(\tilde \beta_1, \tilde \beta_2, \tilde \beta_3)$ to $(\tilde \beta_1 + \tilde \beta_2, \tilde \beta_3)$, and we have
 \begin{equation}\label{eq:wgr-effective}
\KK^{\geq 0}(X, G, \theta)= \left\{
\begin{array}{c}
(\beta_1, \beta_2) \in (\frac{1}{2}\ZZ\times \ZZ) \cup (\frac{1}{3}\ZZ \times \ZZ)\\
\beta_2 \geq 0,\;\; 3\beta + 2\beta_2 \geq 0
\end{array}
\right\} \subset \QQ^2.
\end{equation}

By homogeneity, the $I$-function for $(\bX, \bG, \vartheta_3)$ will have unwieldy asymptotics: for every integer $k$, the class $\tilde \beta_k = (-k, -k, 3k)$ is in $C_{\{4, 10, 11\}}$, but if $\beta_k = r(\tilde \beta_k)$ then 
\[
\deg(q^{\beta_k}) = -9k-9k+15k = -3k.
\]
Hence the $I$-function may contain arbitrarily high powers of $z$. 

\subsection{Weighted Flag Space}\label{sec:wfl}
Weighted homogeneous spaces were introduced in \cite{CR02}: they include the weighted Grassmannian of Section \ref{sec:wgr} and also weighted flag spaces.
In this section we compute the $I$-function of a weighted version of the variety of complete flags in $\CC^3$. It has the following GIT data: 
\begin{itemize}
\item $X := \AA^8$ is viewed as the space of pairs $(M, N)$ where $M$ is a $1 \times 2$ matrix and $N$ is a $2 \times 3$ matrix
\item $G := \left(\Gm \times SL(2)\times \Gm\right)/\bmu_2$, where $(\rho, \Lambda, \mu) \in \Gm \times SL(2) \times \Gm$ acts on $X$ by 
\[(\rho, \Lambda, \mu) \cdot (M, N) = (\rho M \Lambda^{-1}, \Lambda N \diag(\mu, \mu^3, \mu^3)) \]
   and $\bmu_2$ is the subgroup $(\zeta_2, \diag(\zeta_2, \zeta_2), \zeta_2) \subset \Gm \times SL(2) \times \Gm$
\item $\theta$ is induced by the function $\Gm \times SL(2) \times \Gm \to \Gm$ sending $(\rho, \Lambda, \mu)$ to $\rho\mu^5$.  
\end{itemize}

 We choose $T' \subset G$ to be the quotient of the subgroup $(\rho, \diag(\lambda, \lambda^{-1}), \mu) \subset \Gm \times SL(2) \times \Gm$. We prefer to use a different basis for $T'$: there is an isomorphism
 \begin{equation} \label{eq:wfl-T-iso} 
 T' \to \Gm^3 =: T, \quad\quad (\rho, \diag(\lambda, \lambda^{-1}), \mu) \to (\rho\lambda, \mu\lambda, \mu\lambda^{-1})
 \end{equation}
 Let $M = (m_1, m_2)$ and $N = (n_{ij})_{1\leq i \leq 2, 1 \leq j \leq 3}$ be coordinates on $X$.
 The weight matrix for the $T$ action on $X$ is  
 \[
\begin{blockarray}{ccccccccc}
m_1 & m_2 & n_{11} & n_{12} & n_{13} & n_{21} & n_{22} & n_{23}\\
\begin{block}{(ccccccccc)}
 1& 1& 0& 0& 0 & 0 & 0& 0 \\
 -1 & 0 & 1& 2& 2 &0&1& 1\\
 1&0&0&1&1&1&2&2\\
\end{block}
\end{blockarray}
 \]

The character $\theta$ restricts to the character of $T$ given by $(r, s, t) \to rs^2t^3$. The locus $X^{ss}(G)$ is equal to the set of pairs $(M, N)$ where both $M$ and $N$ have full rank. Each element of $\Stab_G(X^{ss}_\theta(G))$ is conjugate to an element of $T'$ that maps to the set $\{(1,1,1), (1, \zeta_3, \zeta_3), (1, \zeta_3^2, \zeta_3^2), (1, 1, \zeta_2), (1, \zeta_2, 1), (\zeta_2, 1, \zeta_2), (\zeta_2, \zeta_2, 1)\} \subset T.$

\subsubsection{Extension}
By Theorem \ref{thm:only} (with $\tilde \beta = (0,-1/3,-1/3)$), if we define $\nu: \Gm \to GL(X)$ by 
\[
\nu(\gamma) \cdot (M, N) = (M, N \diag(1, \gamma, \gamma)) \quad \quad \quad (M, N) \in X,\; \gamma \in \Gm,
\]
we get an extended presentation $(\bX, \bG, \vartheta_6)$ where the action of $\bG:=G \times \Gm$ on $\bX:=X \times \AA^1$ is induced by
\[
(\rho, \Lambda, \mu, \gamma) \cdot (M, N, y) = (\rho M \Lambda^{-1}, \Lambda N \diag(\mu, \gamma\mu^3, \gamma\mu^3), \gamma z).
\]
for $(\rho, \Lambda, \mu, \gamma) \in \Gm \times SL(2) \times \Gm \times \Gm$ and $(M, N, y) \in X \times \AA^1$. 

 The maximal torus $T' \times \Gm \subset G \times \Gm = \bG$ is now isomorphic to $\bT:=\Gm^4$ via \eqref{eq:wfl-T-iso} on the first three factors and the identity map on the last factor. The weights of $\bX$ under the action of $\bT$ are
  \[
\begin{blockarray}{cccccccccc}
m_1 & m_2 & n_{11} & n_{12} & n_{13} & n_{21} & n_{22} & n_{23} & y\;\;\\
\begin{block}{(cccccccccc)}
 1& 1& 0& 0& 0 & 0 & 0& 0 & 0\;\;\\
 -1 & 0 & 1& 2& 2 &0&1& 1& 0\;\;\\
 1&0&0&1&1&1&2&2& 0\;\;\\
 0&0&0&1&1&0&1&1&1\;\;\\
\end{block}
\end{blockarray}
 \]

\note{
\begin{remark}\label{rmk:bad-extension}
One also extend further to create an $I$-function that ``sees'' the twisted sector indexed by $(1, \zeta_3^2, \zeta_3^2) \in T'$, but the resulting $I$-function will have unwieldy asymptotics. On the other hand, if one attempts to apply Strategy \ref{strat:only} with $\tilde \beta = (0, -1/2, 0)$, hoping to gain information about the $(1, \zeta_2, 1)$-sector, one discovers that the resulting one-parameter subgroup $\nu(\gamma): \Gm \to GL(X)$ does not commute with the image of $G$ in $GL(X)$. This is related to the fact that $(1, \zeta_2, 1)$ is not Weyl-invariant.
\end{remark}
}
 

 \subsubsection{I-function}\label{sec:wfl-effective}
 The minimal anticones for $\bX$ with respect to $\bT$ are $\{i, j+2, k+5, 9\}$ where $i=1$ or 2 and $j, k \in \{1,2,3\}$, and of these the $G$-effective ones are those with $j \neq k$. We compute

\begin{center}
{\tabulinesep=.8mm
 \begin{tabu}{c|c||c|c}
 $\sA$ & $C_{\sA}$ &$\sA$ & $C_{\sA}$\\ \hline
  $\begin{array}{c}\{1, 3, 7, 9\}\\ \{1, 3, 8, 9\} \end{array}$& 
$\left\{ \begin{array}{c} 
\tilde \beta \in \ZZ^4 + (\frac{1}{2},0, \frac{1}{2}, 0)\ZZ \\
\tilde \beta_1-\tilde \beta_2+\tilde \beta_3 \geq 0 \\
\tilde \beta_2 + 2\tilde \beta_3 +\tilde \beta_4 \geq 0\\
\tilde \beta_2, \tilde \beta_4 \geq 0
\end{array} \right\}$ &
$\begin{array}{c}\{2, 3, 7, 9\}\\ \{2, 3, 8, 9\} \end{array}$&
$\left\{ \begin{array}{c} 
\tilde \beta \in \ZZ^4 + (0,0, \frac{1}{2}, 0)\ZZ \\
\tilde \beta_1, \tilde \beta_2 \geq 0 \\
\tilde \beta_2 + 2\tilde \beta_3 +\tilde \beta_4 \geq 0\\
\tilde \beta_4 \geq 0
\end{array} \right\}$ \\ 
$\begin{array}{c}\{1, 4, 6, 9\}\\ \{1, 5, 6, 9\}\end{array}$ & $\left\{ \begin{array}{c} 
\tilde \beta \in \ZZ^4 + (\frac{1}{2}, \frac{1}{2}, 0, 0)\ZZ \\
\tilde \beta_1-\tilde \beta_2+\tilde \beta_3 \geq 0 \\
2\tilde \beta_2 + \tilde \beta_3 +\tilde \beta_4 \geq 0 \\
\tilde \beta_3, \tilde \beta_4 \geq 0
\end{array} \right\}$ & 
$\begin{array}{c}\{2, 4, 6, 9\}\\ \{2, 5, 6, 9\}\end{array}$ & $\left\{ \begin{array}{c} 
\tilde \beta \in \ZZ^4 + (0, \frac{1}{2}, 0, 0)\ZZ \\
2\tilde \beta_2 + \tilde \beta_3 +\tilde \beta_4 \geq 0 \\
\tilde \beta_1, \tilde \beta_3 \geq 0\\
\tilde \beta_4 \geq 0
\end{array} \right\}$ \\
$\begin{array}{c} \{1, 4, 7, 9\}\\ \{1, 4, 8, 9\}\\ \{1, 5, 7, 9\}\\ \{1, 5, 8, 9\}\end{array}$  & $\left\{ \begin{array}{c} 
\tilde \beta \in \ZZ^4 + (0, \frac{1}{3}, \frac{1}{3}, 0)\ZZ \\
\tilde \beta_1-\tilde \beta_2+\tilde \beta_3 \geq 0 \\
2\tilde \beta_2 + \tilde \beta_3 +\tilde \beta_4 \geq 0 \\
\tilde \beta_2 + 2\tilde \beta_3 +\tilde \beta_4 \geq 0\\
\tilde \beta_4 \geq 0
\end{array} \right\}$ &$ \begin{array}{c} \{2, 4, 7, 9\}\\ \{2, 4, 8, 9\}\\ \{2, 5, 7, 9\}\\ \{2, 5, 8, 9\}\end{array}$  & $\left\{ \begin{array}{c} 
\tilde \beta \in \ZZ^4 + (0, \frac{1}{3}, \frac{1}{3}, 0)\ZZ \\
\tilde \beta_1 \geq 0 \\
2\tilde \beta_2 + \tilde \beta_3 +\tilde \beta_4 \geq 0 \\
\tilde \beta_2 + 2\tilde \beta_3 +\tilde \beta_4 \geq 0\\
\tilde \beta_4 \geq 0
\end{array} \right\}$ \\

\end{tabu}
    }
\end{center}
 
The lattice $\chi(G)$ is the sublattice of $\chi(\Gm \times SL(2) \times \Gm)$ consisting of characters that have $\bmu_2$ in their kernel. A basis for $\chi(G)$ is given by the characters $(\rho, \Lambda, \mu) \mapsto \rho\mu$ and $(\rho, \Lambda, \mu) \mapsto \mu^2$, which restrict to the characters $e_1+e_2$ and $e_2+e_3$ of $T$, respectively. A basis for $\chi(\bG)$ consists of the characters induced by these two and the projection character $G \times \Gm \to \Gm$. With these bases the map $r$ is identified with the map $\QQ^4 \to \QQ^3$ sending $(\tilde \beta_1, \ldots, \tilde \beta_4)$ to $(\tilde \beta_1+\tilde \beta_2, \tilde \beta_2 + \tilde \beta_3, \tilde \beta_4)$, and we have $\KK^{\geq 0}(\bX, \bG, \vartheta) = r(\bigcup C_{\sA})$.

The roots of $\Gm \times SL(2) \times \Gm$ are given by the characters $\Gm \times \diag(\lambda, \lambda^{-1}) \times \Gm \mapsto \lambda^{\pm 2}$,\note{one way to get these is to restrict the usual roots of $GL(2)$ to the subtorus $\diag(\lambda, \lambda^{-1})$} and hence the roots of $G$ with respect to $T$ (also the roots of $\bG$ with respect to $\bT$) are equal to $\pm(e_2-e_3)$. 
\note{Since $G$ is the quotient of $\Gm \times SL(2) \times \Gm$ by a finite group, the roots of $G$ with respect to $T'$ are identified with the roots of $\Gm \times SL(2) \times \Gm$ with respect to the subgroup $\Gm \times \diag(\lambda, \lambda^{-1}) \times \Gm$ under the map \eqref{eq:wfl-T-iso}. }
We have

\begin{align}\label{eq:wfl-ifunc}
I_{\tilde \beta}(z) = &C(\tilde \beta, e_2-e_3)^{-1}\; C(\tilde \beta, e_3-e_2)^{-1}\;C(\tilde \beta, e_1)\;\\
& C(\tilde \beta, e_1-e_2+e_3)\; C(\tilde \beta, e_2)\; C(\tilde \beta, e_3)\; C(\tilde \beta, 2e_2+e_3)^2\; C(\tilde \beta, e_2+2e_3)^2 \;\one_{g_{\tilde \beta}^{-1}} \notag
\end{align}
 Define the element $\one_{1/3} \in H^\bullet(\In{X\sslash G}; \QQ)$ to be the fundamental class of the sector corresponding to $(1, \zeta_3, \zeta_3) \in T$. 
 \note{(There are more twisted sectors than this, but we will not need notation for their fundamental classes.)} 
By arguments similar to those used in Lemma \ref{lem:bs3-asymptotics}, 
\begin{equation}\label{eq:wfl-ifunc2}
I(q, z) = \one + \frac{1}{z}q_1^{-1/3}q_2^{-2/3}q_3 \one_{1/3} + \OO(z^{-2}).
\end{equation}
The morphism $\iota$ of Novikov variables is injective, since the arrows
\[
H_2^{alg}(\cX; \QQ) \to \Hom(\Pic(X\sslash G), \QQ) \to \Hom(\Pic([X/G]), \QQ) \to \Hom(\Pic([\bX/\bG]), \QQ).
\]
are all injective by Remark \ref{rmk:extend-pic} and Corollary Remark \ref{rmk:embed}. We will write $(q_1, q_2) = \iota(Q)$. If $\mu(q_1, q_2, q_3) = q_1^{-1/3}q_2^{-2/3}q_3\one_{1/3}$ then 
\[
J^{X\sslash G}(q_1^{-1/3}q_2^{-2/3}q_3\one_{1/3}, (q_1,q_2), z) = I(q, z).
\]
Setting $q_3 = q_1^{1/3}q_2^{2/3}x$ gives a restriction of the $J$-function that can be used to compute invariants with insertions in the sector corresponding to $(1, \zeta_3, \zeta_3) \in T'$.

\note{
\begin{lemma}\label{lem:wfl-ifunc2}
The equality \eqref{eq:wfl-ifunc2} holds.
\end{lemma}
\begin{proof}
We use Lemma \ref{lem:bound-degree} to bound $\deg(q^\beta)$: we minimize $\bxi = 2e_1 + 6e_2 + 8e_3 + 5e_4$ on the sets $C_\sA$ given in Section \ref{sec:wfl-effective} and find that $\deg(q^\beta) \geq 1/3$ for nontrivial $\beta$. This implies that $I(q, z) = \one + O(z^{-1}).$

By the homogeneity of $I^{\bX, \bG, \vartheta_6}$ discussed in Section \ref{sec:implications}, a term of $I(q, z)$ with $z^{-1}$ must also contain $q^\beta$ and $\alpha \in H^k(\In{\bX\sslash \bG}_{(g_{\tilde \beta}^{-1})}; \QQ)$ satisfying $\deg(q^\beta) + k + \age_{\bX\sslash \bG}(g) = 1$. We compute the ages of the elements of $T' \times \Gm$ that represent conjugacy classes in $\Stab_\bG(\bX^{ss}(\bG))$ (this can be done using the generalized Euler sequence on $\bX \sslash \bG$ as in the proof of Lemma \ref{lem:bound-degree}):
\[
\age_{\bX\sslash \bG}(1,1,1,1) = 0 \quad \quad 
\age_{\bX\sslash \bG}(1,\zeta_3,\zeta_3,1) = 2/3 \quad \quad 
\age_{\bX\sslash \bG}(1,\zeta_3^2,\zeta_3^2,1) = 4/3 
\]
The ages of the remaining four elements are all equal to 1. Hence, the only possibilities are $\deg(q^\beta)=1$, $g_{\tilde \beta} = (1,1,1,1)$, and $k=0$; or $\deg(q^\beta) = 1/3$, $g_{\tilde \beta}^{-1} = (1, \zeta_3, \zeta_3, 1)$. 

The classes $\tilde \beta \in \bigcup C_{\sA}$ satisfying $\deg(q^{r(\tilde \beta)}) = 1$ are
\[
(-1/2, -1/2, 0, 1) \quad \quad (0,0,-1/2, 1) \quad \quad (0, -1, -1, 3).
\]
When $\tilde \beta$ is equal to either of the first two, $g_{\tilde \beta}$ is not the identity. When $\tilde \beta = (0, -1, -1, 3),$ the corresponding summand of \eqref{eq:wfl-ifunc} has degree at least 1 in the $H_i$. So none of these summands of \eqref{eq:wfl-ifunc} contribute to the $z^{-1}$ part of $I(q, z)$. 

The only class $\tilde \beta \in \bigcup C_{\sA}$ satisfying $\deg(q^{r(\tilde \beta)}) = 1/3$ is $(0, -1/3, -1/3, 1)$. By direct computation we find that the coefficient of $z^{1}q^{r(0,-1/3, -1/3, 1)} = z^{-1}q_1^{-1/3}q_2^{-2/3}q_3$ in \eqref{eq:wfl-ifunc} is $\one_{1/3}$.

\end{proof}

}

\subsection{Bundle on a weighted Grassmannian}\label{sec:wbun}
We compute the $I$-function of a noncompact target: the total space of a tautological subbundle on a weighted Grassmannian. A GIT presentation for our target consists of the following data.
\begin{itemize}
\item $X:=\AA^{12}$ is viewed as the space of pairs $(M, N)$
where $M$ is a $2 \times 5$ matrix and $N$ is a $1 \times 2$ matrix over $\CC$.
\item $G:= GL(2)$, where $\Lambda \in GL(2)$ acts on $X$ by 
\[
\Lambda \cdot (M, N) = (\Lambda M \diag(1,1,1, \det(\Lambda), \det(\Lambda)), \det(\Lambda)^{-1} N \Lambda^{-1}).
\]
\item $\theta \in \chi(G)$ is the determinant character
\end{itemize}
We choose $T \subset G$ to be the diagonal matrices. The weight matrix for the $T$-action on $X$ is given by
\[
\begin{blockarray}{cccccccccccc}
m_{11} & m_{12} & m_{13} & m_{14} & m_{15} & m_{21} & m_{22} & m_{23} & m_{24} & m_{25} &  n_1 & n_2\;\;\\
\begin{block}{(cccccccccccc)}
  1 & 1 & 1 & 2 & 2 & 0 & 0 & 0 & 1 & 1 & -1 & -2\;\;\\
  0 & 0 & 0 & 1 & 1 & 1 & 1 & 1 & 2 & 2 & -2 & -1\;\;\\
\end{block}
\end{blockarray}
 \]
where $(m_{ij})$ are coordinates on an element of the $2\times 5$ matrix $M$, and $N = (n_1, n_2)$. The set $X^{ss}(G)$ is equal to the locus of pairs $(M, N)$ where $M$ has full rank, so the projection $(M, N) \to M$ realizes $X\sslash G$ as the total space of a rank-2 vector bundle on the weighted Grassmanian of Section \ref{sec:wgr}. In particular $\Stab_G(X^{ss}_\theta(G))$ is the same as for that example.

\subsubsection{$I$-function}
The sets $C_{\sA}$ are the same as those computed in Section \ref{sec:wgr-ieff} for the weighted Grassmannian.
We have that 
\[\KK^{\geq 0}(X, G, \theta) = (1/2)\ZZ_{\geq 0} \cup (1/3)\ZZ_{\geq 0} \subset \QQ^2, \quad \quad \quad and\]
\begin{align*}
I_{\tilde \beta}(z) = & C(\tilde \beta, e_1-e_2)^{-1}\; C(\tilde \beta, e_2-e_1)^{-1}\;C(\tilde \beta, e_1)^3\; C(\tilde \beta, e_2)^3\\
&  C(\tilde \beta, 2e_1+e_2)^2\; C(\tilde \beta, e_1+2e_2)^2\; C(\tilde \beta, -2e_1-2e_2)\; C(\tilde \beta, -e_2-2e_1)\; \one_{g_{\tilde \beta}^{-1}}\notag.
\end{align*}
The factors with $2e_1+e_2$ and $e_2+2e_2$ (resp. $-e_1-2e_2$ and $-2e_1-e_2$) will always appear as denominators (resp. numerators).
One checks using Lemma \ref{lem:bound-degree} that $\deg(q^\beta) \geq 2$ for $I$-effective $\beta$, so the mirror map $\mu(q, z)$ is $0$. The map $\iota$ is injective by Remark \ref{rmk:extend-pic}.

\subsubsection{Extension}
By Theorem \ref{thm:only} (with $\tilde \beta = (-1/3, -1/3)$), if we define
$\nu: \Gm \to GL(X)$ to be the 1-parameter subgroup given by the action
\[
\gamma \cdot (M, N) = (M \diag(1,1,1,\gamma, \gamma), \gamma^{-1}N),
\]
we get an extended presentation $(X \times \AA^1, G \times \Gm, \vartheta_3)$ where the action of $G \times \Gm$ on $X \times \AA^1$ is induced by
\[
(\Lambda, \gamma) \cdot (M, N, y) = (\Lambda M \diag(1,1,1, \det(\Lambda)\gamma, \det(\Lambda)\gamma), \det(\Lambda)^{-1}\gamma^{-1} N \Lambda^{-1}, \gamma)
\]
for $(\Lambda, \mu, \gamma) \in SL(2) \times \Gm \times \Gm$ and $((M, N), y) \in X \times \AA^1$.
The $I$-effective classes for $(X\times \AA^1, G\times \Gm, \vartheta_3)$ are the same as for the extended presentation in Section \ref{sec:wgr-effective}. As in that example, the $I$-function has unwieldy asymptotics so we omit an explicit formula. 
 However, $\iota$ is injective by Remark \ref{rmk:extend-pic} and Remark \ref{rmk:embed}

\subsection{Complete intersection in a weighted Grassmannian}\label{sec:delpez}
In this section we compute the full quantum period of the orbifold del Pezzo surface $X_{1, 7/3}$ studied in \cite[Sec~6.2]{OP}. We define (extended) CI-GIT data as follows:
\begin{itemize}
\item $\bX := X\times \AA^1, \bG := GL(2)\times \Gm,$ and $\vartheta_3$ are the extended GIT data in Section \ref{sec:wgr-extension} (we will also use the maximal torus $\bT = T \times \Gm$ and the basis for $\chi(\bG)$ introduced there) 
\item $E:=L^{\oplus 4}$ where $L$ is the 1-dimensional representation of $\bG$ induced by $(\Lambda, \gamma) \mapsto \det(\Lambda)^2\gamma$ for $(\Lambda,\gamma) \in GL(2)\times \Gm$
\item $s$ is defined in Lemma \ref{lem:choose-s} below.
\end{itemize}
We will denote $\bY := Z(s)$ and $Y := \bY \cap (X \times \{1\})$. It follows from Lemma \ref{lem:stack-iso} that $\bY \sslash \bG \simeq Y\sslash G$, and according to \cite{CH}, this stack is isomorphic to $X_{1, 7/3}$.

We now show that  $(\bX, \bG, \vartheta_3, E, s)$ is a GI-GIT presentation for $\bY \sslash \bG \simeq Y\sslash G$; the proof is essentially Bertini's theorem. 
If $G'$ is equal to $G$ or $\bG$, $X'$ is equal to $X$ or $\bX$, and $E'$ is any $G'$-representation, write $E'_{X'}$ for the $G$-equivariant vector bundle $E' \times X'\to X'$ and $\Gamma^{G'}(X', E'_{X'})$ for its equivariant global sections.
Recall the coordinates $((m_{ij}), y)$ on $\bX$ and let $\Delta_{ij}:= m_{1i}m_{2j}-m_{1j}m_{2i}$. Observe that $L$ has weight $(2, 2, 1)$ with respect to $\bT$. Hence $\Gamma^\bG(\bX, L_{\bX})$ is spanned by the set\note{a $\bG$ invariant section is invariant under $SL(2)$ and hence is a polynomial in these $\Delta$s. Now find the polynomials of weight $(2, 2)$.}
\[
\bS= \left\{\begin{array}{c}y\Delta_{12}^2, \;y\Delta_{13}^2,\; y\Delta_{23}^2,\; y\Delta_{12}\Delta_{13},\; y\Delta_{12}\Delta_{23},\; y\Delta_{13}\Delta_{23},\\
\Delta_{14},\; \Delta_{15},\; \Delta_{24}, \;\Delta_{25},\; \Delta_{34}, \;\Delta_{35}
\end{array}\right \}.
\]

\begin{lemma}\label{lem:choose-s}
There is a dense subset $U \subset \Gamma^{\bG}(\bX, E_\bX)$ such that if $s \in U$ then $s$ is a regular section and $Z(s) \cap \bX^{ss}(\bG)$ is smooth and connected with no $\bmu_2$-stabilizers.
\end{lemma}
\begin{proof}
Let $Z(\bS)$ denote the common vanishing locus of all elements of $\bS$. A computation shows that $Z(\bS) \cap \bX^{ss}(\bG)$ is equal to the locus where the first three columns of $M = (m_{ij})$ are zero and $\Delta_{45}(M)\neq =0$ (and $y \neq 0$). 
 This is precisely the locus of points wtih $\bmu_3$ isotropy and it has has codimension 6 in $\bX^{ss}(\bG)$.
Let $V := \Gamma^\bG(\bX, L_\bX) \subset \Gamma(\bX, \OO_\bX)$
and note that $V$ is equal to the span of $\bS$.

Consider the variety
\[
\fX = \{(\bx, s) \in \bX^{ss} \times V^{\oplus 4} \mid s(\bx) = 0\}
\]
together with its projection $p:\fX \to V^{\oplus 4}$. By Bertini's Theorem 
\cite[Thm~3.4.10]{FOV} and \cite[Tag~055A]{tag}\note{bertini gives ONE point where the fiber is irreducible; stacks project shows there must be a neighborhood of this point where the fiber is irreducible} there is an open subset $U_1 \subset V^{\oplus 4}$ consisting of sections $s$ such that the fiber $\fX_s$ is irreducible. 

The map $p$ is equal to a restriction of the composition $\bX^{ss} \times V^{\oplus 4} \to \bX\sslash\bG \times V^{\oplus 4} \to V^{\oplus 4}$ to a closed subvariety of $\bX^{ss}\times V^{\oplus 4}$; since the first map is a $\bG$-torsor and the second map is proper, the map $p$ is closed. 
Let $\Sigma \subset \fX$ denote the singular locus of $p$. Since $p$ is closed we know that the set of $s \in V^{\oplus 4}$ where $\Sigma_s \cap Z(\bS) \neq \emptyset$ is closed; in fact it has positive codimension, so that there is an open subset $U_2' \subset V^{\oplus 4}$ where $\fX_s$ is smooth along $Z(\bS)$. To show that some $s$ has $\Sigma_s \cap Z(\bS) = \emptyset$ we use
$s=(\Delta_{14}, \Delta_{24}, \Delta_{34}, \Delta_{15}) \in V^{\oplus 4}$. One can check directly that the Jacobian matrix of $s$ has full rank when $m_{14}=m_{25}=1$ and all other coordinates are zero.\note{In fact, I checked directly that the jacobian matrix has full rank when $\Delta_{45} \neq 0$.} \note{Since this point is in the unique orbit of $G$ on $Z(S) \cap X^{ss}(G)$, we conclude that $\Sigma_s \cap Z(\bS)$ is empty.}
Finally, by \cite[Thm~3.4.8]{FOV} we can in fact choose $s \in U_2'$ so that $\Sigma_s$ is empty. Since $p$ is closed we get an open subset $U_2 \subset V^{\oplus 4}$ where fibers of $p$ are smooth. \note{I think the order of these smoothness arguments is important. bertini only gives a DENSE subset where fibers are smooth, not necessarily open. we need some map to be closed in order to get that the dense subset is open. If I remove $Z(S)$ from $\fX$ I don't think $p$ will be closed any more?}

Let $D \subset X$ be the locus of points with $\bmu_2$ isotropy; it has dimension 8.\note{the locus in $M_{2x5}$ with Z/2 isotropy coming from the maximal torus has dimension 5, but then you have to add in G-orbits and since the maximal torus of G stays in the locus you already counted this only adds 2 more dimensions.} Let $Z(\bS)^c$ be the complement of $Z(\bS) \subset X^{ss}(\bG)$. By \cite[1.5.4(1)]{FOV} and \cite[Tag~05F7]{tag}, there is an open subset $U_3 \subset V^{\oplus 4}$ such that if $s \in U_3$, then 
\begin{equation}\label{eq:choose-s}
Z(s) \cap (D \cap Z(\bS)^c )
\end{equation}
has the expected dimension $8-4=4$, or it is empty.\note{effective cartier divisors can be empty I guess.} Since $Z(\bS)$ has $\bmu_3$ isotropy, we can replace $D\cap Z(\bS)^c$ with $D$ in \eqref{eq:choose-s}. Finally, $\bG = GL(2)\times \Gm$ has dimension 5 and the orbits of semistable points also have dimension 5, so for $s \in U_3$ the locus \eqref{eq:choose-s} must be empty.

We have constructed an open subset $U_1 \cap U_2 \cap U_3$ of $V^{\oplus 4}$ where fibers of $p$ are smooth irreducible and have no $\bmu_2$-stabilizers. To finish the proof apply Lemma \ref{lem:regular} (note that this requires us to check that $Z(\bS)$ has codimension at least 4 as a subset of $\bX$).  \note{details:
Let $((m_{ij}), y) \in Z(S')$. If $y \neq 0$, then $(m_{ij})$ is in $Z(S)$ and we know by the proof of Lemma \ref{lem:choose-s-1} that $((m_{ij}, y)$ is in a component of codimension at least $4$. On the other hand, if $y =0$, this gives 1 codimension. The relation $\Delta_{14}(M)=0$ implies either $m_4=0$ or $m_1=am_4$, where $m_i$ is the $i^{th}$ column of $M$. If $m_4=0$ this gives 2 codimensions, and then the relation $\Delta_{15}(M)=0$ gives at least 1 more codimension as required. If $m_1=am_4$ this gives 1 codimension, and the relations $\Delta_{24}(M)=0$ and $\Delta_{34}(M)=0$ each give one more codimension as required.}
\end{proof}

\subsubsection{$I$-effective classes}\label{sec:delpez-ieff} 
Define sets $C_{\sA} \subset \Hom(\chi(\bT), \QQ)$ as in Section \ref{sec:wgr-effective}. We will show that $\tilde \beta$ maps to an $I^{\bY, \bG, \vartheta_3}$-effective class if and only if $\tilde \beta \in C_{4, 9, 11}$. Hence the image of the $I^{\bY, \bG, \vartheta_3}$-effective classes in $\Hom(\chi(\bG), \QQ) \simeq \QQ^2$ is 
\begin{equation}\label{eq:delpez-i-effective}
\left\{
\begin{array}{c}
(\beta_1, \beta_2) \in (\frac{1}{3}\ZZ \times \ZZ)\\
\beta_2\geq 0,\;\; 3\beta_1 + 2\beta_2 \geq 0
\end{array}
\right\}.
\end{equation}
\note{$r$ may not be injective, as far as I can tell.} 
We use Lemma \ref{lem:i-effective-cigit}, which says that a class $\tilde \beta$ maps to an $I^{\bY, \bG, \vartheta_3}$-effective class if and only if $\bX^{\tilde \beta} \cap \bY^{ss}(\bG)$ is not empty. 

If $\tilde \beta \in C_{4, 9, 11}$, then $\bX^{\tilde \beta}$ contains the locus where the first three columns of $M = (m_{ij})$ are zero.\note{even if some $\tilde \beta_i$ are negative} This is precisely the base locus of the linear system $\bS$, of which $s$ is a member (see the proof of Lemma \ref{lem:choose-s}). So if $\tilde \beta \in C_{4, 9, 11}$ then $\bX^{\tilde \beta} \cap \bY^{ss}(\bG)$ is not empty.
On the other hand suppose $\tilde \beta \not \in C_{4, 9, 11}$. If $\tilde \beta \in C_{\sA}$ has integral coordinates then $\tilde \beta \in C_{4, 9, 11}$, so we may assume $\tilde \beta_1+\tilde \beta_2=1/2$. In this case $\bX^{\tilde \beta}$ is equal to a set of points with $\bmu_2$ isotropy. By Lemma \ref{lem:choose-s}, $\bY^{ss}(\bG) = \bY\cap \bX^{ss}(\bG)$ does not contain any points with $\bmu_2$ stabilizer, so $\bX^{\tilde \beta} \cap \bY^{ss}(\bG)$ is empty in this case.

\subsubsection{$I$-function}

A class $\tilde \beta \in C_{4, 9, 11}$ fails to be $I$-nonnegative if and only if
\begin{equation}\label{eq:not-convex}
2\tilde \beta_1 + 2\tilde \beta_2 + \tilde \beta_3 \in \ZZ_{<0}.
\end{equation}
In this case, the inequalities defining $C$ imply that both $\tilde \beta_1$ and $\tilde \beta_2$ are integral and strictly negative, and hence the substack $F^0_{\tilde \beta}(\bX\sslash \bT)$ of the untwisted sector is equal to the (zero-dimensional) locus with $\bmu_3$-isotropy. This locus is contained in the base locus of the linear system $\bS$ of which $s$ is a member, and hence we have $F^0_{\tilde \beta}(\bY\sslash \bT) = F^0_{\tilde \beta}(\bX\sslash \bT)$. Moreover if \eqref{eq:not-convex} holds then the number of weights $\epsilon_i$ of $E$ such that $\epsilon_i(\tilde \beta) \in \ZZ_{\geq 0}$ is zero, so Theorem \ref{thm:Ifunc-formula} gives an explicit formula for $I_{\tilde \beta}(z)$.

Let $C:= C_{4, 9, 11}$, $R := r^{\bY, \bG}_{\bX, \bG}$, and $\varphi: Y\sslash_G T \to Y\sslash G$. The roots $\rho_1, \rho_2$ of $\bG$ with respect to $\bT$ are $\pm (e_1 - e_2)$. Since $\tilde \beta(\rho_i) \in \ZZ$ for $\tilde \beta \in C$, we may make the replacement
\[
 \prod_{i=1}^2 C(\tilde \beta, \rho_i)^{-1} = (-1)^{\tilde \beta(\rho)} \frac{c_1(\cL_{\rho}) + \tilde \beta(\rho)z}{c_1(\cL_{\rho})}
\]
in our $I$-function formulas.
We obtain
\begin{equation}\label{eq:delpez-ifunc-shape}\varphi^*I(r(q), z) = \sum_{\tilde \beta \in C} q_1^{\tilde \beta_1+\tilde \beta_2}q_2^{\tilde \beta_3}I_{\tilde \beta}(z)\end{equation}
where if $\tilde \beta$ is $I$-nonnegative (i.e. $2 \tilde \beta_1 + 2\tilde \beta_2 + \tilde \beta_3 \not \in \ZZ_{<0}$), we have
\begin{align} 
\label{eq:delpez-ifunc-convex}
I_{\tilde \beta}(z)=  & \frac{1}{\tilde\beta_3!z^{\tilde\beta_3}}   (-1)^{\tilde\beta_1-\tilde\beta_2} \frac{H_1-H_2 +(\tilde\beta_1-\tilde\beta_2)z}{H_1-H_2}  \\
&C(\tilde \beta, e_1)^3\;C(\tilde \beta, e_2)^3\; C(\tilde \beta, 2e_1+e_2)^2\; C(\tilde \beta, e_1+2e_2)^2\; C(\tilde \beta, 2e_1+2e_2)^{-4} \;\one_{g_{\tilde \beta}^{-1}}. \nonumber
\end{align}
If $\tilde \beta$ is not $I$-nonnegative (i.e. $2 \tilde \beta_1 + 2\tilde \beta_2 + \tilde \beta_3 \in \ZZ_{<0}$) we have 
\begin{align} 
\label{eq:delpez-ifunc-nonconvex}
I_{\tilde \beta}(z)=  & \frac{1}{\tilde\beta_3!z^{\tilde\beta_3}}   (-1)^{\tilde\beta_1-\tilde\beta_2} \frac{H_1-H_2 +(\tilde\beta_1-\tilde\beta_2)z}{H_1-H_2}  
  \\
&C^\circ(\tilde \beta, e_1)^3\; C^\circ(\tilde \beta, e_2)^3\; C(\tilde \beta, e_1+2e_2)^2\; C(\tilde \beta, 2e_1+e_2)^2\; C^\circ(\tilde \beta, 2e_1+2e_2)^{-4}\;[B\bmu_3]_1 \nonumber
\end{align}
where $[B\bmu_3]_1$ is the fundamental class of the locus in the untwisted sector of $Y\sslash G$ with $\bmu_3$ isotropy (note that this locus is isomorphic to $B\bmu_3$). In both formulas, the factors with $e_1+2e_2$ and $e_2+2e_1$ will always appear as denominators.

Let $\one_{1/3} \in H^\bullet(\In{Y\sslash G}; \QQ)$ be the fundamental class of the sector corresponding to $(\zeta_3, \zeta_3) \in T'$. 
\note{(There is another twisted sector, corresponding to $(\zeta_3^2, \zeta_3^2) \in T'$, but we will not need notation for its fundamental class.)} 
In Lemma \ref{lem:delpez-asymptotics} below, we show that
\begin{equation}\label{eq:delpez-asymptotics}
I(r(q), z) = \one + z^{-1}(8q_1 + q_2)\one_{} + z^{-1}(q_1^{-2/3}q_2 + 3q_1^{1/3}) \one_{1/3} + O(z^{-2})
\end{equation}
It is clear from the definition of $\mu$ that its formation commutes with the restriction $r$. Hence, if we define \note{ $c_2=3$, $c_1=8$ }
\begin{equation}\label{eq:mu}
\mu(q_1, q_2) = (8q_1+q_1)\one_{} + (q_1^{-2/3}q_2 + 3q_1^{1/3}) \one_{1/3}
\end{equation}
by applying $r$ to both sides of \eqref{eq:mirror-thm} we have that\note{we have $\mu(q, z) = \mu(q) = (c_1q_1+q_1)\one_{id} + (q_1^{-2/3}q_2 + 3q_1^{1/3}) \one_{(\zeta_3, \zeta_3)}$}
\begin{equation}\label{eq:delpez-mirror}
J(\mu(q_1, q_2), r\circ\iota(Q), z) = I(r(q), z) .
\end{equation}

\begin{lemma}\label{lem:delpez-asymptotics}
The equality \eqref{eq:delpez-asymptotics} holds.
\end{lemma}
\begin{proof}
The set $C$ is additively generated over $\ZZ_{\geq 0}$ by $\frac{1}{3}(2, -1, 0), \frac{1}{3}(-1, 2, 0),$ and $\frac{1}{3}(-1, -1, 3)$. To bound $\deg(q^\delta)$ in the $I$-function of $(\bX, \bG, \vartheta_3, E, s)$ we use Lemma \ref{lem:bound-degree}: since $\pmb \xi - \pmb \epsilon = (1, 1, 1)$ we see that  $\deg(q^\delta) \geq 1/3$ for nontrivial $\delta$. This implies that $I(r(q), z) = \one + \OO(z^{-1})$.

Define $\deg(q^{r(\delta)}):= \deg(q^{\delta}) $, using Lemma \ref{lem:bound-degree} to see that this is well-defined. By the homogeneity of $I^{\bY, \bG, \vartheta_3}(q, z)$ discussed in Section \ref{sec:implications}, a term of $I(r(q), z)$ with $z^{-1}$ must also contain $q^\beta$ and $\alpha \in H^k(\In{\bY\sslash \bG}_{(g_{\tilde \beta}^{-1})}; \QQ)$ satisfying $\deg(q^\beta) + k + \age_{\bY\sslash \bG}(g_{\tilde \beta}^{-1}) =  1$.\note{Note that by Lemma \ref{lem:stack-iso}, there is an equality $\age_{Y\sslash G}(g_{\tilde \beta}^{-1}) = \age_{\bY\sslash \bG}((g_{\tilde \beta}^{-1}, 1)).$} We compute the ages of the elements of $T'\times \Gm$ that represent conjugacy classes in $\Stab_\bG(\bY^{ss}(\bG))$ (see \cite[Example~3.7]{OP}):
\[
\age_{\bY\sslash \bG}(1,1,1) = 0 \quad \quad \age_{\bY\sslash \bG}(\zeta_3, \zeta_3, 1) =  2/3\quad \quad \age_{\bY \sslash \bG}(\zeta_3^2, \zeta_3^2, 1) = 4/3.
\]
Hence, the only possibilities are $\deg(q^\beta)=1$, $g_{\tilde \beta}^{-1} = (1, 1, 1)$, and $k=0$; or $\deg(q^\beta) = 1/3$, $g_{\tilde \beta}^{-1} = (\zeta_3, \zeta_3, 1)$, and $k=0$.

The classes $\tilde\beta \in C$ mapping to $\beta$ satisfing $\deg(q^{\beta})=1$ are\note{each generator of $C$ has weight $1/3$ wrt $\xi-\epsilon=(1,1,1)$, so choose exactly three of these to add up (counted with multiplicity)}
\begin{align*}
(1,0,0) && (0,1,0) && (0,0,1) && (1,-1,1) && (-1,1,1)\\
(2,-1,0) && (-1,2,0) && (-1,-1,3) && (-1,0,2) && (0,-1,2).
\end{align*}
Of these, only the class $(-1, -1, 3)$ is not $I$-nonnegative.
We see from \eqref{eq:delpez-ifunc-nonconvex} that the summand $I_{\tilde \beta}(z)$ is equal to the cup product of some cohomology class with $[B\bmu_3]_1$ in the untwisted sector, which has dimension 2. Hence the codimension $k$ of the result is not zero, and this term does not contribute to the $z^{-1}$ term.

Of the remaining classes, one can check from \eqref{eq:delpez-ifunc-convex} that if $\tilde \beta$ has any negative coordinate, then every summand
$I_{\tilde \beta}( z)$ has degree at least 1 in the $H_i$.
This happens because if $\tilde\beta(\xi_i)$ is negative, then $I_{\tilde \beta}( z)$ has a factor of $c_1(\cL_{\xi_i})$.  \note{reason: each of these has negative inner product with at least two columns of the charge matrix. one column produces one H in the numerator. one $H$ could cancel with the abelianization factor. subtelty: $(-1,-1,3)$ produces more denominator $H's$ in the ql factor. but I think there are enough numerator $H$'s to cancel.} Hence these summands also do not contribute to the $z^{-1}$ term. 

By direct computation as Lemma \ref{lem:bs3-asymptotics}\note{or Section \ref{sec:period}}, we find that the coefficient of $z^{-1}q^{r(0,0,1)} = z^{-1}q_2$ in \eqref{eq:delpez-ifunc-convex} is $\one$ and the coefficient of $z^{-1}q^{r(1,0,0)} = z^{-1}q_1$ is $8$ times $\one$.

The classes $\tilde \beta \in C$ satisfying $\deg(q^{r(\tilde \beta)})=1/3$ are 
\[
(-1/3, -1/3, 1) \quad \quad (2/3, -1/3, 0) \quad \quad (-1/3, 2/3, 0).
\]
These classes are all $I$-nonnegative. Direct computation shows that the coefficient of $z^{-1}q^{r(-1/3, -1/3, 1)} = z^{-1}q_1^{-2/3}q_2$ is $\one_{1/3}$ and the coefficient of $z^{-1}q^{r(2/3, -1/3, 0)} = z^{-1}q_1^{1/3}$ is $3\one_{1/3}$.

\end{proof}

\subsubsection{Quantum period}\label{sec:period}
We compute the quantum period of $Y\sslash G$, following the definition in \cite{many}. 
The quantum period for $Y\sslash G$ is\note{a priori we're using $zJ(\gamma, q, z)$, but then we set $z=1$ so the leading $z$ disappears.}\note{problem: the formula for the quantum period says to use the anticanonical divisor of the coarse space. I asked Andrea. He says `` there is no difference between the canonical divisor of the orbifold and the canonical divisor of the coarse space. This is because our orbifolds (in the definition of the Fano orbifolds) are "well-formed", i.e. the stacky locus has codimension at least 2. In other words there is no stackiness on divisors.''} 
\[
G(x,t) := [J\left(\gamma(t, x), t^{-K_{Y\sslash G}}, 1\right)]_{\one_{}} \quad \quad \text{where} \quad \quad \gamma(t, x):= t^{1/3}x\one_{1/3}.
\]
Here, $[\cdot]_{\one_{}}$ denotes the coefficient of the class $\one_{}$, and by writing $t^{-K_{Y\sslash G}}$ as an argument we mean that the Novikov variable $Q^B$ is replaced by $t^{B(-K_{Y\sslash G})},$ where $-K_{Y\sslash G}$ is the anti-canonical bundle on $Y\sslash G$.\note{the anticanonical bundle of the orbifold and coarse moduli ``agree'' becaue the stacky locus has codimension 2, so divisors ``are not stacky''}

To compute $G(x, t)$, we set $q_1 = t$ and $q_2 = t(x-3)$ in \eqref{eq:delpez-mirror}. If we write $I(q_1, q_2, z) := I(r(q), z)$ then the right side of \eqref{eq:delpez-mirror} becomes $I(t, t(x-3), z)$. On the left side, we compute
\[
\mu(t, t(x-3)) = t(x+5)\one_{} + t^{1/3}x\one_{1/3} = t(x+5)\one_{} + \gamma(t, x).
\]
To compute the restriction of $r \circ \iota(Q)$, note that if $r \circ \iota(Q) = q_1^{\beta_1}q_2^{\beta_2}$ then $\beta_2=0$, since the inclusion $X\sslash G \to [\bX/\bG]$ is induced by the group homomorphism $G \to \bG = G \times \Gm$ given by the identity homomorphism in the first factor and the trivial homomorphism in the second factor. So the restriction of $r \circ \iota(Q^B)$ is equal to $t^{\beta_2} = t^{\tilde \beta_1 + \tilde \beta_2}$, where $r\circ \iota(B)$ is the image of $(\tilde \beta_1, \tilde \beta_2) \in \QQ^2 \simeq \Hom(\chi(T'), \QQ)$. On the other hand we have that $B(-K_{Y\sslash G}) = \iota(B)(-K_{Y\sslash G})$, and by \eqref{eq:thing0} and \eqref{eq:thing1} this equals $\tilde \beta_1 + \tilde \beta_2$. So the restriction of $r \circ \iota(Q)$ is precisely $t^{-K_{Y\sslash G}}$.
We have shown that setting set $q_1 = t$ and $q_2 = t(x-3)$ and $z=1$ in \eqref{eq:delpez-mirror} yields the equation
 \note{I don't have a clear way to explain this, but: a priori when we restrict Yang's thm along$q_1 = t$ and $q_2 = t(x-3)$ and $z=1$ we obtain a formula for $J(\mu_{rest}, r\circ \iota(Q)_{rest}, 1)$, where $r\circ \iota(Q)_{rest}$ means that we write $r\circ \iota(Q)_{rest}$ as $q_1^{d_1}q_2^{d_2}$ and then replace $q_1$ and $q_2$. I think an important point is that $\iota(Q)$ will always be 0 in the $q_2$ coordinate because the inclusion $X\sslash G \to [\bX/\bG]$ is the trivial homomorphism on the second factor of $G \to \bG$. So $r\circ \iota(Q)_{rest}$ is just equal to $q_1^{d_1}$, and $d_1$ happens to equal the pairing of the degree with the anticanonical character? }
\[
[J(t(x+5)\one_{} + \gamma(t, x), t^{-K_{Y\sslash G}}, 1)]_{\one_{}} = [I(t, t(x-3), 1)]_{\one_{}}
\]
and hence by the string equation
\[
G(x, t) = e^{-t(x+5)}[I(t, t(x-3), 1)]_{\one_{}}.
\]

We use \eqref{eq:delpez-ifunc-convex} and \eqref{eq:delpez-ifunc-nonconvex} to get an explicit formula for $G(x, t)$. 
From \eqref{eq:delpez-ifunc-nonconvex} we see that only $I$-nonnegative classes contribute to the coefficient of $\one$.
From the formula for $g_{\tilde \beta}$, an $I$-nonnegative class contributes to the coefficient of $\one_{}$ only if $\tilde \beta \in C$ and $\tilde \beta_1, \tilde \beta_2 \in \ZZ_{\geq 0}$ (see the proof of Lemma \ref{lem:delpez-asymptotics}).
Let
\[
A_{\tilde \beta}(q_1, q_2, z) := \frac{q_1^{\tilde \beta_1+\tilde \beta_2}q_2^{\tilde \beta_3}(-1)^{\tilde \beta_1-\tilde \beta_2}(2\tilde \beta_1+2\tilde \beta_2+\tilde \beta_3)!^4}{\tilde \beta_3!\tilde \beta_1!^3\tilde \beta_2!^3(2\tilde \beta_1+\tilde \beta_2 + \tilde \beta_3)!^2  (\tilde \beta_1 + 2\tilde \beta_2 + \tilde \beta_3)!^2z^{\tilde \beta_1 + \tilde \beta_2 + \tilde \beta_3}} 
\]
Then by \eqref{eq:delpez-ifunc-convex}, we have that the $(\tilde \beta_1, \tilde \beta_2, \tilde \beta_3)$-summand of \eqref{eq:delpez-ifunc-shape}
is equal to
\begin{align*}
A_{\tilde \beta}&\left(1 + \frac{(\tilde \beta_1 - \tilde \beta_2)z}{H_1-H_2} \right)\left(
-\sum_{k=1}^{\tilde \beta_1} \frac{3H_1}{kz}
-\sum_{k=1}^{\tilde \beta_2} \frac{3H_2}{kz}
-\sum_{k=1}^{2\tilde \beta_1+\tilde \beta_2+\tilde \beta_3} \frac{2(2H_1+H_2)}{kz}\right.\\
 &-\sum_{k=1}^{\tilde \beta_1+2\tilde \beta_2+\tilde \beta_3} \frac{2(H_1+2H_2)}{kz}
+ \left.\sum_{k=1}^{2\tilde \beta_1+ 2\tilde \beta_2 + \tilde \beta_3}\frac{4(2H_1+2H_2)}{kz}\right) + O(H_i)
\end{align*}
Define
\[ B_\ell := \sum_{k=1}^\ell \frac{1}{k}. \]
When we add 
the $(\tilde \beta_1, \tilde \beta_2, \tilde \beta_3)$ and $(\tilde \beta_2, \tilde \beta_1, \tilde \beta_3)$ summands of \eqref{eq:delpez-ifunc-shape} 
and set $q_1=t$, $q_2=t(x-3)$ and $z=1$, we get 
\[
A_{\tilde \beta}(t, t(x-3), 1)\left(2 + (\tilde \beta_1-\tilde \beta_2)(-3B_{\tilde \beta_1}+3B_{\tilde \beta_2} - 2B_{2\tilde \beta_1 + \tilde \beta_2 + \tilde \beta_3} + 2B_{\tilde \beta_1 + 2\tilde \beta_2 + \tilde \beta_3} )\right) + O(H_i).
\]
Therefore, $G(x, t)$ is equal to
\[
 e^{-t(x+5)}\sum_{\tilde \beta_i \in \ZZ_{\geq 0}} A_{\tilde \beta}(t, t(x-3), 1)\left(1 + \frac{\tilde \beta_1-\tilde \beta_2}{2}(-3B_{\tilde \beta_1}+3B_{\tilde \beta_2} - 2B_{2\tilde \beta_1 + \tilde \beta_2 + \tilde \beta_3} + 2B_{\tilde \beta_1 + 2\tilde \beta_2 + \tilde \beta_3} )\right).
\]
Observe that if we set $x=3$, we recover the specialization at the end of \cite[Section~6.2]{OP}.

The regularization of $G(x, t)$ is conjecturally equal to the classical period of the Laurent polynomial
\[
ay + \frac{x}{y^2}(1+y)^3 + \frac{1}{xy^2}(1+y)^4 + \frac{7}{y} + \frac{2}{y^2}\]
(see \cite[Conj. B]{many} and \cite{OP}). As computed by the Fanosearch code package written by Coates and Kasprzyk \cite{andrea-email}, the first four terms of said regularization are
\[
1 + (14x + 70)t^2 + (6x^2 + 210x + 966)t^3 + (546x^2 + 6888x + 22470)t^4 + \dots  ,
\]
and one can check that these terms agree with the first four terms of the classical period.

\appendix

\section{Comparing $J$ and $I$ functions}\label{sec:compare}
\subsection{A proof of \eqref{eq:mirror-thm}}\label{sec:pf-of-mirror}
In this section we prove that \eqref{eq:mirror-thm} holds with the definitions of $J^{\cX}$ and $I^{X, G, \theta}$ given in this paper. Let $\br$ be the locally constant function on $\In{\cX}$ that is equal to the order of the generic stabilizer on that component.

Let $\widetilde{J}^{\cX}$, $\widetilde{I}^{X, G, \theta}$, and $\widetilde{\mu}$ denote the analogous objects in \cite{yang}. By definition, we have
\begin{align*}
\widetilde{J}^{\cX}(\bt(z), Q, z) = \one &+ \frac{\bt(-z)}{z} \\
&+ \sum_{\substack{n \geq 0\\B \in \Eff(\cX)}}\frac{Q^B}{n!}  \inv^* \br^2 ev_{1*}\left[\left( \frac{1}{z(z-\psi_1)} \cup 
\prod_{i=2}^{n+1} ev_i^*(\bt(\psi_i))
\right) \cap [\overline{\cM}]^\vir\right].
\end{align*}
On the other hand, by \cite[Rmk~3.4.5]{Webb21} we have
\[
\br I^{X, G, \theta}(q, z) = \widetilde{I}^{X, G, \theta}(q, z) \quad \quad \quad \br \mu(q, z) = \widetilde{\mu}(q, z).
\]
By \cite[Thm~1.12.2]{yang} we have
\[
\widetilde{J}^{\cX}(\widetilde{\mu}(q, -z), \iota(Q), z) = \widetilde{I}^{X, G, \theta}(q, z).
\]
Multiplying both sides by $\br^{-1}$ and making the substitution $ \widetilde{\mu}(q, z) = \br \mu(q, z) $ we obtain
\[
\br^{-1} \widetilde{J}^{\cX}(\br{\mu}(q, -z), \iota(Q), z) = {I}^{X, G, \theta}(q, z).
\]
It is straightforward to check that $\br^{-1} \widetilde{J}^\cX(\br\bt(z), Q, z) = J^{\cX}(\bt(z), Q, z)$.

\subsection{Gromov-Witten invariants of noncompact GIT quotients}\label{sec:noncompact}
Throughout this section we work with an algebraic torus $\TT$, letting $R_\TT = H^\bullet_{\TT}(\Spec(\CC); \QQ)$ and letting $S_\TT$ be the field of fractions of $R_{\TT}$. 
\subsubsection{Integration on noncompact stacks via virtual localization}

The following technique is used in \cite{Liu} to define Gromov-Witten invariants of noncompact targets. Let $\cX$ be a Deligne-Mumford stack equipped with a $\TT$-action and a $\TT$-equivariant perfect obstruction theory, such that the fixed locus $\cX^\TT$ is proper and the virtual localization formula of \cite{GP99} holds.\footnote{By \cite[Thm~3.5]{CKL17}, the virtual localization formula holds if the virtual normal bundle to $\cY^{\TT}$ has a global resolution by a 2-term complex of locally free sheaves.} In this situation, for $\gamma \in H^\bullet_{\TT}(\cX; \QQ)$ one defines
\begin{equation}\label{eq:integration}
\int_{\cX}^\TT \gamma := \sum_j p_{j*} \frac{i_j^* \gamma \cap [F_j]^{\vir}}{e(N_{F_j})} \in S_{\TT},
\end{equation}
where $i_j: F_j \to \cX$ are the inclusions of the connected components of the fixed locus, $p_j$ is the projection $F_j \to \Spec(\CC)$, and all classes and morphisms on the right hand side of \eqref{eq:integration} are to be interpreted $\TT$-equivariantly. In \eqref{eq:integration} the pushforward happens in (equivariant) Borel-Moore homology, and then we apply Poincar\'e duality to get an element of $S_\TT$.\note{I'm being sloppy here because I don't know about equivariant BM homology. There is a paper by Brion, "Poincare duality and equivariant (co)homology." The big results in that paper only apply to compact spaces. But he says you can define equivariant BM homology as the corresponding theory of the classifying space. I think that lets you make sense of an equivariant cycle map, so $[F]^{\vir}$ is in $H^{\BM, \TT}_\bullet$, and pushforwards. Finally we need equivariant Poincar\'e duality for $\Spec(\CC)$. I'm just guessing that this must be true.}


\begin{example}
If $\cX$ is smooth (and not necessarily compact), then it has a natural $\TT$-equivariant perfect obstruction theory whose induced virtual class is Poincar\'e dual to 1. The components of the fixed locus $F_j$ are also smooth, and \eqref{eq:integration} becomes 
\[
\int_{\cX}^\TT \gamma := \sum_j p_{j*} \frac{i_j^* \gamma}{e(N_{F_j})} \in S_{\TT}.
\]
\end{example}

We will use the following lemma later.

\begin{lemma}\label{lem:push-twice}
Let $\cX$ and $\cY$ be Deligne-Mumford stacks equipped with $\TT$-actions such that $\cX^\TT$ and $\cY^\TT$ are proper. Assume that $\cY$ is smooth and that $\cX$ has a $\TT$-equivariant perfect obstruction theory such that virtual localization holds. If $f: \cX \to \cY$ is a proper morphism, then
\[
\int_{\cX}^\TT \gamma = \int_\cY^\TT f_*(\gamma \cap [X]^{\vir})
\]
where $f_*$ and $[X]^{\vir}$ are defined equivariantly.
\end{lemma}
\begin{proof}

Let $F_j$ (resp. $G_\ell$) be the components of $\cX^\TT$ (resp. $\cY^{\TT}$). For each $j$ there exists an $\ell(j)$ and a commuting diagram
\[
\begin{tikzcd}
F_j \arrow[d, "i_j"] \arrow[r, "f_j"] & G_{\ell(j)} \arrow[d, "i_{\ell(j)}"]\\
\cX \arrow[r, "f"] & \cY
\end{tikzcd}
\]
By the virtual localization theorem for $\cX$ and the projection formula \cite[IX.3.7]{iversen}, we have an equality
\[
f_*(\gamma \cap [X]^{\vir}) = f_*\left[\sum_j i_{j*} \left( \frac{i_j^*\gamma \cap [F_j]^{\vir}}{e(N_{F_j})}\right)\right]
\]
in $H^{\BM, \TT}_\bullet(\cY; \QQ)\otimes_{R_\TT} S_{\TT}$. Since $\cY$ is smooth, by Poincar\'e duality we may view this as an equality in $H^\bullet_{\TT}(\cY, \QQ)\otimes_{R_\TT} S_{\TT}.$
Hence we can apply the localization formula on $\cY$ to the left hand side. After commuting the pushforwards on the right hand side, we obtain
\[
\sum_\ell i_{\ell*}\left( \frac{i_\ell^*(f_*(\gamma \cap [X]^\vir))\cap[G_\ell]^\vir}{e(N_{G_\ell})}\right) = \sum_j i_{\ell(j)*}f_*  \left( \frac{i_j^*\gamma \cap [F_j]^{\vir}}{e(N_{F_j})}\right).
\]
Now pushing forward both sides to $\Spec(\CC)$ yields the result.
\end{proof}

\subsubsection{Gromov-Witten theory of noncompact targets via localization}\label{sec:app2}
Recall that $\TT$ is an algebraic torus.
Let $\cX$ be a smooth Deligne-Mumford stack (not necessarily arising from an affine GIT quotient) equipped with a $\TT$-action such that the fixed locus $\cX^{\TT}$ is proper. Following \cite[Sec~2]{CCIT15}, we define a $\TT$-equivariant $J$-function in this situation via localization. Note that the J-function defined here differs from the one in \cite[Sec~2]{CCIT15} by a sign and a factor of $z$.

We define a pairing on $H^\bullet_{CR, \TT}(\cX; \QQ)$ by the rule
\[
(\alpha, \beta) = \int^\TT_\cX \alpha \cup \inv^*\beta.
\]
 The equivariant Novikov ring $\Lambda^\TT_\cX$ is the completion of $S_\TT [\Eff(\cX)]$ with respect to the same valuation as before.
For $\gamma_i \in H^\bullet_{\CR, \TT}(\cX; \QQ)$ we define
\begin{equation}\label{eq:def-corr}
\langle \gamma_1\psi^{k_1}, \ldots, \gamma_n\psi^{k_n} \rangle^{\cX, \TT}_{0, n, B} = \int^\TT_{\overline{\cM}} ev_1^*(\gamma_1)\psi^{k_1} \cup \ldots \cup ev_n^*(\gamma_n)\psi^{k_n} 
\end{equation}
where $\overline{\cM} := \overline{\cM}_{0, n}(\cX, B)$ is the moduli space of stable maps to $\cX$, equipped with the \textit{weighted} $\TT$-equivariant virtual cycle.

Let $\{\phi_\alpha\}$ be a basis for $H^\bullet_{\CR, \TT}(\cX; \QQ)$ and let $\{\phi^\alpha\}$ be a dual basis with respect to the pairing given above. Now for any set of formal variables $\{t_i\}_{i=1}^N$ and for any $\bt(z) \in H^\bullet_{CR, \TT}(\cX; \QQ) \otimes_{R_\TT} \Lambda_\cX^\TT\llbracket t_i, z\rrbracket$ we define
\[
J^{\cX}_{\loc}(\bt(z), Q, z) = 1 + \frac{\bt(-z)}{z} + \sum_{\substack{n \geq 0\\B \in \Eff(\cX)}}\sum_{\alpha} \frac{Q^B}{n!}\phi^\alpha \left\langle \frac{\phi_{\alpha}}{z(z-\psi_1)}, \bt(\psi_2), \ldots, \bt(\psi_{n+1}) \right\rangle^{\cX, \TT}_{0, n+1, B}
\]
where the the correlator in $J^{\cX}_{\loc}$ is defined as in \cite[Sec~2.2]{CCIT15}---it is a sum of the invariants in \eqref{eq:def-corr}. Observe that $J^{\cX}_{\loc}(\bt(z), Q, z)$ is an element of $H^\bullet_{\CR, \TT}(\cX; \QQ)\otimes_{R_\TT} \Lambda_\cX^\TT\llbracket t_i\rrbracket\{z, z^{-1}\}.$

\subsubsection{Comparison of the two approaches}
Let $X$, $G$, and $\theta$ be as in Section \ref{sec:jfuncs} and let $\cX := X\sslash G$. Suppose moreover that $\TT$ acts on $\cX$ with a compact fixed locus. 

\begin{proposition}
For any $\bt(z)$, the non-equivariant limit of $J^{\cX}_{\loc}(\bt(z), Q, B)$ exists and is equal to $J^{\cX}(\bt(z), Q, B)$.
\end{proposition}
\begin{proof}
We first observe that the definition of $J^{\cX}$ can be extended to the $\TT$-equivariant setting by using $\bt(z)$ and $\Lambda^\TT_\cX$ as defined in Section \ref{sec:app2}, and by replacing the pullbacks and pushforwards in \eqref{eq:jfunc-def} with their equivariant counterparts. We notate these series by $J^{\cX, \TT}(\bt(z), Q, z)$. We claim
\begin{equation}\label{eq:j-equals}
J^{\cX, \TT}(\bt(z), Q, z) = J^{\cX}_{\loc}(\bt(z), Q, z).
\end{equation}
This can be checked by computing the pairing of each side with $\phi_\alpha$. To compute the pairing on the left hand side we apply the projection formula \cite[IX.3.7]{iversen} for $ev_{1*}$ and use Lemma \ref{lem:push-twice} for the proper morphism $ev_{1*}$. 
Since the nonequivariant limit of the left hand side is equal to $J^{\cX}(\bt(z), Q, z)$ by construction, the proposition follows. 
\end{proof}



\printbibliography
\end{document}